\date{} 
\title{Liouville quantum gravity and the Brownian map I:\\ The $\QLE(8/3,0)$ metric}
\author{Jason Miller and Scott Sheffield}
\documentclass[12pt,naturalnames]{article}
\usepackage{amsmath}
\usepackage{amssymb}
\usepackage{amsthm}
\usepackage{amsfonts}
\usepackage{graphicx}
\usepackage{xcolor}
\usepackage{tabularx}
\usepackage{subfig}
\usepackage{enumerate}
\usepackage{microtype}
\usepackage{comment}
\usepackage{mathrsfs}

\def\@rst #1 #2other{#1}
\newcommand\MR[1]{\relax\ifhmode\unskip\spacefactor3000 \space\fi
  \MRhref{\expandafter\@rst #1 other}{#1}}
\newcommand{\MRhref}[2]{\href{http://www.ams.org/mathscinet-getitem?mr=#1}{MR#2}}

\usepackage[margin=1.18in]{geometry}

\newcommand{\CH}{{\mathcal H}}
\newcommand{\CU}{{\mathcal U}}
\newcommand{\one}{{\bf 1}}
\newcommand{\CX}{{\mathcal X}}
\newcommand{\CY}{{\mathcal Y}}

\newcommand{\CN}{\mathcal {N}}
\newcommand{\CW}{{\mathcal W}}
\newcommand{\CF}{{\mathcal F}}
\newcommand{\CR}{{\mathcal R}}
\newcommand{\CS}{{\mathcal S}}

\newcommand{\fb}[2]{B^\bullet(#1,#2)}

\allowdisplaybreaks

\newif\ifhyper\IfFileExists{hyperref.sty}{\hypertrue}{\hyperfalse}

\ifhyper\usepackage{hyperref}\fi

\newif\ifdraft
\drafttrue

\numberwithin{equation}{section}
\numberwithin{figure}{section}

\newtheorem{theorem}{Theorem}
\numberwithin{theorem}{section}

\newtheorem{lemma}[theorem]{Lemma}

\newtheorem{proposition}[theorem]{Proposition}

\theoremstyle{remark}\newtheorem{definition}[theorem]{Definition}
\theoremstyle{remark}\newtheorem{remark}[theorem]{Remark}

\newcommand{\R}{\mathbf{R}}
\newcommand{\Q}{\mathbf{Q}}
\newcommand{\CZ}{{\mathcal Z}}
\newcommand{\C}{\mathbf{C}}
\newcommand{\D}{\mathbf{D}}
\newcommand{\Z}{\mathbf{Z}}

\newcommand{\N}{\mathbf{N}}

\newcommand{\HH}{\mathbf{H}}
\newcommand{\h}{\HH}

\definecolor{purple}{rgb}{0.7,0,0.7}
\definecolor{gray}{rgb}{0.6,0.6,0.6}
\definecolor{dgreen}{rgb}{0.0,0.4,0.0}
\definecolor{dblue}{rgb}{0.0,0.0,0.5}

\newcommand{\X}{{\mathbf X}}
\newcommand{\re}{{\mathrm {Re}}}

\newcommand{\ol}{\overline}

\newcommand{\wh}{\widehat}
\newcommand{\wt}{\widetilde}

\newcommand{\CQ}{{\mathcal Q}}

\newcommand{\CB}{{\mathcal B}}
\newcommand{\CD}{{\mathcal D}}

\newcommand{\s}{{\mathbf S}}
\newcommand{\giv}{\,|\,}
\newcommand{\bes}{\mathrm{BES}}

\def\diam{\mathop{\mathrm{diam}}}

\def\dist{\mathop{\mathrm{dist}}}

\newcommand{\SLE}{{\rm SLE}}
\newcommand{\QLE}{{\rm QLE}}
\newcommand{\CLE}{{\rm CLE}}

\def\Ito/{It\^o}

\def \P {{\bf P}}

\def \p {{\P}}
\def \E {{\bf E}}

\newcommand{\ex}[1]{\E\!\left[#1\right]}

\newcommand{\qdist}{d_\CQ}
\newcommand{\oqdist}{\ol{d}_\CQ}

\newcommand{\SM}{\mathsf M}
\newcommand{\SN}{\mathsf N}

\newcommand{\bsphere}{\SM_{\rm BES}}
\newcommand{\lexcursion}{\SN}

\newcommand{\MstwoW}{\SM_{\mathrm{SPH,W}}^2}
\newcommand{\MstwoD}{\SM_{\mathrm{SPH,D}}^2}

\newcommand{\Mstwo}{\SM_{\mathrm{SPH}}^2}

\newcommand{\mustwo}{\mu_{\mathrm{SPH}}^2}

\newcommand{\musk}{\mu_{\mathrm{SPH}}^k}
\newcommand{\musa}{{\mu_{\mathrm{SPH}}^{A=1}}}

\newcommand{\mudonel}{\mu_{ \mathrm{DISK}}^{1, L}}
\newcommand{\mudl}{\mu_{\mathrm{DISK}}^L}

\newcommand{\cadlag}{c\`adl\`ag}

\newcommand{\qlegrowth}{\Gamma}
\newcommand{\qlegrowthleft}{\qlegrowth}
\newcommand{\qlegrowthright}{\ol{\qlegrowth}}
\newcommand{\stoppingleft}{\tau}
\newcommand{\righthit}{{\ol{\stoppingleft}}}
\newcommand{\lefthit}{\sigma}
\newcommand{\stoppingright}{{\ol{\lefthit}}}
\newcommand{\lawonall}{\Theta}

\newcommand{\strip}{{\mathscr{S}}}

\newcommand{\cyl}{\mathscr{C}}

\newcommand{\ttime}{{\mathfrak t}}

\begin{document} \maketitle

\begin{abstract}
Liouville quantum gravity (LQG) and the Brownian map (TBM) are two distinct models of measure-endowed random surfaces. LQG is defined in terms of a real parameter $\gamma$, and it has long been believed that when $\gamma = \sqrt{8/3}$, the LQG sphere should be equivalent (in some sense) to TBM. However, the LQG sphere comes equipped with a conformal structure, and TBM comes equipped with a metric space structure, and endowing either one with the other's structure has been an open problem for some time.

This paper is the first in a three-part series that unifies LQG and TBM by endowing each object with the other's structure and showing that the resulting laws agree. The present work considers a growth process called {\em quantum Loewner evolution} (QLE) on a $\sqrt{8/3}$-LQG surface $\mathcal S$ and defines $\qdist(x,y)$ to be the amount of time it takes QLE to grow from $x \in \mathcal S$ to $y \in \mathcal S$. We show that $\qdist(x,y)$ is a.s.\ determined by the triple $(\mathcal S,x,y)$ (which is far from clear from the definition of QLE) and that $\qdist$ a.s.\ satisfies symmetry (i.e., $\qdist(x,y) = \qdist(y,x)$) for a.a.\ $(x,y)$ pairs and the triangle inequality for a.a.\ triples. This implies that $\qdist$ is a.s.\ a metric on any countable sequence sampled i.i.d.\ from the area measure on $\mathcal S$.  We establish several facts about the law of this metric, which are in agreement with similar facts known for TBM. The subsequent papers will show that this metric a.s.\ extends uniquely and continuously to the entire $\sqrt{8/3}$-LQG surface and that the resulting measure-endowed metric space is TBM.
\end{abstract}

\newpage
\tableofcontents
\newpage

\parindent 0 pt
\setlength{\parskip}{0.25cm plus1mm minus1mm}

\medbreak {\noindent\bf Acknowledgments.}  We have benefited from conversations about this work with many people, a partial list of whom includes Omer Angel, Itai Benjamini, Nicolas Curien, Hugo Duminil-Copin, Amir Dembo, Bertrand Duplantier, Ewain Gwynne, Nina Holden, Jean-Fran{\c{c}}ois Le Gall, Gregory Miermont, R\'emi Rhodes, Steffen Rohde, Oded Schramm, Stanislav Smirnov, Xin Sun, Vincent Vargas, Menglu Wang, Samuel Watson, Wendelin Werner, David Wilson, and Hao Wu.

We thank an anonymous referee for providing a number of helpful comments which have led to many improvements in the exposition.

We would also like to thank the Isaac Newton Institute (INI) for Mathematical Sciences, Cambridge, for its support and hospitality during the program on Random Geometry where part of this work was completed.  J.M.'s work was also partially supported by DMS-1204894 and J.M.\ thanks Institut Henri Poincar\'e for support as a holder of the Poincar\'e chair, during which part of this work was completed.  S.S.'s work was also partially supported by DMS-1209044, DMS-1712862, a fellowship from the Simons Foundation, and EPSRC grants {EP/L018896/1} and {EP/I03372X/1}.

\section{Introduction}
\label{sec::intro}

\subsection{Overview}
\label{subsec::overview}

Brownian motion is in many ways the ``canonical'' or ``most natural''
probability measure on the space of continuous paths. It is uniquely
characterized by special properties (independence of increments,
stationarity, etc.) and it arises as a scaling limit of many kinds of
discrete random walks.

It is natural to wonder whether there is a similarly ``canonical'' or
``most natural'' probability measure on the space of two dimensional
surfaces that are topologically equivalent to the sphere. In fact,
over the past few decades, {\em two} mathematical objects have emerged
which have seemed to be equally valid candidates for the title of
``canonical random sphere-homeomorphic surface.''

The first candidate is the spherical form of {\em Liouville quantum
gravity} (LQG) \cite{DS08, SHE_WELD, dms2014mating, quantum_spheres,lqg_sphere,twoperspectives}. Roughly speaking, the LQG sphere is a random surface
obtained by exponentiating a form of a conformally invariant random
distribution called the Gaussian free field (GFF). LQG has its roots in the
physics literature, particularly the work of Polyakov \cite{pol81bosonic,pol81fermionic,pol88qg} and others in
string theory and conformal field theory in the 1980's. The definition
of LQG involves a parameter $\gamma$ that can be tuned to make
surfaces  more or less ``rough.'' The particular value $\gamma =
\sqrt{8/3}$ has long been understood to be special, and is said to
correspond to {\em pure quantum gravity}. Roughly speaking, this means
that $\sqrt{8/3}$-LQG (a.k.a.\ {\em pure LQG}) should arise as a scaling
limit of discretized random surfaces that are {\em not} decorated by
extra statistical physical structures (a.k.a.\ {\em matter fields}).

The second candidate is an object called {\em the Brownian map} (TBM)
\cite{mm2006bm, le2008scaling, legalluniqueanduniversal, miermontlimit, legalluniversallimit} which has its roots in the mathematical analysis of discretized random surfaces (a.k.a.\ {\em random planar maps}), beginning with the work of Tutte \cite{tutte63cenus} in the 1960's. Around 2000 it was noticed by Chassaing and Schaeffer \cite{cg2004super} (building on the bijections \cite{cori-vauquelin,schaeffer}) that the large scale behavior of the profile of distances from a random point on a random planar map could be encoded using the Brownian snake of Le Gall (see \cite{legall1999spatial} for a general review).  This perspective ultimately led to the definition of the Brownian map as a random metric measure space constructed from the Brownian snake, and to the proofs of Le Gall \cite{legalluniversallimit} and Miermont \cite{miermontlimit} that the Brownian map arises as the Gromov-Hausdorff limit of uniformly random planar maps.

Both LQG and TBM have been thoroughly studied and appear in hundreds
of papers in physics and mathematics. However, until now, there has
not been a direct link between these two objects, and the
corresponding literatures have been relatively disconnected.

It has long been believed that TBM and the $\sqrt{8/3}$-LQG sphere should be equivalent in some sense. Both objects are random measure-endowed, sphere-homeomorphic surfaces. The problem is that the LQG sphere comes equipped with a conformal structure and TBM comes equipped with a metric space structure, and it is far from obvious how to endow either of these objects with the {\em other's} structure. Until one does this, it is not clear how an equivalence statement can even be formulated.

This paper is the first in a three part series (also including \cite{qle_continuity,qle_determined}) that will establish the equivalence of $\sqrt{8/3}$-LQG and TBM and provide a robust unification of the corresponding theories.  Over the course of these three papers, we will show the following:
\begin{enumerate}
\item An instance of the $\sqrt{8/3}$-LQG sphere a.s.\ comes with a {\bf canonical metric space structure}, and the resulting measure-endowed metric space agrees in law with TBM.
\item Given an instance of TBM, the $\sqrt{8/3}$-LQG sphere that generates it is a.s.\ uniquely determined. This implies that an instance of TBM a.s.\ comes with a canonical (up to  M\"obius transformation) embedding in the Euclidean sphere. In other words:
\item An instance of TBM a.s.\ comes with a {\bf canonical conformal structure} and the resulting measure-endowed conformal sphere agrees in law with the $\sqrt{8/3}$-LQG sphere.
\item The canonical (up to M\"obius transformation) embedding of TBM (with its intrinsic metric) into the Euclidean sphere (with the Euclidean metric) is a.s.\ H\"older continuous with H\"older continuous inverse. This in particular implies that a.s.\ all geodesics in TBM are sent to H\"older continuous curves in the Euclidean sphere.
\end{enumerate}
In \cite{qle_continuity,qle_determined}, we will also extend these results to infinite volume surfaces (the so-called Brownian plane \cite{cl2012brownianplane} and the $\sqrt{8/3}$-LQG quantum cone \cite{SHE_WELD,dms2014mating}) and to surfaces with boundary (the Brownian disk and its LQG analog).

The main technical achievements of the current paper concern a growth process called quantum Loewner evolution (QLE), introduced in \cite{ms2013qle}. We show that QLE growth on a $\sqrt{8/3}$-LQG surface is a.s.\ {\em determined} by the starting point of the growth process and the surface it is growing on. (It is still an open question whether this remains true when $\gamma \not = \sqrt{8/3}$.) Moreover, given two typical points $x$ and $y$ on the surface, the amount of time it takes QLE to grow from $x$ to $y$ a.s.\ equals the amount of time it takes QLE to grow from $y$ to $x$. This time can be interpreted as a distance between $x$ and $y$ (this paper establishes the appropriate triangle inequality) and this is what ultimately allows us to endow $\sqrt{8/3}$-LQG surfaces with a metric space structure. However, we stress that the results in the current paper can also be appreciated as stand alone conclusions about QLE.

\subsection{Main result}

Before we begin to explain how these results will be proved, let us describe one reason one might expect them to be true.  Both TBM and the $\sqrt{8/3}$-LQG sphere are known to be $n \to \infty$ scaling limits of the uniformly random planar map with $n$ edges, albeit w.r.t.\ different topologies.\footnote{TBM is the scaling limit w.r.t.\ the Gromov-Hausdorff topology on metric spaces \cite{le2008scaling, legalluniqueanduniversal, miermontlimit, legalluniversallimit}. The $\sqrt{8/3}$-LQG sphere (decorated by $\CLE_6$) is the scaling limit of the uniformly random planar map (decorated by critical percolation) w.r.t.\ the so-called peanosphere topology, as well as a stronger topology that encodes loop lengths and intersection patterns (see \cite{2011arXiv1108.2241S, dms2014mating}, the forthcoming works \cite{finitevolumeestimates,finitevolumelimit, strongertopology}, and the brief outline in \cite{map_making}).}  One should therefore be able to use a compactness argument to show that the uniformly random planar map has a scaling limit (at least subsequentially) in the product of these topologies, and that this scaling limit is a {\em coupling} of TBM and the $\sqrt{8/3}$-LQG sphere. It is not obvious that, in this coupling, the instance of TBM and the instance of the $\sqrt{8/3}$-LQG sphere a.s.\ uniquely determine one another. But it seems reasonable to guess that this would be the case. And it is at least conceivable that one could prove this through a sophisticated analysis of the planar maps themselves (e.g., by showing that pairs of random planar maps are highly likely to be close in one topology if and only if they are close in the other topology).

Another reason to guess that an LQG sphere should have a canonical metric structure, and that TBM should have a canonical conformal structure, is that it is rather easy to formulate reasonable sounding {\em conjectures} about how a metric on an LQG sphere might be obtained as limit of {\em approximate} metrics, or how a conformal structure on TBM might be obtained as a limit of {\em approximate} conformal structures. For example, the peanosphere construction of \cite{dms2014mating} gives a space-filling curve on the LQG sphere; one might divide space into regions traversed by length-$\delta$ increments of time, declare two such regions adjacent if they intersect, and conjecture that the corresponding graph distance (suitably rescaled) converges to a continuum distance as $\delta \to 0$. Similarly, an instance of TBM comes with a natural space-filling curve; one can use this to define a graph structure as above, embed the graph in the Euclidean sphere using circle packing (or some other method thought to respect conformal structure), and conjecture that as $\delta \to 0$ these embeddings converge to a canonical  (up to M\"obius transformation) embedding of TBM in the Euclidean sphere. In both of these cases, the approximating graph can be constructed in a simple way (in terms of Brownian motion or the Brownian snake) and could in principle be studied directly.

The current series of papers will approach the problem from a completely different direction, which we believe to be easier and arguably more enlightening than the approaches suggested above. Instead of using approximations of the sort described above, we will use a combination of the quantum Loewner evolution (QLE) ideas introduced in \cite{ms2013qle}, TBM analysis that appears in \cite{map_making}, and the $\sqrt{8/3}$-LQG sphere analysis that appears in \cite{quantum_spheres}. There are approximations involved in defining the relevant form of QLE, but they seem to respect the natural symmetries of the problem in a way that the approximation schemes discussed above do not. In particular, our approach will allow us to take full advantage of an exact relationship between the LQG disks ``cut out'' by $\SLE_6$ and those cut out by a metric exploration.

In order to explain our approach, let us introduce some notation. If $(S,d)$ is a metric space, like TBM, and $x \in S$ then we let $B(x,r)$ denote the radius $r$ ball centered at $x$. If the space is homeomorphic to $\s^2$ and comes with a distinguished ``target'' point $y$, then we let $\fb{x}{r}$ denote the {\em filled} metric ball of radius $r$, i.e., the set of all points that are disconnected from $y$ by $\ol{B(x,r)}$. Note that if $0 \leq r < d(x,y)$, then the complement of $\fb{x}{r}$ contains $y$ and is homeomorphic to the unit disk $\D$.

The starting point of our approach is to let $x$ and $y$ be points on an LQG-sphere and to define a certain ``growth process'' growing from $x$ to $y$. We assume that $x$ and $y$ are ``quantum typical,'' i.e., that given the LQG-sphere itself, the points $x$ and $y$ are independent samples from the LQG measure on that sphere. The growth process is an increasing family of closed subsets of the LQG-sphere, indexed by a time parameter $t$, which we denote by $\qlegrowth_t = \qlegrowth_t^{x \to y}$. Ultimately, the set $\qlegrowth_t$ will represent the filled metric ball $\fb{x}{t}$ corresponding to an appropriately defined metric on the LQG-sphere (when $y$ is taken to be the distinguished target point). However, we will get to this correspondence somewhat indirectly.  Namely, we will {\em first} define $\qlegrowth_t= \qlegrowth_t^{x \to y}$ as a random growth process (for quantum typical points $x$ and $y$) and only show {\em a posteriori} that there is a metric for which the $\qlegrowth_t$ thus defined are a.s.\ the filled metric balls.

As presented in \cite{ms2013qle}, the idea behind this growth process (whose discrete analog we briefly review and motivate in Section~\ref{sec::rpmintro}) is that one should be able to ``reshuffle'' the $\SLE_6$ decorated quantum sphere in a particular way in order to obtain a growth process on a LQG surface that hits points in order of their distance from the origin. This process is a variant of the $\QLE(8/3,0)$ process originally constructed in \cite{ms2013qle} by starting with an $\SLE_6$ process and then ``resampling'' the tip location at small capacity time increments to obtain a type of first passage percolation on top of a $\sqrt{8/3}$-LQG surface.  The form of $\QLE(8/3,0)$ that we use here differs from that given in \cite{ms2013qle} in that we will resample the tip at ``quantum natural time'' increments as defined in \cite{dms2014mating} (i.e., time steps which are intrinsic to the surface rather than to its specific choice of embedding).  We expect that these two constructions are in fact equivalent, but we will not establish that fact here.

As discussed in \cite{ms2013qle}, the growth process $\QLE(8/3,0)$ can in some sense be understood as a continuum analog of the Eden model. The idea explained there is that in some small-increment limiting sense, the (random) Eden model growth should correspond to (deterministic) metric growth. In fact, a version of this statement for random planar maps has recently been verified in \cite{cl2014peeling}, which shows that on a random planar map, the random metric associated with an Eden model (or first passage percolation) instance closely approximates graph distance.

Once we have defined the growth process for quantum typical points $x$ and $y$, we will define the quantity $\qdist(x,y)$ to be the amount of time it takes for a $\QLE$ growth process to evolve from $x$ to $y$. This $\qdist$ is a good candidate to be a distance function, at least for those $x$ and $y$ for which it is defined. However, while our initial construction of $\QLE$ will produce the joint law of the doubly marked LQG surface and the growth process~$\qlegrowth$, it will not be obvious from this construction that $\qdist(x,y) = \qdist(y,x)$ a.s. In fact, it will not even be obvious that the growth process is a.s.\ {\em determined} by the LQG sphere and the $(x,y)$ pair, so we will {\em a priori} have to treat $\qdist(x,y)$ as a random variable whose value might depend on more than the LQG-surface and the $(x,y)$ pair.

The bulk of the current paper is dedicated to showing that if one first samples a $\sqrt{8/3}$-LQG sphere, and then samples $x_1, x_2, \ldots$ as i.i.d.\ samples from its LQG measure, and then samples conditionally independent growth processes from each $x_i$ to each $x_j$, then it is a.s.\  the case that the $\qdist$ defined from these growth processes is a metric, and that this metric is determined by the LQG sphere and the points, as stated in Theorem~\ref{thm::qle_metric} below.

Both the $\sqrt{8/3}$-LQG sphere and TBM have some natural variants that differ in how one handles the issues of total area and special marked points; these variants are explained for example in \cite{map_making,quantum_spheres}. On both sides, there is a natural unit area sphere measure $d\CS$ in which the total area measure is a.s.\ one. On both sides, one can represent a sphere of arbitrary positive and finite area by a pair $(\CS,A)$, where $\CS$ is a unit area sphere and $A$ is a positive real number. The pair represents the unit area sphere $\CS$ except with area scaled by a factor of $A$ and distance (where defined) scaled by a factor of $A^{1/4}$.  On both sides it turns out to be natural to define infinite measures on quantum spheres such that the ``total area'' marginal has the form $A^\beta dA$ for some $\beta$. In particular, on both sides, one can define a natural ``grand canonical'' quantum sphere measure on spheres with $k$ marked points (see the notation in Section~\ref{subsec::prequel_overview}). Sampling from this infinite measure amounts to
\begin{enumerate}
\item first sampling a unit area sphere $\CS$,
\item then sampling $k$ marked points i.i.d.\ from the measure on $\CS$,
\item then independently selecting $A$ from the infinite measure $A^{k-7/2}dA$ and rescaling the sphere's area by a factor of $A$ (and distance by a factor of $A^{1/4}$).
\end{enumerate}

Theorem~\ref{thm::qle_metric}, stated below, applies to all of these variants. Recall that in the context of an infinite measure, {\em almost surely} (a.s.) means outside a set of measure zero.

\begin{theorem}
\label{thm::qle_metric}
Suppose that $\CS$ is an instance of the $\sqrt{8/3}$-quantum sphere, as defined in 
\cite{dms2014mating, quantum_spheres} (either the unit area version or the ``grand canonical'' version involving one or more distinguished points).  Let $(x_n)$ be a sequence of points chosen independently from the quantum area measure on $\CS$. Then it is a.s.\ the case that for each pair $x_i, x_j$, the quantity $\qdist(x_i, x_j)$ is uniquely determined by $\CS$ and the points $x_i$ and $x_j$. Moreover, it is a.s.\ the case that
\begin{enumerate}
\item $\qdist(x_i,x_j) = \qdist(x_j,x_i) \in (0,\infty)$ for all distinct $i$ and $j$, and
\item $\qdist$ satisfies the (strict) triangle inequality $\qdist(x_i,x_k) < \qdist(x_i,x_j) + \qdist(x_j,x_k)$ for all distinct $i$, $j$, and $k$.
\end{enumerate}
\end{theorem}

The fact that the triangle inequality in Theorem~\ref{thm::qle_metric} is a.s.\ strict implies that {\em if} the metric $\qdist$ can be extended to a geodesic metric on the entire LQG-sphere (something we will establish in the subsequent paper \cite{qle_continuity}) then in this metric it is a.s.\ the case that none of the points on the countable sequence lies on a geodesic between two other points in the sequence. This is unsurprising given that, in TBM, the measure of a geodesic between two randomly chosen points is a.s.\  zero. (This is well known and essentially immediate from the definition of TBM; see \cite{map_making} for some discussion of this point.)

The construction of the metric in Theorem~\ref{thm::qle_metric} is ``local'' in the sense that it only requires that the field near any given point is absolutely continuous with respect to a $\sqrt{8/3}$-LQG sphere.  In particular, Theorem~\ref{thm::qle_metric} yields a construction of the metric on a countable, dense subset of any $\sqrt{8/3}$-LQG surface chosen i.i.d.\ from the quantum measure.  Moreover, the results of the later papers \cite{qle_continuity,qle_determined} also apply in this generality, which allows one to define geodesic metrics on other $\sqrt{8/3}$-LQG surfaces, such as the torus constructed in \cite{drvtorus}.

The proof of Theorem~\ref{thm::qle_metric} is inspired by a closely related argument used in a paper by the second author, Sam Watson, and Hao Wu (still in preparation) to define a metric on the set of loops in a $\CLE_4$ process. To briefly sketch how the proof goes, suppose that we choose a $\sqrt{8/3}$-LQG sphere $\CS$ with marked points $x$ and $y$, and then choose a growth process $\qlegrowth$ from $x$ to $y$ and a conditionally independent growth process $\qlegrowthright$ from $y$ to $x$.  We also let $U$ be chosen uniformly in $[0,1]$ independently of everything else.  Let $\lawonall$ be the joint distribution of $(\CS,x,y,\qlegrowth,\qlegrowthright,U)$.  Since the natural measure on $\sqrt{8/3}$-LQG spheres is an infinite measure, so is $\lawonall$.  However, we can make sense of $\lawonall$ conditioned on $\CS$ as a probability measure.  Given $\CS$, we have that $\qdist(x,y)$ and $\qdist(y,x)$ are well defined as random variables denoting the respective time durations of $\qlegrowth$ and $\qlegrowthright$. As discussed above, we interpret $\qdist(x,y)$ (resp.\ $\qdist(y,x)$) as a measure of the distance from $x$ to $y$ (resp.\ $y$ to $x$).  Write $\lawonall^{x \to y}$ for the weighted measure $\qdist(x,y)d\lawonall$. In light of the uniqueness of Radon-Nikodym derivatives, in order to show that $\qdist(x,y) = \qdist(y,x)$ a.s., it will suffice to show that $\lawonall^{x \to y} = \lawonall^{y \to x}$.

\begin{figure}[ht!]
\begin{center}
\includegraphics[scale=0.85, page=4,trim={0 0 0 7.3cm},clip]{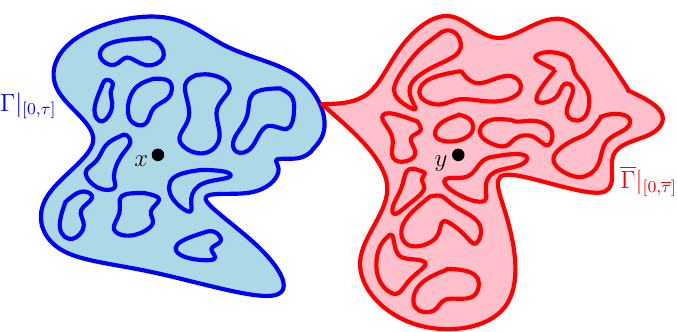}
\end{center}
\caption{\label{fig::figure_8} A ``figure-8'' consisting of meeting $\QLE(8/3,0)$ processes which illustrates the main input into the proof of Theorem~\ref{thm::qle_metric}.  Here, $\stoppingleft = U \qdist(x,y)$ and $\righthit$ is the first time that $\qlegrowthright$ hits $\qlegrowthleft_\stoppingleft$.  The blue (resp.\ red) arcs contained in the illustration of $\qlegrowth|_{[0,\stoppingleft]}$ (resp.\ $\qlegrowthright|_{[0,\righthit]}$) represent the regions which have been cut off from $y$ (resp.\ $x$).}
\end{figure}

The main input into the proof of this is Lemma~\ref{lem::mainsymmetry}, stated below (which will later be restated slightly more precisely and proved as Theorem~\ref{thm::qle_symmetry}). Suppose we sample $(\CS, x,y, \qlegrowth, \qlegrowthright,U)$ from $\lawonall^{x \to y}$, then let $\stoppingleft = U \qdist(x,y)$ --- so that $\stoppingleft$ is uniform in $[0,\qdist(x,y)]$, and then define $\righthit = \inf\{t \geq 0: \qlegrowthright_t \cap \qlegrowthleft_\stoppingleft \neq \emptyset\}$. Both $\qlegrowthleft|_{[0,\stoppingleft]}$ and $\qlegrowthright|_{[0,\righthit]}$ are understood as growth processes {\em truncated} at random times, as illustrated in Figure~\ref{fig::figure_8}.  We also let $\stoppingright = U \qdist(y,x)$ and $\lefthit = \inf\{t \geq 0 : \qlegrowthleft_t \cap \qlegrowthright_\stoppingright \neq \emptyset\}$.  Under $\lawonall^{y \to x}$, we have that $\stoppingright$ is uniform in $[0,\qdist(y,x)]$.

\begin{lemma}
\label{lem::mainsymmetry}
Using the definitions above, the $\lawonall^{x \to y}$-induced law of $(\CS,x,y,\qlegrowthleft|_{[0,\stoppingleft]}, \qlegrowthright |_{[0,\righthit]})$ is equal to the $\lawonall^{y \to x}$-induced law of $(\CS,x,y,\qlegrowthleft|_{[0,\lefthit]},\qlegrowthright|_{[0,\stoppingright]})$.
\end{lemma}

The proof of Lemma~\ref{lem::mainsymmetry} is in some sense the heart of the paper. It is established in Section~\ref{sec::qle_symmetry} (see Theorem~\ref{thm::qle_symmetry}), using tools developed over several previous sections, which in turn rely on the detailed understanding of $\SLE_6$ processes on $\sqrt{8/3}$-LQG spheres developed in \cite{dms2014mating, quantum_spheres} as well as in \cite{MS_IMAG,MS_IMAG2,MS_IMAG3,MS_IMAG4}.  We note that the intuition behind this symmetry is also sketched at the end of \cite{map_making} in the context of TBM.

To derive $\lawonall^{x \to y} = \lawonall^{y \to x}$ from Lemma~\ref{lem::mainsymmetry}, it will suffice to show the following:

\begin{lemma}
\label{lem::lawgivenfivetuple} The $\lawonall^{x \to y}$ conditional law of $\qlegrowthleft$, $\qlegrowthright$ given $(\CS,x,y,\qlegrowthleft|_{[0,\stoppingleft]}, \qlegrowthright |_{[0,\righthit]})$ is the same as the $\lawonall^{y \to x}$ conditional law of $\qlegrowthleft$, $\qlegrowthright$ given $(\CS,x,y,\qlegrowthleft|_{[0,\lefthit]}, \qlegrowthright |_{[0,\stoppingright]})$.
\end{lemma}

Intuitively, sampling from either conditional law should amount to just continuing the evolution of $\qlegrowthleft$ and $\qlegrowthright$ on $\CS$, beyond their given stopping times, independently of each other. However, it will take some work to make this intuition precise and we will carry this out in Section~\ref{sec::metric}.

We will see {\em a posteriori} that $\stoppingleft = \qdist(x,y) - \righthit$, which we will prove by using the fact that $\qdist(x,y) = \qdist(y,x)$ and the symmetry of Lemma~\ref{lem::mainsymmetry}, which implies that both $\righthit$ and $\stoppingleft$ are $\lawonall^{x \to y}$-conditionally uniform on $[0,\qdist(x,y)]$, once $(\CS,x,y)$ is given.  We will also use this fact to derive the triangle inequality. Note that if $z$ is a third point and we are working on the event that $\qdist(x,z) < \qdist(x,y)$, then $\qlegrowthleft_{\qdist(x,z)}$ and $\qlegrowthright_{\qdist(y,z)}$ must intersect each other, at least at the point $z$. In fact, it will not be hard to see that a.s.\ for some $\epsilon >0$ the processes $\qlegrowthleft_{\qdist(x,z)}$ and $\qlegrowthright_{\qdist(y,z) - \epsilon}$ still intersect. This implies that if $\stoppingleft = \qdist(x,z)$ then $\righthit \leq \qdist(y,z) - \epsilon < \qdist(y,z)$. Plugging in $\righthit=  \qdist(x,y) - \stoppingleft = \qdist(x,y) - \qdist(x,z)$, we obtain the strict triangle inequality.

The $\QLE(8/3,0)$ process that we construct will in fact be given as a subsequential limit of approximations which are defined by resampling the tip of an $\SLE_6$ after each $\delta > 0$ units of time.  The symmetry statements (e.g., Lemma~\ref{lem::mainsymmetry}) hold when $\qlegrowthleft$ and $\qlegrowthright$ arise as subsequential limits as $\delta \to 0$ along different subsequences.  Therefore the distance $\qdist$ that we define does not depend on the choice of subsequence.  We will not prove in this paper that the growth process $\qlegrowthleft$ itself does not depend on the choice of subsequence, but this will be a consequence of \cite{qle_continuity}.

\subsection{Observations and sequel overview}

In the course of establishing Theorem~\ref{thm::qle_metric}, it will also become clear (almost immediately from the definitions) that the growth process from $x_i$ to $x_j$ and the growth process from $x_i$ to $x_k$ a.s.\  agree up until some random time at which $x_j$ and $x_k$ are first separated from each other, after which the two processes evolve independently. Thus one can describe the full collection of growth processes from $x_i$ to all the other points in terms of a single ``branching'' growth process with countably many ``branch times'' (i.e., times at which some $x_j$ and $x_k$ are separated for the first time).

It will also become clear from our construction that when exploring from a marked point $x$ to a marked point $y$, one can make sense of the length of $\partial \fb{x}{t}$, and that as a process indexed by $t$ this evolves as the time-reversal of an excursion of a continuous state branching process (CSBP), with the jumps in this process corresponding to branch times.\footnote{We will see in \cite{qle_continuity} that if one defines a quantum-time $\QLE(8/3,0)$ on an infinite volume LQG surface, namely a $\sqrt{8/3}$ quantum cone, then the evolution of the boundary length is a process that matches the one described by Krikun (discrete) \cite{krikun2005uniform} and Curien and Le Gall \cite{cl2014hull} for the Brownian plane \cite{cl2012brownianplane}.}  We will review the definition of a CSBP in Section~\ref{subsec::csbp}. Letting $y$ vary over all of the points $x_j$, one obtains a branching version of a time-reversed CSBP excursion, and it will become clear that the law of this branching process agrees with the analogous law for TBM, as explained in \cite{map_making}.

All of this suggests that we are well on our way to establishing the equivalence of the $\sqrt{8/3}$-LQG sphere and TBM. As further evidence in this direction, note that it was established in \cite{map_making} that the time-reversed branching process (together with a countable set of real numbers indicating where along the boundary each ``pinch point'' occurs) contains all of the information necessary to reconstruct an instance of the entire Brownian map. That is, given a complete understanding of the exploration process rooted at a single point, one can a.s.\ reconstruct the distances between all pairs of points. This suggests (though we will not make this precise until the subsequent paper \cite{qle_continuity}) that, given the information described in the QLE branching process provided in this paper, one should also be able to recover an entire Brownian map instance.

In order to finish the project, the program in \cite{qle_continuity} is to:
\begin{enumerate}
\item Derive some continuity estimates and use them to show that $\qdist$ a.s.\ extends uniquely to a metric defined on the entire LQG sphere, and to establish H\"older continuity for the identify map from the sphere endowed with the Euclidean metric to the sphere endowed with the random metric, and then
\item Learn enough about the geodesics within the so-called metric net (as defined in \cite{quantum_spheres}) to allow us to show that the random metric satisfies the properties that are shown in \cite{quantum_spheres} to characterize TBM.
\end{enumerate}
This will imply that the metric space described by $\qdist$ has the law of TBM and, moreover, that the instance of TBM is a.s.\ determined by the underlying $\sqrt{8/3}$-LQG sphere.  The program in \cite{qle_determined} will be to prove that in the coupling between TBM and the $\sqrt{8/3}$-LQG sphere, the former a.s.\ determines the latter, i.e., to show that an instance of TBM a.s.\ has a canonical embedding into the sphere.  Thus we will have that the $\sqrt{8/3}$-LQG sphere and TBM are equivalent in the sense that an instance of one a.s.\ determines the other.  The ideas used in \cite{qle_determined} will be related to the arguments used in \cite{dms2014mating} to show that an instance of the peanosphere a.s.\ has a canonical embedding.

\subsection{Prequel overview}
\label{subsec::prequel_overview}

As noted in Section~\ref{subsec::overview}, both TBM and the $\sqrt{8/3}$-LQG sphere have natural infinite volume variants that in some sense correspond to grand canonical ensembles decorated by some fixed number of marked points. In this paper, because we deal frequently with exploration processes from one marked point to another, we will be particularly interested in the natural infinite measures on {\em doubly} marked spheres. We recall that
\begin{enumerate}
\item In \cite{map_making} this natural measure on doubly marked Brownian map spheres with two marked points is denoted $\mustwo$ (and more generally $\musk$ refers to the measure with $k$ marked points).
\item In \cite{quantum_spheres} the natural measure on doubly marked $\sqrt{8/3}$-LQG spheres is denoted by $\Mstwo$.
\end{enumerate}
As noted in Section~\ref{subsec::overview}, in both cases, the law of the overall area is given (up to a multiplicative constant) by $A^{-3/2} dA$. In both cases, the conditional law of the surface {\em given} $A$ is that of a sample from a probability measure on unit area surfaces (with the measure rescaled by a factor of $A$, and distance rescaled by $A^{1/4}$ --- though of course distance is not {\em a priori} defined on the LQG side). We remark that in much of the literature on TBM the unit area measure is the primary focus of attention (and it is denoted by $\musa$ in \cite{map_making}).

The paper \cite{quantum_spheres} explains how to explore a doubly marked surface $(\CS,x,y)$ sampled from $\Mstwo$ with an $\SLE_6$ curve drawn from $x$ to $y$. The paper \cite{map_making} explains how to explore a doubly marked surface $(\CS,x,y)$ sampled from $\mustwo$ by exploring the so-called ``metric net,'' which consists of the set of points that lie on the outer boundary of $\fb{x}{r}$ for some $r \in [0,d(x,y)]$.  (We are abusing notation slightly here in that $\CS$ represents a quantum surface in the first case and a metric space in the second, and these are {\em a priori} different types of objects.)  In both cases, the exploration/growth procedure ``cuts out'' a countable collection of disks, each of which comes with a well defined boundary length. Also in both cases, the process that encodes the boundary length corresponds (up to time change) to the set of jumps in the {\em time-reversal} of a $3/2$-stable L\'evy excursion with only positive jumps. Moreover, in both cases, the boundary length of each disk ``cut out'' is encoded by the length of the corresponding jump in the time-reversed $3/2$-stable L\'evy excursion. Finally, in both cases, the disks can be understood as conditionally independent measure-endowed random surfaces, given their boundary lengths.

The intuitive reason for the similarities between these two types of explorations is explained in the QLE paper \cite{ms2013qle}, and briefly reviewed in Section~\ref{sec::rpmintro}. The basic idea is that in the discrete models involving triangulations, the conditional law of the unexplored region (the component containing $y$) does not depend on the rule one uses to decide which triangle to explore next; if one is exploring via the Eden model, one picks a random location on the boundary to explore, and if one is exploring a percolation interface, one explores along a given path. The law of the set of disks cut out by the exploration is the same in both cases.

The law of a ``cut out'' disk, given that its boundary length is $L$, is referred to as $\mudl$ in \cite{map_making}. If one explores up to some stopping time before encountering $y$, then the conditional law of the unexplored region containing $y$ is that of a {\em marked} Brownian disk with boundary length $L$ (here $y$ is the marked point), and is referred to as $\mudonel$ in \cite{map_making}. It is not hard to describe how these two measures are related. If one forgets the marked point $y$, then both $\mudl$ and $\mudonel$ describe probability measures on the space of quantum disks; and from this perspective, the Radon-Nikodym derivative of $\mudonel$ w.r.t.\ $\mudl$ is given (up to multiplicative constant) by the total surface area. Given the quantum disk sampled from $\mudonel$, the conditional law of the marked point~$y$ is that of a sample from the quantum measure on the surface.

\begin{figure}[ht!]
\begin{center}
\includegraphics[scale=.85]{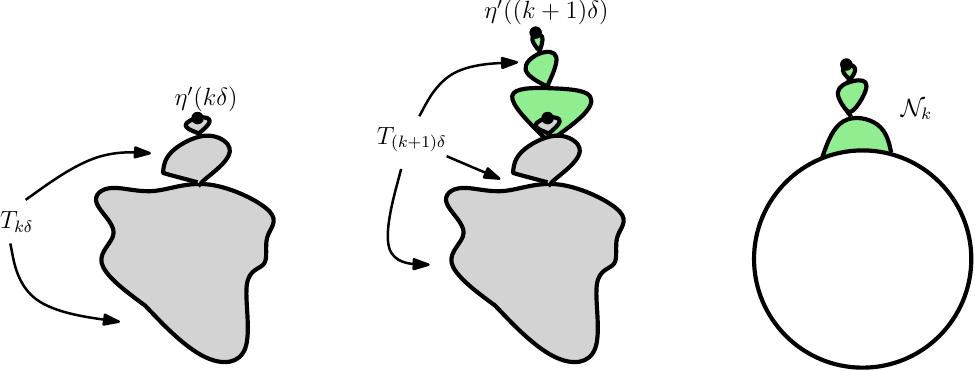}
\end{center}
\caption{\label{fig::sle_6_necklaces} {\bf Left:} A whole-plane $\SLE_6$ process $\eta'$ from $x$ to $y$ is drawn on top of a doubly-marked $\sqrt{8/3}$-LQG sphere $(\CS,x,y)$.  The grey region represents the part of $\CS$ that $\eta'([0,k\delta])$ has disconnected from $y$ at time $k \delta$, and $T_{k\delta}$ is the outer boundary of that region.  We note that $T_{k\delta}$ has cut points because whole-plane $\SLE_6$ has cut points.  {\bf Middle:} The green region indicates the additional area cut off from $y$ by $\eta'([0,(k+1)\delta])$, and $T_{(k+1)\delta}$ is the outer boundary $\eta'([0, (k+1)]\delta])$. {\bf Right:} It is conceptually useful to imagine that we ``cut'' along both $T_{k\delta}$ and $T_{(k+1)\delta}$ to produce a beaded quantum surface~$\CN_k$ (a so-called ``necklace'') attached to a loose ``string'' whose length is the length of $T_{(k+1)\delta} \cap T_{k\delta}$. The inner circle is understood to have total length equal to the length of $T_{k\delta}$, and the outer boundary has total length equal to the length of $T_{(k+1)\delta}$. }
\end{figure}

Precisely analogous statements are given in \cite{quantum_spheres} for the $\SLE_6$ exploration of a sample from $\Mstwo$.\footnote{These results are in turn consequences of the fact, derived by the authors and Duplantier in an infinite volume setting in \cite{dms2014mating}, that one can weld together two so-called L\'evy trees of $\sqrt{8/3}$-LQG disks to produce a new $\sqrt{8/3}$-LQG surface decorated by an independent $\SLE_6$ curve that represents the interface between the two trees.} The following objects are shown in \cite{quantum_spheres} well defined, and are analogous to objects produced by the measures $\mudl$ and $\mudonel$ in \cite{map_making}:
\begin{enumerate} \item {\bf The $\sqrt{8/3}$-LQG disk with boundary length $L$}. This is a random quantum surface whose law is the conditional law of a surface cut out by the $\SLE_6$ exploration, given only its boundary length (and not its embedding in the larger surface).
\item {\bf The {\it marked} $\sqrt{8/3}$-LQG disk with boundary length $L$}. This is a random quantum surface whose law is obtained by weighting the {\em unmarked} law by total area, and letting the conditional law of $y$ given the surface be that of a uniformly random sample from the area measure (normalized to be a probability measure). It represents the conditional law of the unexplored quantum component containing~$y$.
\end{enumerate}

\begin{proposition}
\label{prop::cutandspin}
Consider a doubly marked $\sqrt{8/3}$-LQG sphere decorated by an independent whole plane $\SLE_6$ path $\eta'$ from its first marked point~$x$ to its second marked point~$y$. We consider $\eta'$ to be parameterized by its quantum natural time. Fix an $s>0$ and let $T_s$ denote the outer boundary of the closed set $\eta'([0,s])$, i.e., the boundary of the $y$-containing component of the complement of $\eta'([0,s])$.  Then the conditional law of the $y$-containing region (given that its boundary length is $L_s$) is that of a marked $\sqrt{8/3}$-LQG disk with boundary length $L_s$. In particular, since this law is rotationally invariant, the overall law of the surface is unchanged by the following random operation: ``cut'' along $T_s$, rotate the disk cut out by a uniformly random number in $[0,L_s]$, and then weld this disk back to the beaded quantum surface $\eta'([0,s])$ (again matching up quantum boundary lengths).
\end{proposition}

It is natural to allow $s$ to range over integer multiples of a constant $\delta$. As illustrated in Figure~\ref{fig::sle_6_necklaces}, we let $\mathcal N_k$ denote the ``necklace'' described by the union $T_{k\delta} \cup T_{(k+1)\delta} \cup \eta'([k\delta, (k+1)\delta])$, which we interpret as a {\em beaded quantum surface} (see \cite{dms2014mating} and Section~\ref{subsec::quantum_surfaces}) attached to a ``string'' of some well defined length. Applying the above resampling for each integer multiple of $\delta$ corresponds to ``reshuffling'' these necklaces in the manner depicted in Figure~\ref{slotmachine}.

\begin{proposition}
\label{prop::limitgrowthmodel}
Fix $\delta > 0$ and apply the random rotation described in Proposition~\ref{prop::cutandspin} for each $s$ that is an integer multiple of $\delta$.  Taking any subsequential limit as $\delta \to 0$, we obtain a coupling of a $\sqrt{8/3}$-quantum sphere with a growth process on that sphere, such that the law of the ordered set of disks cut out by that process is the same as in the $\SLE_6$ case.
\end{proposition}

The growth process obtained this way is what we will call the {\em quantum natural time $\QLE(8/3,0)$} (as opposed to the {\em capacity time $\QLE(8/3,0)$} process described in \cite{ms2013qle}, which we expect but do not prove to be equivalent to the quantum time version). As already noted in Section~\ref{subsec::overview}, we will make extensive use of quantum natural time $\QLE(8/3,0)$ in this paper. When we use the term $\QLE(8/3,0)$ without a qualifier, we will mean the quantum natural time variant.

Let us highlight one subtle point about this paper.  Although we {\it a priori} construct only \emph{subsequential limits} for the growth process $\QLE(8/3,0)$ using the procedure described in Proposition~\ref{prop::limitgrowthmodel}, we ultimately show that the metric $\qdist$ defined on a countable sequence of i.i.d.\ points $(x_n)$ chosen from the quantum measure does not depend on the particular choice of subsequence.  Once we know this metric we know, for each $t \geq 0$ and each $x$ and $y$ in $(x_n)$, which points from the set $(x_n)$ lie in the set $\qlegrowth_t^{x \to y}$. Since $\qlegrowth_t^{x \to y}$ is closed, we would expect it to be given by precisely the closure of this set of points, which would imply that the growth process described in Proposition~\ref{prop::limitgrowthmodel} is a.s.\ defined as a true (non-subsequential) limit. This would follow immediately if we knew, say, that $\qlegrowth_t^{x \to y}$ was a.s.\ the closure of its interior. However, we will not prove in this paper that this is the case.  That is, we will not rule out the possibility that the boundary of $\qlegrowth_t^{x \to y}$ contains extra ``tentacles'' that possess zero quantum area and somehow fail to intersect any of the $(x_n)$ values.  Ruling out this type of behavior will be part of the program in \cite{qle_continuity}, where we establish a number of continuity estimates for $\QLE(8/3,0)$ and~$\qdist$.  Upon showing this, we will be able to remove the word ``subsequential'' from the statement of Proposition~\ref{prop::limitgrowthmodel}.

\subsection{Outline}
\label{subsec::outline}

The remainder of this article is structured as follows. In Section~\ref{sec::preliminaries} we review preliminary facts about continuous state branching processes, quantum surfaces, and conformal removability.  In Section~\ref{sec::rpmintro}, we recall some of the discrete constructions on random planar triangulations that appeared in \cite{ms2013qle}, which we use to explain and motivate our continuum growth processes. In particular, we will recall that on these triangulated surfaces {\em random metric explorations} are in some sense ``reshufflings'' of {\em percolation explorations}, and in Section~\ref{sec::spheres} we construct quantum-time $\QLE(8/3,0)$ using an analogous reshuffling of $\SLE_6$. In Section~\ref{sec::cut_points} we establish a certain symmetry property for continuum percolation explorations ($\SLE_6$) on $\sqrt{8/3}$-LQG surfaces (a precursor to the main symmetry result we require).  Then in Section~\ref{sec::qle_construction} we will give the construction of the quantum natural time variant of $\QLE(8/3,0)$.  In Section~\ref{sec::qle_symmetry}, we establish the main symmetry result we require, and in Section~\ref{sec::metric} we will finish the proof of Theorem~\ref{thm::qle_metric}.

\section{Preliminaries}
\label{sec::preliminaries}

\subsection{Continuous state branching processes}
\label{subsec::csbp}

We will now review some of the basic properties of continuous state branching processes (CSBPs) and their relationship to L\'evy processes.  CSBPs will arise in this article because they describe the time-evolution of the quantum boundary length of the boundary of a $\QLE(8/3,0)$.  We refer the reader to \cite{bertoin96levy} for an introduction to L\'evy processes and to \cite{legall1999spatial,kyp2006levy_fluctuations} for an introduction to CSBPs.

A CSBP with branching mechanism $\psi$ (or $\psi$-CSBP for short) is a Markov process~$Y$ on~$\R_+$ whose transition kernels are characterized by the property that
\begin{equation}
\label{eqn::csbp_def}
\ex{ \exp(-\lambda Y_t) \giv Y_s } = \exp(-Y_s u_{t-s}(\lambda)) \quad\text{for all}\quad t > s \geq 0
\end{equation}
where $u_t(\lambda)$, $t \geq 0$, is the non-negative solution to the differential equation
\begin{equation}
\label{eqn::csbp_diffeq}
\frac{\partial u_t}{\partial t}(\lambda) = -\psi(u_t(\lambda))\quad\text{for}\quad u_0(\lambda) = \lambda.
\end{equation}
Let
\begin{equation}
\label{eqn::phi_def}
\Phi(q) = \sup\{\theta \geq 0 : \psi(\theta) = q\}
\end{equation}
and let
\begin{equation}
\label{eqn::extinction}
\zeta = \inf\{t \geq 0: Y_t = 0\}
\end{equation}
be the extinction time for $Y$.  Then we have that \cite[Corollary~10.9]{kyp2006levy_fluctuations}
\begin{equation}
\label{eqn::csbp_exponential_integral}
\ex{ e^{-q \int_0^\zeta Y_s ds} } = e^{-\Phi(q) Y_0}.
\end{equation}

A $\psi$-CSBP can be constructed from a L\'evy process with only positive jumps and vice-versa \cite{lamp1967csbp} (see also \cite[Theorem~10.2]{kyp2006levy_fluctuations}).  Namely, suppose that $X$ is a L\'evy process with Laplace exponent $\psi$.  That is, if $X_0 = x$ then we have that
\[ \E[ e^{-\lambda (X_t-x)}] = e^{-\psi(\lambda) t}.\]
Let
\begin{equation}
\label{eqn::levy_to_csbp}
s(t) = \int_0^t \frac{1}{X_u} du \quad\text{and}\quad s^*(t) = \inf\{ r > 0 : s(r) > t\}.
\end{equation}
Then the time-changed process $Y_t = X_{s^*(t)}$ is a $\psi$-CSBP.  Conversely, if $Y$ is a $\psi$-CSBP and we let
\begin{equation}
\label{eqn::csbp_to_levy}
t(s) = \int_0^s Y_u du \quad\text{and}\quad t^*(s) = \inf\{ r > 0 : t(r) > s\}
\end{equation}
then $X_s = Y_{t^*(s)}$ is a L\'evy process with Laplace exponent $\psi$.

We will be interested in the particular case that $\psi(u) = u^\alpha$ for $\alpha \in (1,2)$.  For this choice, we note that
\begin{equation}
\label{eqn::csbp_u_form}
u_t(\lambda) = \left( \lambda^{1-\alpha} + (\alpha-1)t\right)^{1/(1-\alpha)}.
\end{equation}

\subsection{Quantum surfaces}
\label{subsec::quantum_surfaces}

Suppose that $h$ is an instance of (a form of) the Gaussian free field (GFF) on a planar domain $D$ and $\gamma \in [0,2)$ is fixed.  Then the $\gamma$-LQG surface associated with $h$ is described by the measure $\mu_h$ which is formally given by $e^{\gamma h(z)} dz$ where $dz$ denotes Lebesgue measure on~$D$.  Since the GFF $h$ does not take values at points (it is a random variable which takes values in the space of distributions), it takes some care to make this definition precise.  One way of doing so is to let, for each $\epsilon > 0$ and $z \in D$ such that $B(z,\epsilon) \subseteq D$, $h_\epsilon(z)$ be the average of $h$ on $\partial B(z,\epsilon)$ (see \cite[Section~3]{DS08} for more on the circle average process).  The process $(z,\epsilon) \mapsto h_\epsilon(z)$ is jointly continuous in $(z,\epsilon)$ and one can define $e^{\gamma h(z)} dz$ to be the weak limit as $\epsilon \to 0$ along negative powers of $2$ of $\epsilon^{\gamma^2/2} e^{\gamma h_\epsilon(z)} dz$ \cite{DS08}; the normalization factor $\epsilon^{\gamma^2/2}$ is necessary for the limit to be non-trivial.  We will often write $\mu_h$ for the measure $e^{\gamma h(z)} dz$.  In the case that $h$ has free boundary conditions, one can also construct the natural boundary length measure $\nu_h = e^{\gamma h(z)/2} dz$ in a similar manner.

The regularization procedure used to construct $\mu_h$ leads to the following change of coordinates formula \cite[Proposition~2.1]{DS08}.  Let
\begin{equation}
\label{eqn::q_def}
Q = \frac{2}{\gamma} + \frac{\gamma}{2} \quad\text{for}\quad \gamma \in (0,2).
\end{equation}
Suppose that $D_1,D_2$ are planar domains and $\varphi \colon D_1 \to D_2$ is a  conformal map.  If $h_2$ is (a form of) a GFF on $D_2$ and
\begin{equation}
\label{eqn::coord_change}
h_1= h_2 \circ \varphi + Q \log|\varphi'|
\end{equation}
then
\begin{equation}
\label{eqn::measures_equivalent}
\mu_{h_1}(A) = \mu_{h_2}(\varphi(A)) \quad\text{and}\quad \nu_{h_1}(A) = \nu_{h_2}(\varphi(A))
\end{equation}
for all Borel sets $A$.  This allows us to define an equivalence relation on pairs $(D,h)$ by declaring $(D_1,h_1)$ and $(D_2,h_2)$ to be equivalent as quantum surfaces if $h_1$ and $h_2$ are related as in~\eqref{eqn::coord_change}.  An equivalence class of such a $(D,h)$ is then referred to as a {\bf quantum surface}.  A choice of representative is referred to as an {\bf embedding} of the quantum surface.

More generally, suppose that $D_1,D_2$ are planar domains, $x_1^i,\ldots,x_n^i \in \ol{D}_i$ for $i=1,2$ are given points, and $h_i$ is a distribution on $D_i$.  Then we say that the marked quantum surfaces $(D_i,h_i,x_1^i,\ldots,x_n^i)$ are equivalent if there exists a conformal transformation $\varphi \colon D_1 \to D_2$, $\varphi(x_j^1) = x_j^2$ for each $1 \leq j \leq n$, and $h_1$, $h_2$ are related as in~\eqref{eqn::coord_change}.  Finally, if $\eta_1^i,\ldots,\eta_k^i$ is a collection of paths on $D_i$, then we say that the marked and path-decorated quantum surfaces $(D_i,h_i,x_1^i,\ldots,x_n^i,\eta_1^i,\ldots,\eta_k^i)$ are equivalent for $i=1,2$ if the marked quantum surfaces $(D_i,h_i,x_1^i,\ldots,x_n^i)$ are equivalent and $\varphi(\eta_j^1) = \eta_j^2$ for each $1 \leq j \leq k$ (where $\varphi$ is the associated conformal transformation).  We emphasize that we will refer to a quantum surface without marked points as simply a quantum surface.

In this work, we will be primarily interested in two types of quantum surfaces, namely quantum disks and spheres.  We will remind the reader of the particular construction of a quantum sphere that we will be interested in for this work in Section~\ref{subsec::levy_spheres}.  We also refer the reader to \cite{quantum_spheres} and \cite{dms2014mating} for a careful definition of a quantum disk as well as several equivalent constructions of a quantum sphere.

We will also consider so-called {\bf beaded quantum surfaces}, which can be defined as a pair $(D,h)$, modulo the equivalence relation described in \eqref{eqn::measures_equivalent}, except that $D$ is now a closed set (not necessarily homeomorphic to a disk) such that each component of its interior is homeomorphic to the disk, $h$ is only defined as distribution on each of these components, and $\varphi$ is allowed to be any homeomorphism from $D$ to another closed set that is conformal on each component of the interior of $D$.  One can also consider marked and path-decorated beaded quantum surfaces as above.

\subsection{Conformal removability}
\label{subsec::removability}

An LQG surface can be obtained by endowing a topological surface with both a good measure and a conformal structure in a random way.  (And we can imagine that these two structures are added in either order.)  Given two topological disks with boundary (each endowed with a good area measure in the interior, and a good length measure on the boundary) it is a simple matter to produce a new good-measure-endowed topological surface by taking a quotient that involves gluing (all or part of) the boundaries to each other in a boundary length preserving way.

The problem of {\em conformally welding} two surfaces is the problem of obtaining a conformal structure on the combined surface, given the conformal structure on the individual surfaces.  (See, e.g., \cite{bishop2007conformal} for further discussion and references.)  To make sense of this idea, we will draw from the theory of {\em removable sets}, as explained below.

A compact subset $K$ of a domain $D \subseteq \C$ is called (conformally) {\bf removable} if every homeomorphism from $D$ into $\C$ that is conformal on $D \setminus K$ is also conformal on all of $D$.  A Jordan domain $D \subseteq \C$ is said to be a {\bf H\"older domain} if any conformal transformation from~$\D$ to $D$ is H\"older continuous all of the way up to $\partial \D$.  It was shown by Jones and Smirnov \cite{js2000remove} that if~$K \subseteq D$ is the boundary of a H\"older domain, then $K$ is removable; it is also noted there that if a compact set $K$ is removable as a subset of $D$, then it is removable in any domain containing $K$, including all of $\C$.  Thus, at least for compact sets $K$, one can speak of removability without specifying a particular domain $D$.

The following proposition illustrates the importance of removability in the setting of quantum surfaces (see also \cite[Section~3.5]{dms2014mating}):
\begin{proposition}
\label{prop::removability_usage}
Suppose that $(D,h)$ is a quantum surface, $K \subseteq \ol{D}$ is compact such that $D \setminus K = D_1 \cup D_2$ for $D_1,D_2$ disjoint.  Suppose that $(D',h')$ is another quantum surface, $K' \subseteq \ol{D}'$ is compact such that $D' \setminus K' = D_1' \cup D_1'$ for $D_1',D_2'$ disjoint.  Assume that $(D_j,h)$ is equivalent to $(D_j',h')$ as a quantum surface for $j=1,2$ and that, furthermore the associated conformal transformations $\varphi_j \colon D_j \to D_j'$ for $j=1,2$ extend to a homeomorphism $D \to D'$.  If $K$ is conformally removable, then $(D,h)$ and $(D',h')$ are equivalent as quantum surfaces.
\end{proposition}
\begin{proof}
This follows immediately from the definition of conformal removability.
\end{proof}

One example of a setting in which Proposition~\ref{prop::removability_usage} applies is when the quantum surface $(D,h)$ is given by a so-called quantum wedge and $K$ is the range of an $\SLE_\kappa$ curve $\eta$ for $\kappa \in (0,4)$ \cite{SHE_WELD,dms2014mating}.  A quantum wedge naturally comes with two marked points~$x$ and~$y$, which are also the seed and the target point of the $\SLE$ curve.  In this case, the conformal maps $\varphi_j$ (which are defined on the left and right components of $D \setminus K$) are chosen so that the quantum length of the image of a segment of $\eta$ as measured from the left and right sides matches up.  With this choice, the $\varphi_j$ extend to a homeomorphism of the whole domain and it shown in \cite{RS05} that the range of $\eta$ is a.s.\  conformally removable, so Proposition~\ref{prop::removability_usage} applies.

\section{Eden model and percolation interface}
\label{sec::rpmintro}

\begin {figure}[ht!]
\begin {center}
\includegraphics[width=4.5in]{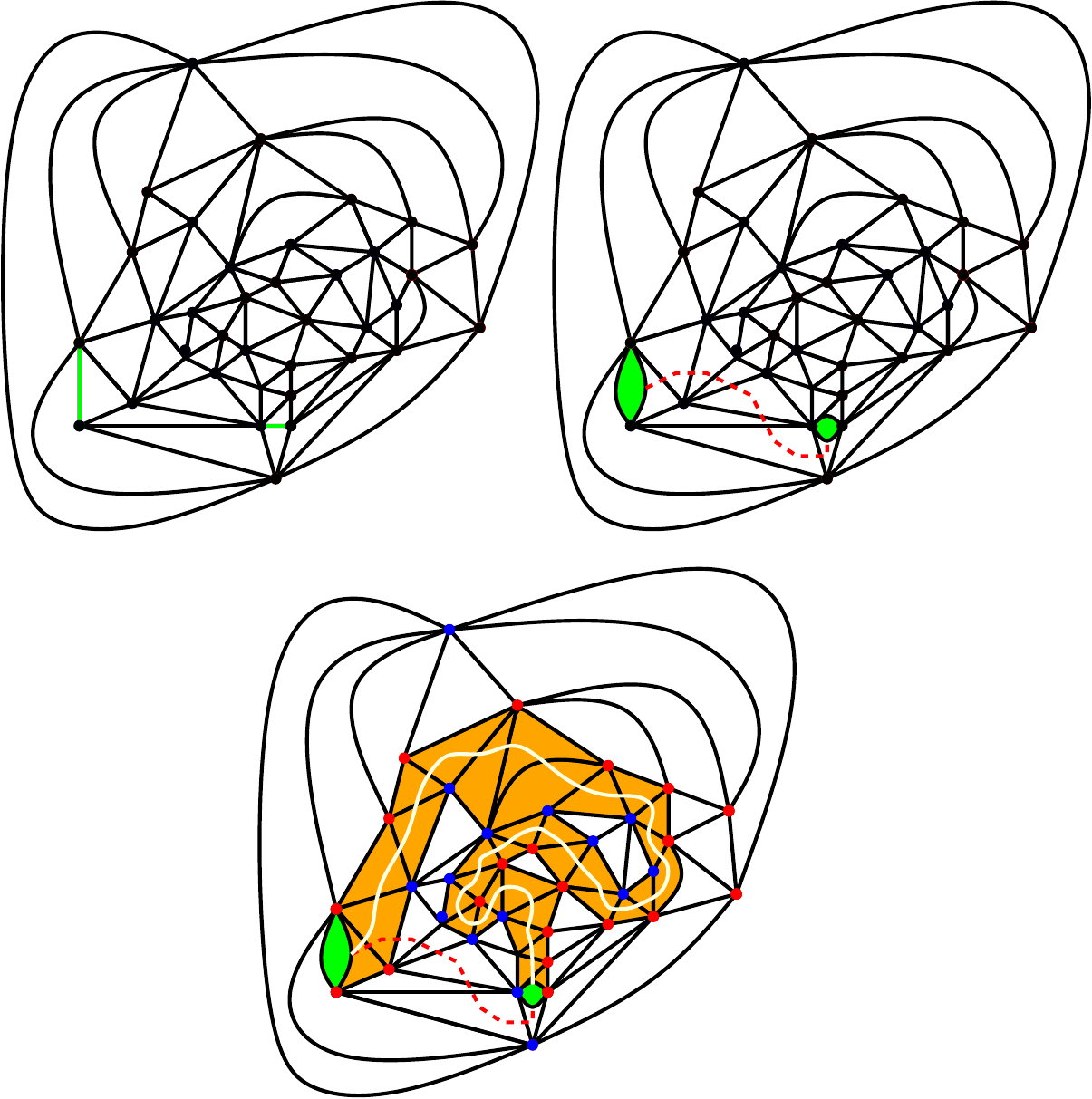}
\caption {\label{triangulation} {\bf Upper left:} a triangulation of the sphere together with two distinguished edges colored green.  {\bf Upper right:} It is conceptually useful to ``fatten'' each green edge into a $2$-gon.  We fix a distinguished non-self-intersecting dual-lattice path $p$ (dotted red line) from one $2$-gon to the other.  {\bf Bottom:} Vertices are colored red or blue with i.i.d.\ fair coins.  There is then a unique dual-lattice path from one $2$-gon to the other (triangles in the path colored orange) such that each edge it crosses either has opposite-colored endpoints and {\em does not} cross $p$, or has same-colored endpoints and {\em does} cross $p$.  The law of the orange path does not depend on the choice of $p$, since shifting $p$ across a vertex has the same effect as flipping the color of that vertex.  }
\end {center}
\end {figure}

\begin {figure}[ht!]
\begin {center}
\includegraphics [width=5.5in]{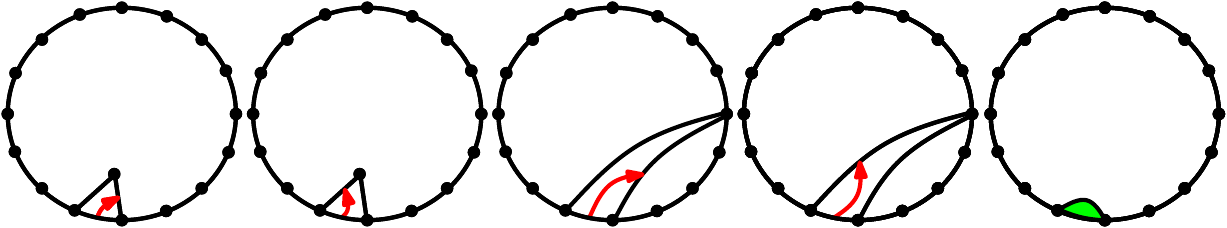}
\caption {\label{peelinstep} When ``exploring'' the polygon adjacent to a chosen edge on the boundary of the unexplored region, one encounters either a triangle whose third vertex is in the interior (leftmost two figures), a triangle whose third vertex is on the boundary (next two figures) or the terminal green $2$-gon (right figure).}
\end {center}
\end {figure}

\begin {figure}[!ht]
\begin {center}
\includegraphics [width=3.5in]{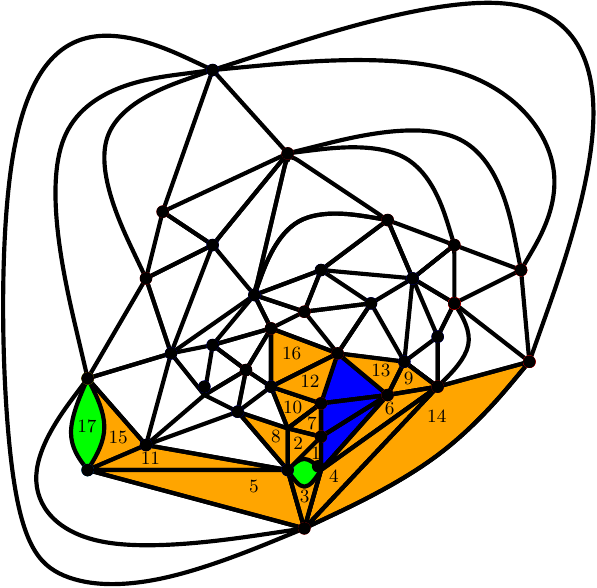}
\caption {\label{edentriangulation} Same as Figure~\ref{triangulation} except that one explores using the Eden model instead of percolation.  At each step, one chooses a uniformly random edge on the boundary of the unexplored region containing the target and explores the face incident to that edge.  The faces are numbered according to the order in which they were explored.  When the unexplored region is divided into two pieces, each with one or more triangles, the piece without the target is called a {\em bubble} and is never subsequently explored by this process.  In this figure there is only one bubble, which is colored blue.}
\end {center}
\end {figure}

In this section we briefly recall a few constructions from \cite[Section~2]{ms2013qle}, together with some figures included there.  Figure~\ref{triangulation} shows a triangulation $T$ of the sphere with two distinguished edges $e_1$ and $e_2$, and the caption describes a mechanism for choosing a random path in the dual graph of the triangulation, consisting of distinct triangles $t_1, t_2, \ldots, t_k$, that goes from $e_1$ to $e_2$.  It will be useful to imagine that we begin with a single $2$-gon and then grow the path dynamically, exploring new territory as we go.  At any given step, we keep track of the total number edges on the boundary of the already-explored region and the number of vertices remaining to be seen in the component of the unexplored region that contains the target edge.  The caption of Figure~\ref{peelinstep} explains one step of the exploration process.  The exploration process induces a Markov chain on the set of pairs $(m,n)$ with $m \geq 0$ and $n \geq 0$.  In this chain, the $n$ coordinate is a.s.\  non-increasing, and the $m$ coordinate can only increase by $1$ when the $n$ coordinate decreases by~$1$.

Now consider the version of the Eden model in which new triangles are only added to the unexplored region containing the target edge, as illustrated Figure~\ref{edentriangulation}.  In both Figure~\ref{triangulation} and Figure~\ref{edentriangulation}, each time an exploration step separates the unexplored region into two pieces (each containing at least one triangle) we refer to the one that does not contain the target as a {\em bubble}.  The exploration process described in Figure~\ref{triangulation} created two bubbles (the two small white components), and the exploration process described in Figure~\ref{edentriangulation} created one (colored blue).   We can interpret the bubble as a triangulation of a polygon, rooted at a boundary edge (the edge it shares with the triangle that was observed when the bubble was created).

The specific growth pattern in Figure~\ref{edentriangulation} is very different from the one depicted in Figure~\ref{triangulation}.  However, the analysis used in Figure~\ref{peelinstep} applies equally well to both scenarios.  The only difference between the two is that in Figure~\ref{edentriangulation} one re-randomizes the seed edge (choosing it uniformly from all possible values) after each step.

In either of these models, we can define $C_k$ to be the boundary of the target-containing unexplored region after $k$ steps.  If $(M_k, N_k)$ is the corresponding Markov chain, then the length of $C_k$ is $M_k + 2$ for each $k$.  Let $D_k$ denote the union of the edges and vertices in $C_k$, the edges and vertices in $C_{k-1}$ and the triangle and bubble (if applicable) added at step $k$, as in Figure~\ref{slotmachine}.  We refer to each $D_k$ as a {\em necklace} since it typically contains a cycle of edges together with a cluster of one or more triangles hanging off of it.  The  analysis used in Figure~\ref{peelinstep} (and discussed above) immediately implies the following:

\begin{proposition}
\label{prop::edenandpercolationinterface}
Consider a random rooted triangulation of the sphere with a fixed number $n>2$ of vertices together with two distinguished edges chosen uniformly from the set of possible edges.    If we start at one edge and explore using the Eden model as in Figure~\ref{edentriangulation}, or if we explore using the percolation interface of Figure~\ref{triangulation}, we will find that the following are the same:
\begin{enumerate}
\item The law of the Markov chain $(M_k,N_k)$ (which terminates when the target 2-gon is reached).
\item The law of the total number of triangles observed before the target is reached.
\item The law of the sequence $D_k$ of necklaces.
\end{enumerate}
Indeed, one way to construct an instance of the Eden model process is to start with an instance of the percolation interface exploration process and then randomly rotate the necklaces in the manner illustrated in Figure~\ref{slotmachine}.
\end{proposition}

\begin {figure}[!ht]
\begin {center}
\includegraphics [width=5.5in]{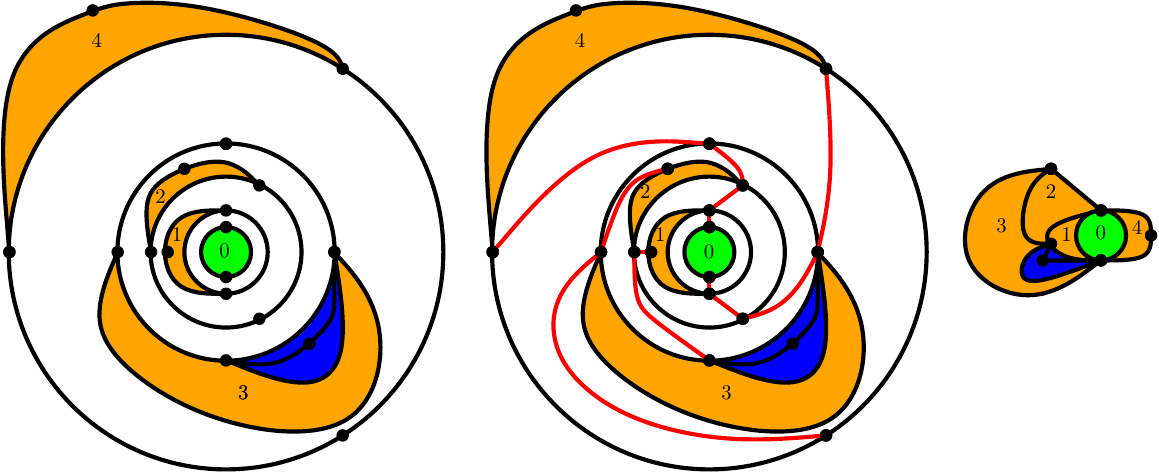}
\caption {\label{slotmachine} {\bf Left:} the first four necklaces (separated by white space) generated by an Eden model exploration.  {\bf Middle:} one possible way of identifying the vertices on the outside of each necklace with those on the inside of the next necklace outward.  {\bf Right:} The map with exploration associated to this identification.  If a necklace has $n$ vertices on its outer boundary, then there are $n$ ways to glue this outer boundary to the inner boundary of the next necklace outward.  It is natural to choose one of these ways uniformly at random, independently for each consecutive pair of necklaces.  Intuitively, we imagine that before gluing them together, we randomly spin the necklaces like the reels of a slot machine.  A fanciful interpretation of Proposition~\ref{prop::edenandpercolationinterface} is that if we take a percolation interface exploration as in Figure~\ref{triangulation} (which describes a sequence of necklaces) and we pull a slot machine lever that spins the necklaces independently (see \cite{ms2013qle}), then we end up with an Eden model exploration of the type shown in Figure~\ref{edentriangulation}.}
\end {center}
\end {figure}

\section{Infinite measures on quantum spheres}
\label{sec::spheres}

\subsection{L\'evy excursion description of doubly-marked quantum spheres}
\label{subsec::levy_spheres}

The purpose of this section is to review the results established in \cite{quantum_spheres} which are relevant for this article.  First, we let $\Mstwo$ be the measure which is defined as follows.  Suppose that $X_t$ is a $3/2$-stable L\'evy process with only upward jumps and let $I_t$ be its running infimum.  Then we let $\lexcursion$ be the It\^o excursion measure associated with the excursions that $X_t - I_t$ makes from $0$.  The law of the duration of such an excursion follows a power law. Indeed, following \cite{bertoin96levy}, the process of sampling $\lexcursion$ can be described as follows (see \cite[Section~VIII.4]{bertoin96levy}):
\begin{enumerate} 
\item Pick a lifetime $T$ from the measure $c T^{\rho-2} dT = c T^{-5/3} dT$ on $\R_+$ where $dT$ denotes Lebesgue measure and $c > 0$ is a constant.  Here, $\rho = 1-1/\alpha = 1/3$ (where $\alpha=3/2$) is the so-called positivity parameter of the process \cite[Section~VIII.1]{bertoin96levy}.
\item Given $T$, pick a unit length excursion from the normalized excursion measure ${\mathbf n}$ associated with a $3/2$-stable L\'evy process with only positive jumps and then rescale it spatially and in time so that it has length $T$.
\end{enumerate}

As explained in \cite{quantum_spheres}, we can construct a doubly-marked quantum sphere $(\CS,x,y)$ decorated by a non-crossing path $\eta'$ connecting $x$ and $y$ from such an excursion $e \colon [0,T] \to \R_+$ as follows.  For each upward jump of $e$, we sample a conditionally independent quantum disk whose boundary length is equal to the size of the jump.  We assume that each of the quantum disks has a marked boundary point sampled uniformly from its boundary measure together with a uniformly chosen orientation of its boundary.  We assume that the marked points and orientations are chosen conditionally independently given the realizations of the quantum disks.  Let $\CD$ denote the collection of marked and oriented quantum disks sampled in this way.  Then the pair $e,\CD$ together uniquely determines a doubly-marked surface $(\CS,x,y)$ which is homeomorphic to the sphere together with a non-crossing path $\eta'$ which connects $x$ and $y$.  For each $t \geq 0$, we let $U_t$ be the component of $\CS \setminus \eta'([0,t])$ which contains $y$.  Then the time-reversal of $e$ describes the evolution of the quantum boundary length of $\partial U_t$ and the jumps of $e$ describe the boundary lengths of the quantum disks that $\eta'$ cuts off from~$y$.  The time-parameterization of $\eta'$ so that the quantum boundary length of $\partial U_t$ is equal to $e(T-t)$ is the so-called {\bf quantum natural time} introduced in \cite{dms2014mating}.

One of the main results of \cite{quantum_spheres} is that a doubly-marked surface/path pair $(\CS,x,y),\eta'$ produced from $\Mstwo$ \emph{conditioned} to have quantum area equal to $1$ (although $\Mstwo$ is infinite, this conditioning yields a probability measure) has the law of the unit area quantum sphere constructed in \cite{dms2014mating}, the points $x,y$ conditional on $\CS$ are chosen uniformly at random from the quantum measure, and the conditional law of $\eta'$ given $x,y,\CS$ is that of a whole-plane $\SLE_6$ process connecting $x$ and $y$.  This holds more generally when we condition the surface to have quantum area equal to $m$ for any fixed $m > 0$ except in this setting $\CS$ is a quantum sphere of area $m$ rather than $1$.  (As noted earlier, a sample from the law of such a surface can be produced by starting with a unit area quantum sphere and then scaling its associated area measure by the factor $m$.)

The relationship between $\Mstwo$ and the law of a unit area quantum sphere decorated with an independent whole-plane $\SLE_6$ process implies that $\Mstwo$ possesses certain symmetries.  These symmetries will be important later on so we will pause for a moment to point them out.

\begin{itemize}
\item If we condition on $\CS$, then the points $x$ and $y$ are both chosen independently from the quantum area measure on $\CS$.
\item If we condition on $x$, $y$, and $\CS$, then $\eta'$ is whole-plane $\SLE_6$ from $x$ to $y$.
\item The amount of quantum natural time elapsed for $\eta'$ to travel from $x$ to $y$ is equal to $T$ (the time corresponding to the L\'evy excursion).
\end{itemize}

This also implies that $\Mstwo$ is invariant under the operation of swapping $x$ and $y$ and then reversing the time of $\eta'$ \cite{MS_IMAG4} (with the quantum natural time parameterization).  To see the symmetry of the quantum natural time parameterization under time-reversal, we have from the previous observations that the law of the ordered collection of bubbles cut off by $\eta'$ from its target point is invariant under the operation of swapping $x$ and $y$ and reversing the time of $\eta'$.  The claim thus follows because the quantum natural time parameterization can be constructed by fixing $j \in \N$, counting the number $N(e^{-j-1},e^{-j})$ of bubbles cut off by $\eta'$ with quantum boundary length in $[e^{-j-1},e^{-j}]$, and then normalizing by a constant times the factor $e^{3/2j}$.  That this is the correct normalization follows since the L\'evy measure for a $3/2$-stable L\'evy process is given by a constant times $u^{-5/2} du$ where $du$ denotes Lebesgue measure on $\R_+$.  See, for example, \cite[Section~6.2]{quantum_spheres} for additional discussion of this point as well as Remark~\ref{rem::qle_quantum_natural_defined} below in the context of the construction of $\QLE(8/3,0)$.

We also emphasize that under $\Mstwo$, we have that:
\begin{itemize}
\item $\eta'(t)$ is distributed uniformly from the quantum boundary measure on $\partial U_t$ (see \cite[Proposition~6.4]{quantum_spheres}). 
\item The components of $\CS \setminus \eta'([0,t])$, viewed as quantum surfaces, are conditionally independent given their boundary lengths.  Those components which do not contain~$y$ are quantum disks given their boundary lengths.  The component which does contain $y$ has the law of a quantum disk with the given boundary length weighted by its total quantum area.  
\end{itemize}

\subsection{Weighted measures}
\label{subsec::weighted_measures}

Throughout, we will work with the following two measures which are defined with $\Mstwo$ as the starting point.  Namely, with $T$ equal to the length of the associated L\'evy excursion and $X_t$ the quantum boundary length of the complementary component of $\eta'([0,t])$ containing $y$ we write
\begin{align}
  d \MstwoW &= \one_{[0,T]}(\ttime) d\ttime d\Mstwo \quad\text{and} \label{eqn::m_w_def}\\
  d \MstwoD &= \one_{[0,T]}(\ttime) \frac{1}{X_\ttime} d\ttime d \Mstwo = \frac{1}{X_\ttime} d \MstwoW \label{eqn::m_d_def}
\end{align}
where $d\ttime$ denotes Lebesgue measure.  We note that the marginal of $\MstwoW$ on $(\CS,x,y)$ and $\eta'$ is given by $T d\Mstwo$, i.e., by weighting $\Mstwo$ by the length of the L\'evy excursion.  (The additional subscript ``${\mathrm W}$'' is to indicate that $\MstwoW$ is a weighted measure.)  It will be convenient throughout to think of $\MstwoW$ as a measure on triples $\Bigl( \ttime$, $(\CS,x,y)$, $\eta'\Bigr)$ where $\ttime$ is a uniformly random point chosen from the total length of the L\'evy excursion.

The marginal of $\MstwoD$ on $(\CS,x,y)$ and $\eta'$ is given by $D d \Mstwo$ where
\begin{equation}
\label{eqn::d_def}
	D = \int_0^T \frac{1}{X_s} ds.
\end{equation}
As we will see later, after the $\QLE$ ``reshuffling'' this will have the interpretation of taking $\Mstwo$ and then weighting it by the amount of $\QLE(8/3,0)$ quantum distance from~$x$ to~$y$.  (The additional subscript ``${\mathrm D}$'' is to indicate that $\MstwoD$ is the ``distance weighted'' measure.)

We finish this section by recording the following proposition which relates the conditional law of~$\MstwoW$ and~$\MstwoD$ given $\ttime$ and $X_\ttime$ to $\Mstwo$.

\begin{proposition}
\label{prop::measure_properties}
\begin{enumerate}[(i)]
\item\label{it::weighted_given_time} Given~$\ttime$, the conditional distribution of~$\MstwoW$ is the same as the conditional distribution of~$\Mstwo$ when we condition on the event that the length of the L\'evy excursion is at least~$\ttime$.
\item\label{it::weighted_given_time_and_x} For both ${\mathsf m} = \MstwoW$ and ${\mathsf m} =\MstwoD$, given~$\ttime$ and $X_\ttime$, the conditional distribution of~${\mathsf m}$ is the same as the conditional distribution of~$\Mstwo$ when we condition on the event that the length of the L\'evy excursion is at least~$\ttime$ and the given value of $X_\ttime$.
\end{enumerate}
\end{proposition}
\begin{proof}
We will explain the argument in the case that ${\mathsf m} = \MstwoW$; the same argument gives part~\eqref{it::weighted_given_time_and_x}.

If we fix the value of $\ttime$, the conditional distribution of the L\'evy excursion in the definition of $\MstwoW$ is given by $c \one_{[\ttime,\infty)}(T) d{\mathbf n}_T T^{-5/3} dT$ where ${\mathbf n}_T$ is the measure on $3/2$-stable L\'evy excursions which arises by scaling ${\mathbf n}$ spatially and in time so that the excursion length is equal to $T$.  This representation clearly implies~\eqref{it::weighted_given_time}.  A similar argument gives~\eqref{it::weighted_given_time_and_x}.
\end{proof}

\subsection{Continuum scaling exponents}
\label{subsec::continuum_scaling_exponents}

We now determine the distribution of $D$ (as defined in~\eqref{eqn::d_def}) and $A$ (total quantum area of the surface) under $\Mstwo$.

\begin{proposition}
\label{prop::levy_decay}
There exists constants $c_0,c_1 > 0$ such that
\begin{align}
\Mstwo[ D \geq t ] &= \frac{c_0}{t^2} \quad\text{and} \label{eqn::levy_distance_decay}\\
\Mstwo[A \geq a] &= \frac{c_1}{a^{1/2}}. \label{eqn::levy_area_decay}
\end{align}
\end{proposition}

Note that the exponents in~\eqref{eqn::levy_distance_decay},~\eqref{eqn::levy_area_decay} match the corresponding exponents derived in \cite{map_making}.

\begin{proof}[Proof of Proposition~\ref{prop::levy_decay}]
We note that~\eqref{eqn::levy_area_decay} is explained just after the statement of \cite[Theorem~1.4]{quantum_spheres}.  Therefore we only need to prove~\eqref{eqn::levy_distance_decay}.

For an excursion $e \colon [0,T] \to \R_+$ sampled from $\lexcursion$ we write $e^* = \sup_{t \in [0,T]} e(t)$.  By scaling and the explicit form of $\lexcursion$ described above, it is not difficult to see by making the change of variables $u=t T^{-2/3}$ that there exists a constant $c_0 > 0$ such that
\begin{align}
\label{eqn::excursion_measure_max_tail}
     \lexcursion[ e^* \geq t ] 
&=\int_0^\infty  {\mathbf n}[ e^* \geq t T^{-2/3} ] T^{-5/3} dT =  \frac{c_0}{t}.
\end{align}
For each $\epsilon > 0$, we let $T_\epsilon = \inf\{t \geq 0 : e(t) \geq \epsilon\}$ and let
\[ D^\epsilon = \int_{T_\epsilon}^T \frac{1}{X_s} ds.\]
Suppose that $Y$ is a $u^{3/2}$-CSBP with $Y_0 = \epsilon$.  Let $\zeta = \inf\{ t \geq 0 : Y_t = 0\}$ be the extinction time of $Y$.  Then it follows from~\eqref{eqn::csbp_def} and~\eqref{eqn::csbp_u_form} that
\[ \p[ \zeta \leq t] = \lim_{\lambda \to \infty} \E[ e^{- \lambda Y_t}] = \lim_{\lambda \to \infty} e^{ -\epsilon u_t(\lambda)} = e^{-4\epsilon / t^2}.\]
Therefore
\[\p[ \zeta \geq t] = \frac{4 \epsilon}{t^2} + o(\epsilon) \quad\text{as}\quad \epsilon \to 0.\]
By combining this with~\eqref{eqn::excursion_measure_max_tail}, it therefore follows that there exists a constant $c_1 > 0$ such that
\begin{align*}
 \Mstwo[ D^\epsilon \geq t ]
&= \Mstwo[ D^\epsilon \geq t \giv e^* \geq \epsilon] \Mstwo[ e^* \geq \epsilon]\\
&= \frac{c_1 \epsilon}{t^2} \cdot \frac{c_0}{\epsilon} + o(1) = \frac{c_0 c_1}{t^2} + o(1) \quad\text{as}\quad \epsilon \to 0.
 \end{align*}
Sending $\epsilon \to 0$ implies the result.
\end{proof}

\section{Meeting in the middle}
\label{sec::cut_points}

In this section, we shall assume that we are working in the setting described in Section~\ref{subsec::weighted_measures}.  The main result is the following theorem which proves a certain symmetry statement for the measure $\MstwoD$.  This result will later be used in Section~\ref{sec::qle_symmetry} to prove an analogous symmetry result for $\QLE(8/3,0)$ which, in turn, is one of the main inputs in the proof of Theorem~\ref{thm::qle_metric}.  See Figure~\ref{fig::doubly_marked_sphere} for an illustration of the result.

\begin{figure}[ht!]
\begin{center}
\includegraphics[scale=0.85, page=1, trim={0 6.6cm 0 0},clip]{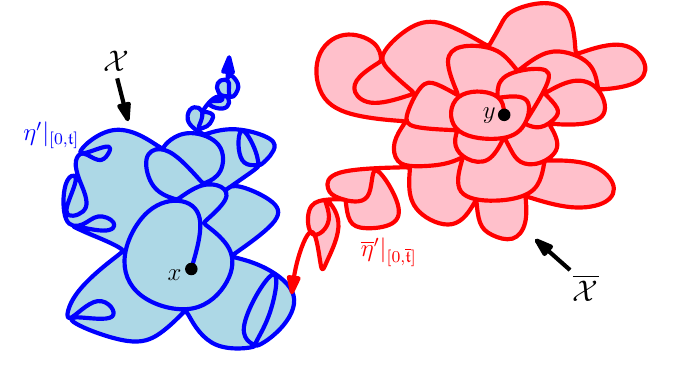}
\end{center}
\caption{\label{fig::doubly_marked_sphere} Illustration of the setup for Theorem~\ref{thm::typical_cut_point}, the main result of Section~\ref{sec::cut_points}.  Shown is a doubly-marked quantum sphere $(\CS,x,y)$ decorated with a whole-plane $\SLE_6$ process~$\eta'$ connecting~$x$ and~$y$.  The blue path shows~$\eta'$ drawn up to a time~$\ttime$ which is uniform in the total amount of quantum distance time required by~$\eta'$ to connect $x$ and $y$ and the red path is the time-reversal $\ol{\eta}'$ drawn up until the first time $\ol{\ttime}$ that it hits $\eta'([0,\ttime])$.  Let $\CX$ (resp.\ $\ol{\CX}$) consist of $\eta'|_{[0,\ttime]}$ (resp.\ $\ol{\eta}'|_{[0,\ol{\ttime}]}$) and the part of $\CS$ separated by $\eta'([0,\ttime])$ (resp.\ $\ol{\eta}'([0,\ol{\ttime}])$) from $y$ (resp.\ $x$) and let $\CB$ be the part of $\CS$ which is not in $\CX$ and $\ol{\CX}$.  In the illustration, $\CX$ (resp.\ $\ol{\CX}$) is shown in light blue (resp.\ light red).  In Theorem~\ref{thm::typical_cut_point}, we show that the $\MstwoD$ distribution of the triple $(\CX,\ol{\CX},\CB)$ is invariant under the operation of swapping $\CX$ and $\ol{\CX}$.}
\end{figure}

\begin{theorem}
\label{thm::typical_cut_point}
Suppose that $(\CS,x,y)$ is a doubly-marked quantum sphere, $\eta'$ is a path on $\CS$ from $x$ to $y$, and $\ttime \geq 0$, where these objects are sampled from $\MstwoD$ as in Section~\ref{subsec::weighted_measures}.  Let $\ol{\eta}'$ be the time-reversal of $\eta'$ and let $\ol{\ttime}$ be the first time that $\ol{\eta}'$ hits $\eta'([0,\ttime])$. Let
\begin{itemize}
\item $\CX$ be the path-decorated and beaded quantum surface parameterized by the union of $\eta'|_{[0,\ttime]}$ and the part of $\CS$ separated from $y$ by $\eta'([0,\ttime])$,
\item $\ol{\CX}$ be the path-decorated and beaded quantum surface parameterized by the union of $\ol{\eta}'|_{[0,\ol{\ttime}]}$ and the part of $\CS$ separated from $x$ by $\ol{\eta}'([0,\ol{\ttime}])$, and
\item $\CB$ be the (disk-homeomorphic) quantum surface parameterized by $\CS \setminus ( \CX \cap \ol{\CX} )$.
\end{itemize}
We view $\CX$ (resp.\ $\ol{\CX}$) as a random variable taking values in the space of path-decorated and beaded quantum surfaces with a marked point which corresponds to $\eta'(\ttime)$ (resp.\ $\ol{\eta}'(\ol{\ttime})$) and we view $\CB$ as a random variable taking values in the space of quantum surfaces.  Under $\MstwoD$, we have that $\CX$, $\ol{\CX}$, and $\CB$ are conditionally independent given the quantum lengths of the boundaries of $\CX$ and $\ol{\CX}$ (which together determine the boundary length of $\CB$).  Moreover, the $\MstwoD$ distribution of $(\CX,\ol{\CX},\CB)$ is invariant under the operation of swapping~$\CX$ and~$\ol{\CX}$.
\end{theorem}

\begin{remark}
\label{rem::sle_whole_picture_not_symmetric}
We emphasize that the assertion of Theorem~\ref{thm::typical_cut_point} does not give that the surface $\CS$ decorated by the paths $\eta'|_{[0,\ttime]}$ and $\ol{\eta}'|_{[0,\ol{\ttime}]}$ is invariant under swapping $\eta'|_{[0,\ttime]}$ and $\ol{\eta}'|_{[0,\ol{\ttime}]}$.  This statement does not hold because there is an asymmetry in that $\eta'(\ttime) \notin \ol{\eta}'([0,\ol{\ttime}])$ while we have that $\ol{\eta}'(\ol{\ttime}) \in \eta'([0,\ttime])$.  Rather, Theorem~\ref{thm::typical_cut_point} implies that the induced distribution on triples of $(\CX,\ol{\CX},\CB)$ is invariant under swapping~$\CX$ and~$\ol{\CX}$.  The difference between the two statements is that the first statement depends on how the surfaces which correspond to $\CX$, $\ol{\CX}$, and $\CB$ are glued together (which determines the locations of the tips) while the second statement does not depend on how everything is glued together.  In the $\QLE(8/3,0)$ analog of Theorem~\ref{thm::typical_cut_point}, which is stated as Theorem~\ref{thm::qle_symmetry} below (and will be derived as a consequence of Theorem~\ref{thm::typical_cut_point}), we do have the symmetry of the whole picture because the tips are ``lost'' in the ``reshuffling'' procedure used to construct $\QLE(8/3,0)$ from $\SLE_6$.
\end{remark}

Our ultimate aim in this section is to deduce Theorem~\ref{thm::typical_cut_point} by showing that the $\MstwoD$ distribution on $(\CX,\ol{\CX},\CB)$ can be constructed from $\MstwoW$ by ``conditioning'' on the event that $\ttime$ is a cut time for $\eta'$.  This will indeed lead to the desired result because, as we will explain in Proposition~\ref{prop::cw_properties} just below, $\MstwoW$ possesses a symmetry property which is similar to that described in Theorem~\ref{thm::typical_cut_point}.

\begin{figure}[ht!]
\begin{center}
\includegraphics[scale=0.85, page=2, trim ={0 6.6cm 0 0},clip]{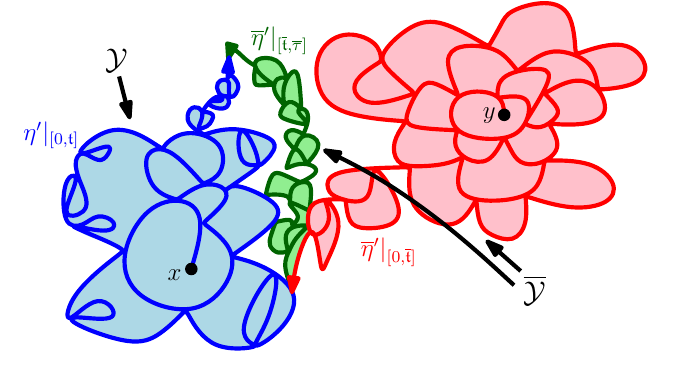}
\end{center}
\caption{\label{fig::doubly_marked_sphere2} Illustration of the setup for Proposition~\ref{prop::cw_properties} (continuation of Figure~\ref{fig::doubly_marked_sphere}).  Let $\CY$ (resp.\ $\ol{\CY}$) consist of $\eta'|_{[0,\ttime]}$ (resp.\ $\ol{\eta}'|_{[0,\ol{\tau}]}$ where $\ol{\tau}$ is the first time that $\ol{\eta}'$ hits $\eta'(\ttime)$) and the part of $\CS$ separated by $\eta'([0,\ttime])$ (resp.\ $\ol{\eta}'([0,\ol{\tau}])$) from $y$ (resp.\ $x$) and let $\CQ$ be the part of $\CS$ which is not in $\CY$ and $\ol{\CY}$.  In the illustration, $\CY$ (resp.\ $\ol{\CY}$) is shown in light blue (resp.\ light red and light green).  In Proposition~\ref{prop::cw_properties}, we show that the $\MstwoW$ distribution of the triple $(\CY,\ol{\CY},\CQ)$ is invariant under the operation of swapping~$\CY$ and~$\ol{\CY}$.}
\end{figure}

\begin{proposition}
\label{prop::cw_properties}
Suppose that $(\CS,x,y)$ is a doubly-marked quantum sphere, $\eta'$ is a path connecting $x$ to $y$, and $\ttime \geq 0$, where these objects are sampled from $\MstwoW$ as defined in Section~\ref{subsec::weighted_measures}.  Let $\ol{\tau}$ be the first time that $\ol{\eta}'$ hits $\eta'(\ttime)$. (We emphasize that $\ol{\tau}$ is not the same as $\ol{\ttime}$, the first time that $\ol{\eta}'$ hits $\eta'([0,\ttime])$.)  Let
\begin{itemize}
\item $\CY$ be the path-decorated and beaded quantum surface parameterized by the union of $\eta'|_{[0,\ttime]}$ and the part of $\CS$ which is separated from $y$ by $\eta'([0,\ttime])$,
\item $\ol{\CY}$ be the path-decorated and beaded quantum surface parameterized by the union of $\ol{\eta}'|_{[0,\ol{\tau}]}$ and the part of $\CS$ which is separated from $x$ by $\ol{\eta}'([0,\ol{\tau}])$, and
\item $\CQ$ be the quantum surface parameterized by $\CS \setminus (\CY \cup \ol{\CY})$.
\end{itemize}
We view $\CY$ (resp.\ $\ol{\CY}$) as a random variable taking values in the space of path-decorated and beaded quantum surfaces with a single marked point which corresponds to $\eta'(\ttime)$ (resp.\ $\ol{\eta}'(\ol{\tau})$) and we view $\CQ$ as a random variable taking values in the space of beaded quantum surfaces.  Then $\MstwoW$ is invariant under the operation of swapping $\CY$ and $\ol{\CY}$.  Moreover, under $\MstwoW$, we have that $\CY$, $\ol{\CY}$, and $\CQ$ are conditionally independent given the quantum lengths of the boundaries of $\CY$ and $\ol{\CY}$ (which together determine the boundary length of~$\CQ$).  
\end{proposition}
\begin{proof}
We start with the first assertion of the proposition.  As mentioned in Section~\ref{subsec::levy_spheres}, we know that if we condition on the quantum area $m$ of $\CS$ then the joint law of $(\CS,x,y)$ and $\eta'$ is given by a quantum sphere of quantum area $m$, $x$ and $y$ are independently and uniformly chosen from the quantum area measure, and, given $x,y$, $\eta'$ is an independent whole-plane $\SLE_6$ connecting $x$ and $y$.  By the reversibility of whole-plane $\SLE_6$ established in \cite{MS_IMAG4}, it thus follows that the joint law of $(\CS,x,y)$ and $\eta'$ is invariant under swapping $x$ and $y$ and reversing the time of $\eta'$ when $m$ is fixed.  Since $\ttime$ is uniform in the amount of quantum natural time $T$ required by $\eta'$ to connect $x$ and $y$, we have by symmetry that $\ol{\tau}$ is uniform in $[0,T]$.  This proves the first part.

It follows from the construction of $\MstwoW$ that $\CY$ is independent of $\ol{\CY}$ and $\CQ$ given its quantum boundary length (recall Proposition~\ref{prop::measure_properties}) and, by symmetry, that $\ol{\CY}$ is independent of $\CY$ and $\CQ$ given its boundary length.  This implies the second assertion of the proposition.
\end{proof}

We will establish Theorem~\ref{thm::typical_cut_point} using the following strategy.  Suppose that we have a triple $T$, $(\CS,x,y)$, $\eta'$ which consists of a number $T \geq 0$, a doubly-marked finite volume quantum surface $\CS$ homeomorphic to $\s^2$, and a non-crossing path $\eta'$ on $\CS$ connecting $x$ and $y$ such that $T$ is equal to the total amount of quantum natural time taken by $\eta'$ to go from~$x$ to~$y$.  We assume that~$\CS$ is parameterized by~$\C$ with~$0$ corresponding to~$x$ and~$\infty$ corresponding to $y$.  Let $\ttime$ be uniform in $[0,T]$ and let $U$ be the unbounded component of $\C \setminus \eta'([0,\ttime])$.  Let $\ol{\eta}'$ be the time-reversal of $\eta'$ and let $\ol{\ttime}$ be the first time that $\ol{\eta}'$ hits $\eta'([0,\ttime])$.  Fix $\epsilon > 0$ and let~$E_\epsilon$ be the event that $\ol{\eta}'(\ol{\ttime})$ is contained in the interval of~$\partial U$ starting from $\eta'(\ttime)$ and continuing in the counterclockwise direction until reaching quantum length~$\epsilon$.  (In the case that $\partial U$ has quantum length at most~$\epsilon$, we take this interval to be all of $\partial U$.)  In other words, $E_\epsilon$ is the event that $\ol{\eta}'(\ol{\ttime})$ is contained in the quantum length~$\epsilon$ interval on the outer boundary of $\eta'([0,\ttime])$ which is immediately to the left of $\eta'(\ttime)$.  Let~$\ol{\tau}$ be the first time that $\ol{\eta}'$ hits $\eta'(\ttime)$ and let $F_\epsilon$ be the event that $\ol{\eta}'([0,\ol{\tau}]) \cap \eta'([0,\ttime])$ is contained in both the interval of $\partial U$ centered at $\eta'(\ttime)$ with quantum length $2\epsilon$ and the similarly defined interval on the outer boundary of $\ol{\eta}'([0,\ol{\tau}])$ with the roles of $\eta'$ and $\ol{\eta}'$ swapped.

The main two steps in the proof of Theorem~\ref{thm::typical_cut_point} are to show that (with $\MstwoW$ viewed as a measure on $(\CY,\ol{\CY},\CQ)$ and $\MstwoD$ viewed as a measure on $(\CX,\ol{\CX},\CB)$) we have
\begin{align}
 \epsilon^{-1} \one_{E_\epsilon} d \MstwoW &\to d \MstwoD \quad\text{as}\quad \epsilon \to 0 \quad\text{and} \label{eqn::m_w_to_m_d_e_eps}\\
 \epsilon^{-1} \one_{F_\epsilon} d \MstwoW &\to c_0 d \MstwoD \quad\text{as}\quad \epsilon \to 0, \label{eqn::m_w_to_m_d_f_eps}
\end{align}
where $c_0 > 0$ is a fixed constant.  We will prove~\eqref{eqn::m_w_to_m_d_e_eps} as a step in proving~\eqref{eqn::m_w_to_m_d_f_eps} and deduce~\eqref{eqn::m_w_to_m_d_f_eps} from~\eqref{eqn::m_w_to_m_d_e_eps} by showing that the conditional probability of $F_\epsilon$ given $E_\epsilon$ converges as $\epsilon \to 0$ to the positive constant $c_0$.  (The main step in proving this is Proposition~\ref{prop::local_picture_near_hit} stated and proved just below, which constructs the scaling limit near the intersection when we condition on $E_\epsilon$ and then send $\epsilon \to 0$.)  The notion of convergence is given by weak convergence with respect to a topology that we will introduce in the proof of Theorem~\ref{thm::typical_cut_point}.  (The exact choice of topology is not important as long as it generates the full $\sigma$-algebra.)  The result then follows because, for each $\epsilon > 0$, $\epsilon^{-1} \one_{F_\epsilon} d \MstwoW$ is invariant under the operation of swapping $\CY$ and $\ol{\CY}$ since $\MstwoW$ itself is invariant under this operation and $F_\epsilon$ is defined in a manner which is symmetric in $\eta'|_{[0,\ttime]}$ and $\ol{\eta}'|_{[0,\ol{\tau}]}$.  (Moreover, given $F_\epsilon$, we note that $\ol{\tau} \to \ol{\ttime}$ as $\epsilon \to 0$.)

Throughout the remainder of this section, we will write $\p_u$ for the probability which gives the conditional law of $\eta'$ and the remaining surface given $\ttime$ and $X_\ttime =u$ under either $\MstwoW$ or $\MstwoD$.  (As we pointed out in Proposition~\ref{prop::measure_properties}, this conditional law is the same under both $\MstwoW$ and $\MstwoD$ and the conditional distribution does not depend on the value of $\ttime$.)

\begin{lemma}
\label{lem::probability_of_hitting}
We have that 
\[ \p_u[ E_\epsilon] = \frac{\epsilon}{u} \wedge 1.\]
\end{lemma}
\begin{proof}
This follows because the location of $\eta'(\ttime)$ on $\partial \CY$ is uniform from the quantum measure given the quantum boundary length of $\partial \CY$.
\end{proof}

In our next lemma, we show that the conditional law of $(\CX,\ol{\CX},\CB)$ given $(\ttime,X_\ttime)$ sampled from $\MstwoW$ does not change when we further condition on $E_\epsilon$.  We emphasize, however, that the conditional law of the part of $\eta'$ which connects $\eta'(\ttime)$ and $\ol{\eta}'(\ol{\ttime})$ (the green path segment in Figure~\ref{fig::doubly_marked_sphere2}) given $E_\epsilon$ is not the same as its unconditioned law.

\begin{lemma}
\label{lem::law_given_e1}
Suppose that $\CX$, $\ol{\CX}$, and $\CB$ are as in the statement of Theorem~\ref{thm::typical_cut_point}.  The joint law of $\CX$, $\ol{\CX}$, and $\CB$ under $\Mstwo$ conditional on $(\ttime,X_\ttime)$ is equal to the joint law of $\CX$, $\ol{\CX}$, and $\CB$ under $\Mstwo$ conditional on $(\ttime, X_\ttime)$ and $E_\epsilon$.  The same holds with either $\MstwoW$ or $\MstwoD$ in place of $\Mstwo$.
\end{lemma}
\begin{proof}
Recall from the end of Section~\ref{subsec::levy_spheres} that, conditionally on $\CX$, the quantum surface parameterized by the component of $\CS \setminus \CX$ which contains $y$ is a quantum disk weighted by its quantum area decorated by the path $\ol{\eta}'|_{[0,\ol{\ttime}]}$.  Moreover, by the locality property for $\SLE_6$, we have that the conditional law of $\ol{\eta}'|_{[0,\ol{\ttime}]}$ is equal to that of a whole-plane $\SLE_6$ stopped at the first time that it hits $\CX$.  (Note that this path-decorated surface determines $\CB$ and $\ol{\CX}$.)  This implies that, given $\CX$ and this path decorated-surface, $\eta'(\ttime)$ is still uniformly distributed according to the quantum measure on its boundary.  Therefore conditioning further on $E_\epsilon$ does not change its conditional law.
\end{proof}

In our next lemma, we show that the local behavior of $X_\ttime$ near $\ttime$ is that of a $3/2$-stable L\'evy process with only downward jumps.

\begin{lemma}
\label{lem::surface_locally_like_wedge}
Suppose that $Y$ is given by the time-reversal of a $3/2$-stable L\'evy excursion with only upward jumps of length at least $t$ conditioned so that $Y_t = u$.  Let $\CZ$ be the Radon-Nikodym derivative between the law of $s \mapsto Y_{t+s}$ for $s \in [0,\epsilon]$ with respect to the law of a $3/2$-stable L\'evy process with only downward jumps in $[0,\epsilon]$ starting from~$u$.  Then we have that $\CZ \to 1$ in probability as $\epsilon \to 0$.
\end{lemma}
\begin{proof}
This follows, for example, from \cite[Lemma~3.19]{map_making}.
\end{proof}

We will now describe the local behavior of a surface sampled from $\p_u$ near $\eta'(\ttime)$ both conditioned on $E_\epsilon$ and unconditioned.  See Figure~\ref{fig::local_picture_near_hit} for an illustration of the setup.

\begin{figure}[ht!]
\begin{center}
\includegraphics[scale=0.85]{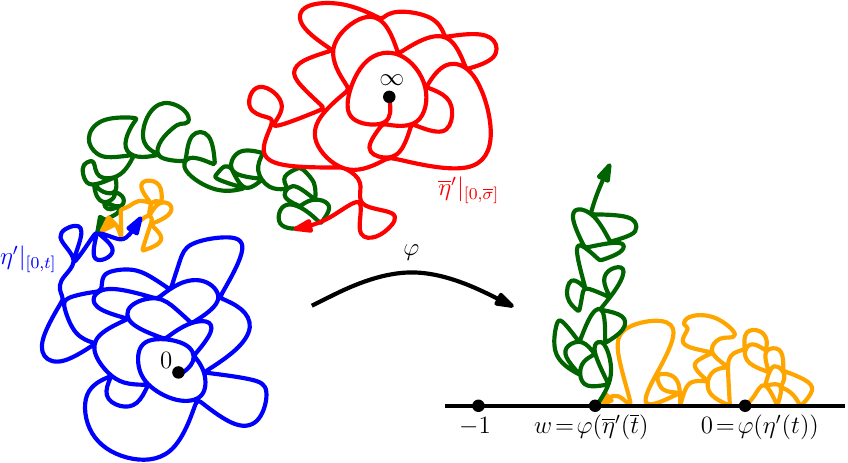}
\end{center}
\caption{\label{fig::local_picture_near_hit} Illustration of the setup of Proposition~\ref{prop::local_picture_near_hit}.  Shown on the left in blue is a path $\eta'$ on top of a doubly-marked quantum surface parameterized by $(\C,h)$ drawn up to time $t$ and in red the time-reversal $\ol{\eta}'$ of $\eta'$ up to a time $\ol{\sigma}$ before it hits $\eta'([0,t])$.  We show in Proposition~\ref{prop::local_picture_near_hit} that under $\p_u$ conditioned on $\ol{\eta}'|_{[0,\ol{\sigma}]}$ and on the event that $\ol{\eta}'$ first hits $\eta'(t)$ in the counterclockwise segment of the outer boundary of $\eta'([0,t])$ of length $\epsilon$ the local behavior of the surface near $\eta'(t)$ is described by a $\sqrt{8/3}$-quantum wedge and the local behavior $\eta'$ is described by an independent path which can be constructed by concatenating two $\SLE_6$-type curves.}
\end{figure}

\begin{proposition}
\label{prop::local_picture_near_hit}
Fix $u,\epsilon > 0$ and let $U$ be the unbounded component of $\C \setminus \eta'([0,t])$.  Let $\varphi_1 \colon U \to \h$ be the unique conformal map which takes~$\infty$ to~$i$ and $\eta'(t)$ to~$0$ and let $\varphi_2 \colon \h \to \h$ be the conformal map which corresponds to scaling so that with $\varphi = \varphi_2 \circ \varphi_1$ the quantum length assigned to $[-1,0]$ by the quantum boundary measure associated with the field $\wt{h} = h \circ \varphi^{-1} + Q\log|(\varphi^{-1})'| + \tfrac{2}{\gamma} \log \epsilon^{-1}$ is equal to $1$.  Let $\wt{\eta}' = \varphi(\eta')$.

Let $\ol{\sigma}$ be any stopping time for the filtration $\CF_s$ generated by the time-reversal of $\eta'$ and $\eta'|_{[0,t]}$, which a.s.\  occurs before the path first hits $\eta'([0,t])$.  We have that the joint law of the quantum surface $(\h,\wt{h})$ and path $\wt{\eta}'$ under $\p_u[\cdot \giv E_\epsilon, \CF_{\ol{\sigma}}]$ converges as $\epsilon \to 0$ to a pair consisting of a $\sqrt{8/3}$-quantum wedge with scaling factor chosen so that the quantum boundary length of $[-1,0]$ is equal to $1$ and an independent path $\wt{\eta}'$ whose law can be sampled from using the following steps.
\begin{itemize}
\item Sample $w \in [-1,0]$ according to Lebesgue measure
\item Sample an $\SLE_6(2;2)$ process $\wt{\eta}_w'$ from $w$ to $\infty$ with force points located at $w^-,w^+$; let $V_0$ be the component of $\h \setminus \wt{\eta}_w'$ with $0$ on its boundary
\item Given $\wt{\eta}_w'$, sample an $\SLE_6$ process $\wt{\eta}_0'$ in $V_0$ from $0$ to $w$
\item Take the concatenation of $\wt{\eta}_0'$ and $\wt{\eta}_w'$.
\end{itemize}
\end{proposition}
In the statement of Proposition~\ref{prop::local_picture_near_hit}, we emphasize that since $\tfrac{\kappa'}{2}-2$ is equal to $1$ for $\kappa'=6$, we have that $\wt{\eta}_w'$ a.s.\  does not hit $\R$ except at its starting point; see \cite[Remark~2.3]{MS_IMAG}.  Before we give the proof of Proposition~\ref{prop::local_picture_near_hit}, we will need to collect several intermediate results.

\begin{figure}[ht!]
\begin{center}
\includegraphics[scale=0.85]{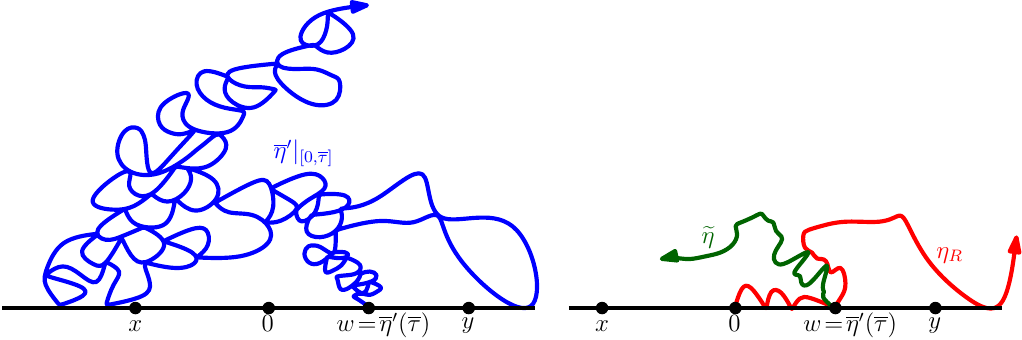}
\end{center}
\caption{\label{fig::sle_6_first_hit} Illustration of Lemma~\ref{lem::sle_6_first_hit} and the setup of its proof.  Shown is the time-reversal $\ol{\eta}'$ of an $\SLE_6$ process $\eta'$ in $\h$ from $0$ to $\infty$ stopped at the first time $\ol{\tau}$ that it hits a given bounded interval $[x,y]$ in $\R$ containing a neighborhood of $0$ on the event that $w = \ol{\eta}'(\ol{\tau}) \in [0,y]$.  We show in Lemma~\ref{lem::sle_6_first_hit} that if $\tau$ denotes the last time~$t$ that~$\eta'$ hits~$w$ then the law of $\delta^{-1}(\eta'(t)-w)$ for $t \in [\tau,\infty)$ converges as $\delta \to 0$ to an $\SLE_6(2;2)$ process.  The $\rho$ values of $(2;2)$ are special because they correspond to $\SLE_6$ conditioned not to hit $\partial \h$ in the sense described in \cite{MS_IMAG2}.  Shown on the right is an initial segment of the right boundary $\eta_R$ of $\ol{\eta}'$ as well as the left boundary $\wt{\eta}$ of $\ol{\eta}'|_{[0,\ol{\tau}]}$.  In particular, $\wt{\eta}$ and $\eta_R$ starting from the last time it hits $[0,y]$ together give the outer boundary of $\ol{\eta}'|_{[0,\ol{\tau}]}$.  If we view $\ol{\eta}'$ as the counterflow line of a GFF $h$ on $\h$ from $\infty$ to $0$ with boundary conditions $\lambda'-\pi \chi$ (resp.\ $-\lambda'+\pi \chi$)  on $\R_+$ (resp.\ $\R_-$), then $\eta_R$ is the flow line of $h$ starting from~$0$ of angle $-\pi/2$ and~$\wt{\eta}$ is the flow line of the GFF given by restricting~$h$ to the left component of $\h \setminus \eta_R$ starting from the rightmost intersection point of $\eta_R$ in $[x,y]$ with angle $\pi/2$.}
\end{figure}

\begin{lemma}
\label{lem::sle_6_first_hit}
Suppose that $\eta'$ is a chordal $\SLE_6$ process from $0$ to $\infty$ in $\h$ and let $\ol{\eta}'$ be its time-reversal.  Let $[x,y] \subseteq \R$ be an interval containing a neighborhood of $0$, let $\ol{\tau}$ be the first time $t$ that $\ol{\eta}'(t) \in [x,y]$, and let $w = \ol{\eta}'(\ol{\tau})$.  Let $\tau$ be the last time that $\eta'$ hits $w$ (so that $\eta'|_{[\tau,\infty)}$ is the time-reversal of $\ol{\eta}'|_{(0,\ol{\tau}]}$).  Then the law of $\delta^{-1} (\eta'|_{[\tau,\infty)}-w)$ reparameterized by capacity converges weakly as $\delta \to 0$ to that of an $\SLE_6(2;2)$ process in~$\h$ from~$0$ to~$\infty$ with force points located at~$0^-$ and~$0^+$ with respect to the topology of local uniform convergence.
\end{lemma}

See Figure~\ref{fig::sle_6_first_hit} for an illustration of the statement of Lemma~\ref{lem::sle_6_first_hit}.  The proof of Lemma~\ref{lem::sle_6_first_hit} will make use of the ideas developed in \cite{MS_IMAG,MS_IMAG2,MS_IMAG3,MS_IMAG4}.  We will not give an in-depth introduction to imaginary geometry here, but rather refer the reader to the introductions of these articles as well as to \cite[Section~2]{MW_INTERSECTIONS} for background.  For $\kappa \in (0,4)$ and $\kappa'=16/\kappa \in (4,8)$, we will make use of the following notation (which matches that used in \cite{MS_IMAG,MS_IMAG2,MS_IMAG3,MS_IMAG4}):
\begin{equation}
\label{eqn::ig_constants}
 \lambda = \frac{\pi}{\sqrt{\kappa}}, \quad \lambda' = \frac{\pi}{\sqrt{\kappa'}} = \lambda - \frac{\pi}{2}\chi,\quad \chi = \frac{2}{\sqrt{\kappa}} - \frac{\sqrt{\kappa}}{2}.
\end{equation}
We recommend that the reader reference \cite[Figure~2.5]{MW_INTERSECTIONS} while reading the proof of Lemma~\ref{lem::sle_6_first_hit}.

\begin{proof}[Proof of Lemma~\ref{lem::sle_6_first_hit}]
We let $\kappa=8/3$, $\kappa'=16/\kappa=6$, and will write $\lambda$, $\lambda'$, and $\chi$ for the constants in~\eqref{eqn::ig_constants} with these values of $\kappa$ and $\kappa'$.  By \cite[Theorem~1.1]{MS_IMAG}, we can view $\ol{\eta}'$ as the counterflow line from $\infty$ to $0$ of a GFF $h$ on $\h$ with boundary conditions given by $-\lambda' + \pi \chi = -\lambda + \tfrac{3}{2}\pi \chi$ (resp.\ $\lambda' - \pi \chi = \lambda - \tfrac{3}{2}\pi \chi$) on $\R_-$ (resp.\ $\R_+$).  Moreover, by \cite[Theorem~1.4]{MS_IMAG} we have that the left (resp.\ right) boundary of $\ol{\eta}'$ is given by the flow line $\eta_L$ (resp.\ $\eta_R$) of $h$ starting from $0$ with angle $\pi/2$ (resp.\ $-\pi/2$).  That is, $\eta_L$ (resp.\ $\eta_R$) is the flow line of $h + \pi\chi/2$ (resp.\ $h-\pi \chi/2$) from $0$ to $\infty$.  By \cite[Theorem~1.1]{MS_IMAG}, the law of $\eta_L$ (resp.\ $\eta_R$) is that of an $\SLE_\kappa(\kappa-4;\kappa/2-2) = \SLE_{8/3}(-4/3;-2/3)$ (resp.\ $\SLE_\kappa(\kappa/2-2;\kappa-4) =\SLE_{8/3}(-2/3;-4/3)$) process in $\h$ from $0$ to $\infty$ with force points located at~$0^-$ and~$0^+$.  (See also \cite[Figure~2.5]{MW_INTERSECTIONS}.)  In particular, $\eta_L$ a.s.\  hits~$\R_-$ but not~$\R_+$ and likewise $\eta_R$ a.s.\  hits~$\R_+$ and not~$\R_-$; recall that $\kappa/2-2$ is the critical $\rho$-value at or above which $\SLE_\kappa(\rho)$ does not hit the boundary.

Since $\ol{\eta}'$ visits the points on $\eta_L$ (resp.\ $\eta_R$) in the reverse order in which they are visited by $\eta_L$ (resp.\ $\eta_R$), we have that $w$ is either equal to the last intersection of $\eta_L$ with $[x,0]$ or to the last intersection of $\eta_R$ with $[0,y]$.  These two possibilities correspond to when $w \in [x,0]$ or $w \in [0,y]$.  We shall assume without loss of generality that we are working on the event that the latter holds.  Let $\wt{\eta}_L$ be the flow line with angle $\pi/2$ of the GFF given by restricting~$h$ to the component of $\h \setminus \eta_R$ which is to the left of~$\eta_R$.  It follows\footnote{\cite[Proposition~5.7]{dms2014mating} is stated in the setting of $\SLE_\kappa(\rho)$ processes with a single boundary force point; however, the same proof goes through verbatim to describe the local behavior of the start point an excursion which straddles a given boundary point for $\SLE_\kappa(\rho_1;\rho_2)$ processes.} from \cite[Proposition~5.7]{dms2014mating} that the two segments of $\eta_R$ (which correspond to before and after the path hits $w$ or, equivalently, before and after the path starts the excursion that it makes over $y$) translated by $-w$ and rescaled by the factor $\delta^{-1}$ converge as $\delta \to 0$ to a pair of flow lines $\eta_1,\eta_2$ of a GFF on $\h$ with boundary conditions $\lambda+\pi \chi/2 = \sqrt{2/3} \pi$ (resp.\ $-3\lambda+5\pi \chi/2 = -\sqrt{2/3}\pi$) on $\R_+$ (resp.\ $\R_-$) with respective angles $\theta_1=-\pi/2$, $\theta_2=2\lambda/\chi - \pi/2=5\pi/2$.  Moreover, $\wt{\eta}_L$ translated by $-w$ and rescaled by $\delta^{-1}$ converges along with the two segments of $\eta_R$ to the flow line $\wt{\eta}$ of the same limiting GFF used to generate $\eta_1,\eta_2$ with angle $\pi/2$ (since the angle gap between $\wt{\eta}_L$ and $\eta_R$ is equal to $\pi$, so the angle gap between $\wt{\eta}$ and $\eta_1$ is the same).  In particular,
\begin{itemize}
\item The marginal law of $\eta_1$ is that of an $\SLE_\kappa(\kappa-2) = \SLE_{8/3}(2/3)$ process with a single force point located at $0^-$.
\item The marginal law of $\eta_2$ is that of an $\SLE_\kappa(\kappa-4;2) = \SLE_{8/3}(-4/3;2)$ process with force points located at $0^-$, $0^+$.
\item The conditional law of $\eta_2$ given $\eta_1$ is that of an $\SLE_\kappa(\kappa-4) = \SLE_{8/3}(-4/3)$ process with force point located at $0^-$.
\item The conditional law of $\wt{\eta}$ given $\eta_1$, $\eta_2$ is that of an $\SLE_\kappa(\kappa/2-2;-\kappa/2) = \SLE_{8/3}(-2/3 ;-4/3)$ process with force points located immediately to the left and right of its starting point.
\end{itemize}

As $\kappa=8/3$ and $\kappa'=6$, we note that 
\begin{align*}
 \left(-3\lambda + \frac{5\pi \chi}{2}\right) - \pi \chi &= -3\lambda +\frac{3 \pi \chi}{2} = -3\lambda' \quad\text{and}\\
 \left(\lambda + \frac{\pi\chi}{2} \right) + \pi \chi &=\lambda + \frac{3\pi \chi}{2} = 3\lambda'.
 \end{align*}
Therefore by \cite[Theorem~1.4]{MS_IMAG}, we have that $\wt{\eta}$ and $\eta_1$ respectively give the left and right boundaries of the counterflow line from $\infty$ to $0$ which, by the form of the boundary data, is an $\SLE_6(2;2)$ process with force points located immediately to the left and right of $\infty$.  The result thus follows by the reversibility of $\SLE_6(2;2)$ established in \cite{MS_IMAG3}.
\end{proof}

\begin{lemma}
\label{lem::sle_6_first_hitting}
Suppose that $\eta'$ is a whole-plane $\SLE_6$ process from $0$ to $\infty$ in $\C$.  Let $\tau$ be a stopping time for $\eta'$ so that $\eta'([0,\tau])$ is a.s.\  bounded and let $\ol{\sigma}$ be a stopping time for the time-reversal $\ol{\eta}'$ of $\eta'$ given $\eta'|_{[0,\tau]}$ such that $\ol{\eta}'|_{[0,\ol{\sigma}]}$ does not hit $\eta'([0,\tau])$ a.s.  Let $\ol{\tau}$ be the first time that $\ol{\eta}'$ hits $\eta'([0,\tau])$.  Let $\psi$ be the unique conformal map from the unbounded component of $\C \setminus \eta'([0,\tau])$ to $\h$ which sends $\infty$ to $i$ and $\eta'(\tau)$ to $0$.  Then law of the time-reversal of $\delta^{-1} \psi(\ol{\eta}'|_{[0,\ol{\tau}]})$ parameterized by capacity converges weakly respect to the topology of local uniform convergence as $\delta \to 0$ to that of a chordal $\SLE_6(2;2)$ process from $0$ to $\infty$ with force points located at~$0^-$ and~$0^+$.
\end{lemma}
\begin{proof}
This follows by combining the locality property of whole-plane $\SLE_6$ with Lemma~\ref{lem::sle_6_first_hit}.
\end{proof}

\begin{lemma}
\label{lem::sle_6_harmonic_measure}
Suppose that $\eta'$ is a whole-plane $\SLE_6$ process from $0$ to $\infty$ in $\C$.  Let $\tau$ be a stopping time for $\eta'$ such that $\eta'([0,\tau])$ is bounded a.s.  Let $\ol{\eta}'$ be the time-reversal of $\eta'$ and let $\ol{\tau}$ be the first time that $\ol{\eta}'$ hits $\eta'([0,\tau])$.  Then $\ol{\eta}'(\ol{\tau})$ is distributed according to harmonic measure as seen from $\infty$ on the boundary of the unbounded component of $\C \setminus \eta'([0,\tau])$.
\end{lemma}
\begin{proof}
This is a consequence of the locality for whole-plane $\SLE_6$, the domain Markov property, and conformal invariance.
\end{proof}

\begin{proof}[Proof of Proposition~\ref{prop::local_picture_near_hit}]
We will first prove a version of the proposition in which we have not conditioned on $E_\epsilon$ or $\CF_{\ol{\sigma}}$.  That is, we will first argue that as $\epsilon \to 0$ (but $u > 0$ fixed), we have that the joint law of $(\h,\wt{h})$ and $\wt{\eta}'$ under $\p_u$ converges to a pair consisting of a $\sqrt{8/3}$-quantum wedge normalized to assign quantum boundary length $1$ to $[-1,0]$ and an independent chordal $\SLE_6$ process in $\h$ from $0$ to $\infty$.  

We will establish this using Lemma~\ref{lem::surface_locally_like_wedge}.  Suppose that $\wh{\CW} = (\h,\wh{h},0,\infty)$ is a $\sqrt{8/3}$-quantum wedge and that $\wh{\eta}'$ is a chordal $\SLE_6$ process from $0$ to $\infty$ sampled independently of $\wh{h}$.  Let $(\wh{f}_t)$ denote the forward centered Loewner flow associated with~$\wh{\eta}'$ and, for each $s$, we let $\wh{X}_s$ denote the change in the boundary length of $h \circ \wh{f}_s^{-1} + Q\log| (\wh{f}_s^{-1})'|$ relative to $s=0$.  Assume that $\wh{\eta}'$ is parameterized according to quantum natural time.  By \cite[Corollary~1.19]{dms2014mating}, we have that $\wh{X}_s$ evolves as a $3/2$-stable L\'evy process with only downward jumps.  Moreover, by \cite[Theorem~1.18]{dms2014mating}, we have that conditional on the realization of the entire process $\wh{X}_s$, the components of $\h \setminus \wh{\eta}'$ (viewed as quantum surfaces) are conditionally independent quantum disks given their boundary length and the boundary length of each such disk is given by the corresponding jump of $\wh{X}_s$.

We can construct a coupling of the surface near $\eta'(\ttime)$ and $\wh{\CW}$ as follows.  Let $X_s$ denote the quantum length of the outer boundary of $\eta'([0,\ttime+s])$ relative to the quantum length of the outer boundary of $\eta'([0,\ttime])$.  In other words, we normalize so that $X_0 = 0$.  Fix $\delta > 0$.  Lemma~\ref{lem::surface_locally_like_wedge} implies that there exists $s_0 > 0$ such that we can find a coupling of $X$ and $\wh{X}$ so that the event $A = \{ X|_{[0,s_0]} = \wh{X}|_{[0,s_0]} \}$ satisfies $\p_u[A] \geq 1-\delta$.  Note that the conditional law of the quantum disks cut off by $\eta'$ in the time interval $[0,s_0]$ given $X|_{[0,s_0]}$ is the same as that for the quantum disks cut off by $\wh{\eta}'$ in the time interval $[0,s_0]$ given $\wh{X}|_{[0,s_0]}$.  Consequently, we can couple together these quantum disks so that they are the same (equal as boundary-marked quantum surfaces and with the same orientation) on $A$.  We sample the remainder of the quantum disks conditionally independently.  On $A^c$, we sample all of the quantum disks conditionally independently given $X|_{[0,s_0]}$ and $\wh{X}|_{[0,s_0]}$.

Let $\wt{K}$ (resp.\ $\wh{K}$) be the region cut off from $\infty$ by $\wt{\eta}'([0,s_0])$ (resp.\ $\wh{\eta}'([0,s_0])$).  On $A$, we have that $(\wt{K},\wt{h},0)$ and $(\wh{K},\wh{h},0)$ are equal as (marked) quantum surfaces and $(\h \setminus \wt{K},\wt{h})$ and $(\h \setminus \wh{K},\wh{h})$ are conditionally independent as quantum surfaces given $X_{s_0} = \wh{X}_{s_0}$.  Let $\wt{V}$ (resp.\ $\wh{V}$) be the component of the interior of $\wt{K}$ (resp.\ $\wh{K}$) which contains $0$.  Then there exists a conformal map $\wt{\psi} \colon \wt{V} \to \wh{V}$ which takes $0$ to $0$ since $(\wt{V},\wt{h},0)$ and $(\wh{V},\wh{h},0)$ are equal as (marked) quantum surfaces.  By scaling, we can take the embedding of $(\h,\wh{h})$ so that $\wt{\psi}'(0) = 1$.  It therefore follows that $|\wt{\psi}(z)-z| \to 0$ as $z \to 0$, hence it is easy to see that the total variation distance between the laws of the restrictions of the two surfaces to $\h \cap B(0,\delta)$ converges to $0$ as $\delta \to 0$.

We have now proved the part of the proposition which involves the behavior of the surface near $\eta'(\ttime)$.  The same proof works if we condition on $E_\epsilon$ and $\CF_{\ol{\sigma}}$ since this conditioning does not change the local behavior of the surface near $0$.  We shall now assume that we are conditioning on both $E_\epsilon$ and $\CF_{\ol{\sigma}}$.  To finish the proof, we need to describe the behavior of~$\wt{\eta}'$ near~$\wt{\eta}'(\ttime)$.  Let $z_0 = \varphi(\infty)$ and let~$\ol{\tau}$ be the first time that $\ol{\eta}'$ hits $\eta'([0,\ttime])$.  Lemma~\ref{lem::sle_6_harmonic_measure} then implies that the distribution of $\varphi(\ol{\eta}'(\ol{\tau}))$ is equal to the distribution on $[-1,0]$ induced by harmonic measure as seen from $z_0$ (normalized to be a probability).  Since we know that $z_0 \to \infty$ as $\epsilon \to 0$, it follows that the distribution of $\varphi(\ol{\eta}'(\ol{\tau}))$ converges to the uniform distribution as $\epsilon \to 0$.  Moreover, that the law of the path is as claimed follows from Lemma~\ref{lem::sle_6_first_hitting}.
\end{proof}

\begin{lemma}
\label{lem::conditional_probability_converges}
There exists $p_0 > 0$ such that for each $u > 0$ we have that
\[ \p_u[F_\epsilon \giv E_\epsilon] \to p_0 \quad\text{as}\quad \epsilon \to 0.\]
\end{lemma}
\begin{proof}
This is obvious from the representation of the conditional law of the configuration near the tip described in Proposition~\ref{prop::local_picture_near_hit}.
\end{proof}

We are now ready to combine everything to complete the proof of Theorem~\ref{thm::typical_cut_point}.

\begin{proof}[Proof of Theorem~\ref{thm::typical_cut_point}]
As mentioned above, we will complete the proof of the theorem by establishing~\eqref{eqn::m_w_to_m_d_e_eps} and~\eqref{eqn::m_w_to_m_d_f_eps}.  We now introduce the topology respect to which we will establish the weak limits in~\eqref{eqn::m_w_to_m_d_e_eps} and~\eqref{eqn::m_w_to_m_d_f_eps}.  (As mentioned earlier, the exact topology is not important.)
 
We will first define a metric on the space $\X$ of finite volume quantum surfaces which are homeomorphic to~$\D$.  We can associate with each quantum surface $\CS = (D,h)$ a probability measure on random distributions $\wt{h}$ by picking $z \in D$ according to the quantum measure, $\theta \in [0,2\pi]$ uniformly at random and independently of $z$, letting $\varphi \colon D \to \D$ be the unique conformal transformation with $\varphi(z) = 0$ and $\varphi'(0)/|\varphi'(0)| =  e^{i \theta}$, and then taking $\wt{h} = h \circ \varphi^{-1} + Q \log|(\varphi^{-1})'|$.  Clearly, if $\CS_1 = (D_1,h_1)$ and $\CS_2 = (D_2,h_2)$ are equivalent as quantum surfaces then the corresponding random distributions $\wt{h}_1$ and $\wt{h}_2$ on $\D$ have the same law.  We define our metric $d$ on~$\X$ by taking $d(\CS_1,\CS_2)$ for $\CS_1,\CS_2 \in \X$ to be given by the sum over $j \in \N$ of~$2^{-j}$ times the Prokhorov distance between the laws of the restrictions of $\wt{h}_1$ and $\wt{h}_2$ to $B(0,1-2^{-j})$.  We can also extend the definition of $\X$ to the setting of $k$-marked quantum surfaces $\X_k$ which are homeomorphic to $\D$ by taking the sum over $j \in \N$ of $2^{-j}$ times the Prokhorov distance between the joint law of the restriction of $\wt{h}_1$ to $B(0,1-2^{-j})$ and its marked points and the joint law of the restriction of $\wt{h}_2$ to $B(0,1-2^{-j})$ and its marked points.

We next recall from \cite{dms2014mating,quantum_spheres} (as well as Section~\ref{sec::spheres}) that if we draw an independent whole-plane $\SLE_6$ on top of a $\sqrt{8/3}$-LQG sphere starting from one quantum typical point and targeted at another, then the quantum boundary length of the complementary component containing the target point is given by the time-reversal of a $3/2$-stable L\'evy excursion $e$.  Moreover, the jumps of $e$ correspond to the components cut out by the $\SLE_6$ where the magnitude of a given jump gives the quantum boundary length of the corresponding component.  Given $e$, these components are conditionally independent quantum disks which are oriented depending on whether the $\SLE_6$ surrounded it clockwise or counterclockwise and are also marked by the first (equivalently last) point on the disk boundary visited by the $\SLE_6$.  It is shown in particular in \cite[Theorem~1.1]{quantum_spheres} that the $\SLE_6$-decorated quantum sphere is a.s.\ determined by $e$ and the collection of oriented and marked quantum disks.  Therefore we can view the path-decorated, beaded quantum surface which corresponds to the region separated by an $\SLE_6$ from its target point as well as the path stopped at a given time as a random variable which takes values in $\X^*$ which is the product of the space of {\cadlag} functions with only downward jumps (with the local Skorokhod topology) and $(\X_1 \times \{0,1\})^\N$; we equip $\X^*$ with the product topology.  Indeed, the time-reversed L\'evy excursion naturally is an element of the former space and the ordered collection of marked, oriented quantum disks naturally take values in the latter space (the orientation corresponds to the extra bit).

Fix $\epsilon > 0$.  For any interval $I = (x_1,x_2)$ with $0 < x_1 < x_2$ and any interval $J = (t_1,t_2)$ with $0 < t_1 < t_2$ we have that $\MstwoW[X_\ttime \in I, \ttime \in J] \in (0,\infty)$ so that $\MstwoW$ conditioned on $X_\ttime \in I$ and $\ttime \in J$ makes sense as a probability measure.  We can thus write
\[ \MstwoW[ E_\epsilon, X_\ttime \in I, \ttime \in J ] = \MstwoW[ E_\epsilon \giv X_\ttime \in I, \ttime \in J] \MstwoW[ X_\ttime \in I, \ttime \in J].\]
Lemma~\ref{lem::probability_of_hitting} implies that for all $\epsilon \in (0,x_1)$ we have that
\begin{equation}
\label{eqn::rn_bound}
\frac{\epsilon}{x_2} \leq \MstwoW[ E_\epsilon \giv X_\ttime \in I, \ttime \in J] \leq \frac{\epsilon}{x_1}.
\end{equation}

We let
\[ G_j = \{ 2^{j-1} \leq \ttime \leq 2^{j}\} \quad\text{for each}\quad j \in \Z.\]

Suppose that $f \colon \X^* \times \X^* \times \X \to \R$ is a bounded, continuous function.  Let $\CX$, $\ol{\CX}$, and~$\CB$ be as in the statement of Theorem~\ref{thm::typical_cut_point}.  Lemma~\ref{lem::law_given_e1} implies that the $\MstwoW$ conditional laws of $\CX$, $\ol{\CX}$, and $\CB$ are the same if we condition on $X_\ttime$ or if we condition on both $X_\ttime$ and $E_\epsilon$.  It therefore follows from~\eqref{eqn::rn_bound} that for each $j \in \Z$ we have
\begin{equation}
\label{eqn::cm_to_cd}
\epsilon^{-1} \int f(\CX,\ol{\CX},\CB) \one_{E_\epsilon \cap G_j} d\MstwoW \to \int f(\CX,\ol{\CX},\CB) \one_{G_j} d \MstwoD \quad\text{as}\quad \epsilon \to 0.
\end{equation}

Let $\CY$, $\ol{\CY}$, and $\CQ$ be as in the statement of Proposition~\ref{prop::cw_properties} and let $\CQ^*$ be the component of $\CQ$ with the largest quantum area (there is a.s.\  a unique such component).  Then Proposition~\ref{prop::cw_properties} implies that $\MstwoW$ is invariant under the operation of swapping $\CY$ and $\ol{\CY}$.  Since the event $F_\epsilon$ is also defined in a way which is symmetric under swapping $\CY$ and $\ol{\CY}$, we have that the measure $\one_{F_\epsilon} d\MstwoW$ is symmetric in the same sense.  Thus to finish the proof it suffices to show that there exists a constant $c_0 > 0$ such that for any $f$ as above, 
\begin{equation}
\label{eqn::cm_to_cd_e2}
\epsilon^{-1} \int f(\CY,\ol{\CY},\CQ^*) \one_{F_\epsilon \cap G_j} d \MstwoW \to c_0 \int f(\CX,\ol{\CX},\CB) \one_{G_j} d\MstwoD \quad\text{as}\quad \epsilon \to 0.
\end{equation}
Note that we in fact have that $\CX = \CY$.  Moreover, it follows from Proposition~\ref{prop::local_picture_near_hit} that for each $\delta > 0$ and $j \in \Z$ we have that
\[ \Mstwo[ d(\ol{\CX},\ol{\CY}) \geq \delta \giv E_\epsilon, G_j ] \to 0 \quad\text{as} \quad \epsilon \to 0.\]
Lemma~\ref{lem::conditional_probability_converges} implies that the same is true when we condition on $F_\epsilon$ instead of $E_\epsilon$.  Therefore for $f$ bounded and continuous as above, we have that
\[ \epsilon^{-1} \int |f(\CY,\ol{\CY},\CQ^*) - f(\CX,\ol{\CX},\CB)| \one_{F_\epsilon \cap G_j} d \MstwoW \to 0 \quad\text{as}\quad \epsilon \to 0\]
for each $j \in \Z$.  Therefore it suffices to prove~\eqref{eqn::cm_to_cd_e2} with $\CX, \ol{\CX}, \CB$ in place of $\CY,\ol{\CY},\CQ^*$.  Assume that $I$, $J$ are as above.  Then we have that
\begin{align*}
    & \MstwoW[ F_\epsilon, X_\ttime \in I, \ttime \in J ]\\
 =& \MstwoW[ F_\epsilon \giv E_\epsilon, X_\ttime \in I, \ttime \in J ] \MstwoW[ E_\epsilon \giv X_\ttime \in I, \ttime \in J] \MstwoW[ X_\ttime \in I, \ttime \in J].
\end{align*}
By Lemma~\ref{lem::conditional_probability_converges}, we have that the first term on the right hand side converges to $p_0$ as $\epsilon \to 0$.  Therefore~\eqref{eqn::cm_to_cd_e2} follows from~\eqref{eqn::cm_to_cd}.
\end{proof}

As a consequence of the analysis which gave the proof of Theorem~\ref{thm::typical_cut_point}, we will also be able to identify the conditional law of $\CB$ under $\MstwoD$ given its boundary length.  This will be rather important for us later in this article because, as we will explain now, it implies that $\partial \CB$ is a.s.\  conformally removable.

\begin{proposition}
\label{prop::figure_8_quantum_disk_and_removable}
Suppose that $\CB$ is as in Theorem~\ref{thm::typical_cut_point}.
\begin{enumerate}[(i)]
\item Conditionally on its boundary length, the law of $\CB$ under $\MstwoD$ is that of a quantum disk.
\item Suppose that $(\C,h)$ is any embedding of $(\CS,x,y)$ distributed according to $\MstwoD$.  Then the image of the embedding of $\partial \CB$ is a.s.\  conformally removable.
\end{enumerate}
\end{proposition}
\begin{proof}
First of all, we know from Lemma~\ref{lem::law_given_e1} that the conditional law of $\CB$ given $\ttime$ and $X_\ttime$ is the same as the conditional law of $\CB$ given $\ttime$, $X_\ttime$, and $E_\epsilon$.  We define $d(\CS^1,\CS^2)$ as in the proof of Theorem~\ref{thm::typical_cut_point}.  The conditional law of the segment $\wt{\eta}'$ of $\eta'$ which is contained in $\CB$ and connects $\eta'(\ttime)$ to $\ol{\eta}'(\ol{\ttime})$ (i.e., the segment indicated in green in Figure~\ref{fig::doubly_marked_sphere2}) is given by a chordal $\SLE_6$ process independent of $\CB$.  Let $\wt{\CB}$ be the component of $\CB \setminus \wt{\eta}'$ with the longest quantum boundary length.  Then we know that $\wt{\CB}$ is a quantum disk conditional on its boundary length because it is a complementary component of $\eta'$ on $\CS$.  The result follows because the $d$-distance between $\wt{\CB}$ and $\CB$ tends to $0$ in probability as $\epsilon \to 0$ while, as we mentioned above, the law of $\CB$ does not change with $\epsilon$.  This proves that the conditional law of $\CB$ given its boundary length is a quantum disk.

Assume that $\CB$ is parameterized by $(\strip,\wt{h})$.  Then we know that the law of $\wt{h}$ is absolutely continuous with respect to a free boundary GFF on $[a,b] \times [0,2\pi]$ for any $a,b \in \R$ with $a < b$.  Likewise, we know that the law of $h$ in any bounded region $U \subseteq \C$ is mutually absolutely continuous with respect to the law of the restriction to $U$ of a free boundary GFF on a bounded domain with $\ol{U} \subseteq V$.  Therefore it follows from \cite[Theorem~8.1]{ms2013qle} that the map $\varphi \colon [a,b] \times [0,2\pi] \to \C$ which corresponds to the embedding of $\CB$ into $\C$ is a.s.\  H\"older continuous in $[a,b] \times [0,2\pi]$.  By \cite[Proposition~A.8]{dms2014mating}, we know that the points on $\partial \CB$ which correspond to $\pm \infty \in \partial \strip$ are uniformly distributed according to the quantum boundary measure on $\partial \CB$.  By resampling them, it follows that the embedding of $\CB$ into $\C$ is a.s.\  a H\"older domain.  Therefore it follows from \cite[Corollary~2]{js2000remove} that the embedding of $\partial \CB$ into~$\C$ is a.s.\  conformally removable.
\end{proof}

\section{Quantum natural time $\QLE(8/3,0)$}
\label{sec::qle_construction}

The purpose of this section is to construct a ``quantum natural time'' version of $\QLE(8/3,0)$.  This is a variant of the process constructed in \cite{ms2013qle} where the approximations involve resampling the tip of the $\SLE_6$ at $\delta$-units of quantum natural time as opposed to $\delta$-units of capacity time.  The construction that we will give here describes a growth process on a quantum sphere, which is also in contrast to \cite{ms2013qle} in which the process is constructed on an (unscaled) quantum cone.  The present construction also generalizes to the setting of quantum cones, but we will focus on the sphere case.  Throughout this article, whenever we refer to the process $\QLE(8/3,0)$ we mean the one constructed just below unless explicitly stated otherwise.  As in the case of the construction given in \cite{ms2013qle}, we will begin in Section~\ref{subsec::qle_approximations} by introducing the approximations to $\QLE(8/3,0)$.  We will then show in Section~\ref{subsec::qle_subsequential_limits} that the subsequential limits of these approximations have the following properties which will be important later on:
\begin{itemize}
\item The bubbles swallowed by the process have the same Poissonian structure as the bubbles swallowed by an $\SLE_6$,
\item The evolution of the quantum boundary length of the complementary component which contains the target point is the same as in the case of an $\SLE_6$ (the time-reversal of a $3/2$-stable L\'evy excursion with only positive jumps), and
\item The law of the region which contains the target point is the same as in the case of $\SLE_6$ (that of a quantum disk weighted by its area).
\end{itemize}

\subsection{Approximations to $\QLE(8/3,0)$}
\label{subsec::qle_approximations}

Fix $\delta > 0$ and suppose that $e \colon [0,T] \to \R_+$ is a sample picked from the excursion measure $\lexcursion$ for a $3/2$-stable L\'evy process with only positive jumps as described in Section~\ref{subsec::levy_spheres}.  We define the $\delta$-approximation to $\QLE(8/3,0)$ associated with the excursion $e$ as follows.  First, we let $(\CS,x,y)$ be a doubly-marked quantum sphere constructed from~$e$ and let $\eta_1' = \eta'$ be the associated whole-plane $\SLE_6$ from $x$ to $y$, as described in Section~\ref{subsec::levy_spheres}, with the quantum natural time parameterization.  For concreteness, we will take the embedding so that $(\CS,x,y)= (\C,h,0,\infty)$ where $0$ (resp.\ $\infty$) corresponds to $x$ (resp.\ $y$) and the scaling factor is determined so that the quantum area of $\D$ is equal to $1/2$ the total quantum area of~$\CS$.

We define a growth process $\qlegrowth^\delta$ inductively as follows.  First, we take $\qlegrowth_t^\delta$ to be the complement of the unbounded component of $\C \setminus \eta'([0,t])$ for each $t \in [0,\delta]$.  We also let $g_t^\delta \colon \C\setminus \qlegrowth_t^\delta \to \C\setminus \D$ be the unique conformal map with $|g_t^\delta(z)-z| \to 0$ as $z \to \infty$.  Fix $j \in \N$ and suppose that we have defined paths $\eta_1',\ldots,\eta_j'$ and a growing family of hulls $\qlegrowth^\delta$ with associated uniformizing conformal maps $(g_t^\delta)$ for $t \in [0,j\delta]$ such that the following hold:
\begin{itemize}
\item The conditional law of the surface parameterized by the complement of $\qlegrowth_{j\delta}^\delta$ given its quantum boundary length is the same as in the setting of ordinary $\SLE_6$.  That is, it is given by a quantum disk with the given boundary length weighted by its area.
\item $\eta_j'(j\delta)$ is distributed uniformly according to the quantum boundary measure on $\partial \qlegrowth_{j\delta}^\delta$ conditional on $\qlegrowth_{j \delta}^{\delta}$ (as a path-decorated beaded quantum surface).
\item The law of the components separated from the target point by time $j \delta$ is the same as in the case of whole-plane $\SLE_6$.  That is, they are given by conditionally independent quantum disks given their boundary lengths.
\end{itemize}
We then let $\eta_{j+1}'$ be an independent radial $\SLE_6$ starting from a point on $\partial \qlegrowth_{j\delta}^\delta$ which is chosen uniformly from the quantum boundary measure conditionally independently of everything else (i.e., we resample the location of the tip $\eta'(j \delta)$).  For each $t \in [j\delta,(j+1)\delta]$, we also let $\qlegrowth_t^\delta$ be the complement of the unbounded component of $\C\setminus (\qlegrowth_{j\delta}^\delta \cup \eta_{j+1}'([0,t]))$.  Then by the construction, all three properties described above are satisfied by the process up to time $(j+1)\delta$.

By resampling the surface parameterized by the complement of $\qlegrowth_{j\delta}^\delta$ at each stage, we can construct a coupling of the doubly marked quantum sphere together with the growth process $\qlegrowth^\delta$ so that the quantum boundary length of the complement of $\qlegrowth^\delta$ at each time $t$ is given by the time-reversal $e(T-t)$ of $e$.  Indeed, this follows because the procedure of resampling the starting point of the $\SLE_6$ according to quantum boundary length leaves the law of this component invariant.  In particular, the jumps of $e$ then correspond to the quantum boundary length of the quantum disks swallowed by~$\qlegrowth^\delta$.

It will be convenient to encode $\qlegrowth^\delta$ in terms of a radial Loewner flow.  That is, if for each $t \geq 0$ we let $s(t)$ be the quantum natural time elapsed by $\qlegrowth^\delta$ at the first time that the capacity as seen from $\infty$ reaches $t$ then there exists a measure $\nu_\delta$ on $\partial \D \times [0,\infty)$ whose second-coordinate marginal is given by Lebesgue measure such that
\[ g_{s(t)}^\delta(z) = z  + \iint_{\partial \D \times [0,t]} g_{s(u)}^\delta(z) \frac{w + g_{s(u)}^\delta(z)}{w - g_{s(u)}^\delta(z)} d \nu_\delta(w,u) \quad\text{for all}\quad t \geq 0.\]
Since the growth during each of the $\delta$-length time intervals is given by a segment of an independent radial $\SLE_6$ process, it follows that there exists a standard Brownian motion $B$ and a sequence of real numbers $(\xi_k)$ such that with $W = \sqrt{6} B$ we have that
\[ d\nu_\delta(\theta,t) = \sum_k \one_{[k \delta,(k+1)\delta)}(s(t)) \delta_{e^{i(W_t+ \xi_k)}} dt\]
where $\delta_x$ denotes the Dirac mass supported at $x \in \partial \D$ and $dt$ denotes Lebesgue measure on~$\R_+$.

We emphasize that the $\delta$-approximation to $\QLE(8/3,0)$ satisfies the following:
\begin{itemize}
\item The bubbles which it separates from~$y$ are conditionally independent quantum disks given their boundary lengths,
\item The boundary length of the complementary component which contains~$y$ at time~$t$ is equal to the time-reversal of $e$ at time $t$ and is conditionally independent of everything else given its boundary length, and
\item The conditional law of the region which contains~$y$ given its quantum boundary length is the same as in the case of $\SLE_6$.  That is, it is given by a quantum disk with the given boundary length weighted by its area.
\end{itemize}

\subsection{Subsequential limits}
\label{subsec::qle_subsequential_limits}

We are now going to construct subsequential limits of the $\delta$-approximations to $\QLE(8/3,0)$ described just above.  We will pick the subsequence so that the limit we obtain satisfies the three properties listed just above.

For each $\delta > 0$, we let $(\CS_\delta,x_\delta,y_\delta)$ be an instance of the doubly-marked quantum sphere decorated together with a $\delta$-approximation to $\QLE(8/3,0)$ which is encoded via a Loewner flow using the measure $\nu_\delta$, as described just above.  We take the embedding of $(\CS_\delta,x_\delta,y_\delta)$ to be given by $(\C,h_\delta,0,\infty)$, as in Section~\ref{subsec::qle_approximations}, where the scaling factor is chosen so that the quantum area of $\D$ is equal to $1/2$ the quantum area of $\CS$.  Since we will be taking a subsequential limit with respect to the weak topology (in a sense we will make more precise below), we will not specify how the surfaces $(\CS_\delta,x_\delta,y_\delta)$ decorated by $\nu_\delta$ are coupled together for different values of $\delta$.  For each $\delta,T > 0$, we let $\nu_\delta^T$ be the restriction of $\nu_\delta$ to $\partial \D \times [0,T]$.  Since $\partial \D \times [0,T]$ is compact and each $\nu_\delta^T$ has total mass equal to $T$, it follows that the law on random measures $(\nu_\delta^T)$ is tight in $\delta > 0$ (but $T$ fixed) with respect to the weak topology on measures on $\partial \D \times [0,T]$.

For each $0 < a_1 < a_2 < \infty$, we let $E_{a_1,a_2}^\delta$ be the event that the area of $\CS_\delta$ is contained in $[a_1,a_2]$.  We note that $E_{a_1,a_2}^\delta$ is an event of positive and finite $\Mstwo$ measure so that the conditional law of $(\C,h_\delta,0,\infty)$ given $E_{a_1,a_2}^\delta$ makes sense as a probability measure.  Since the law of~$h_\delta$ conditioned on $E_{a_1,a_2}^\delta$ does not depend on $\delta$, it follows that there exists a sequence $(\delta_k)$ with $\delta_k > 0$ for all $k$ and $\delta_k \to 0$ as $k \to \infty$ such that the joint law of $h_\delta$ and $\nu_{\delta_k}^T$ given $E_{a_1,a_2}^{\delta_k}$ converges weakly to a limiting law as $k \to \infty$.  (We take the notion of convergence of $h_\delta$ to be that its coordinates converge when expanded in terms of an orthonormal basis of $H(\C)$ consisting of $C_0^\infty(\C)$ functions.)  The Skorokhod representation theorem implies that we can find a coupling of the laws of $h_{\delta_k}$ and $\nu_{\delta_k}^T$ given $E_{a_1,a_2}^{\delta_k}$ (as $k$ varies) with a doubly-marked quantum sphere $(\C,h,0,\infty)$ conditioned on the event $E_{a_1,a_2}$ that its area is contained in $[a_1,a_2]$ together with a measure $\nu^T$ on $\partial \D \times [0,T]$ whose second coordinate marginal is given by Lebesgue measure such that $h_{\delta_k} \to h$ as described above and $\nu_{\delta_k}^T \to \nu^T$ weakly a.s.\  as $k \to \infty$.  By passing to a further (diagonal) subsequence, we can find a coupling of the laws of the $(\C,h_{\delta_k},0,\infty)$ (unconditioned) and $(\nu_{\delta_k})$ and $(\C,h,0,\infty)$ and a measure $\nu$ on $\partial \D \times [0,\infty)$ whose second coordinate marginal is given by Lebesgue measure so that with $\nu^T$ equal to the restriction of $\nu$ to $\partial \D\times [0,T]$ for each $T \geq 0$ we have that $h_{\delta_k} \to h$ and $\nu_{\delta_k}^T \to \nu^T$ weakly a.e.\ for all $T > 0$.  By \cite[Theorem~1.1]{ms2013qle}, it follows that the radial Loewner flows driven by the $\nu_{\delta_k}$ converge locally uniformly in time and space to the radial Loewner flow driven by $\nu$ as $k \to \infty$ as well.

\begin{definition}
\label{def::qle_def}
Suppose that $(\CS,x,y) = (\C,h,0,\infty)$ is a doubly-marked quantum surface and that $\qlegrowth$ is an increasing family of compact hulls with $K_0 = \{0\}$ so that the joint law of $(\CS,x,y)$ and $(K_t)$ is equal to any one of the subsequential limits constructed above.  Then we refer to the increasing family $\qlegrowth$ (modulo time parameterization) as $\QLE(8/3,0)$.  We say that two $\QLE(8/3,0)$-decorated quantum surfaces $(\CS,x,y)$, $\qlegrowth$ and $(\wt{\CS},\wt{x},\wt{y})$, $\wt{\qlegrowth}$ are equivalent if $(\CS,x,y)$ and $(\wt{\CS},\wt{x},\wt{y})$ are equivalent as doubly-marked quantum surfaces and the corresponding conformal transformation takes $\qlegrowth$ to $\wt{\qlegrowth}$ (modulo time parameterization).
\end{definition}

As explained in the introduction, it will be a consequence of the results of this work and \cite{qle_continuity} that it is not necessary to pass along a subsequence $(\delta_k)$ in the construction of $\QLE(8/3,0)$.

\begin{remark}
\label{rem::time_parameterizations}
As we will explain below, there are several natural parameterizations of time for $\QLE(8/3,0)$.  We emphasize that the definition of $\QLE(8/3,0)$ given in Definition~\ref{def::qle_def} does not carry with it a time parameterization.  We in particular do not build the capacity parameterization into the definition because it depends on the embedding of the ambient surface into $\C$.  Later, we will want to be able to cut a $\QLE(8/3,0)$ out of a surface and glue it into another and such an operation in general does not preserve capacity because it can lead to a different embedding.  The time-parameterizations that we introduce shortly are, however, intrinsic to $\QLE(8/3,0)$ viewed as a growing family of quantum surfaces.

Definition~\ref{def::qle_def} includes a notion of equivalence of $\QLE(8/3,0)$-decorated quantum surfaces.  We will later introduce a notion of equivalence for a $\QLE(8/3,0)$ which is stopped at a certain type of stopping time.
\end{remark}

Suppose that $\qlegrowth$ is a $\QLE(8/3,0)$ on a doubly-marked sphere $(\CS,x,y)$.  We say that a region $U$ is swallowed by $\qlegrowth$ if $U$ is equal to the interior of $\qlegrowth_t \setminus \qlegrowth_{t^-}$ for some $t$.  The following proposition will imply that we can make sense of the quantum natural time parameterization of $\qlegrowth$.

\begin{proposition}
\label{prop::qle_bubbles}
Suppose that $(\CS,x,y)$ is a doubly-marked quantum sphere with distribution $\Mstwo$ and that $\qlegrowth$ is a $\QLE(8/3,0)$ on $\CS$ from $x$ targeted at $y$.  Then the joint law of the regions swallowed by $\qlegrowth$, ordered by the time at which they are swallowed and viewed as quantum surfaces, is the same as for a whole-plane $\SLE_6$ on $\CS$ connecting~$x$ to~$y$.
\end{proposition}
\begin{proof}
We shall assume that we are in the setting described at the beginning of this subsection.  In particular, we take the embedding of $(\CS,x,y)$ to be given by $(\C,h,0,\infty)$.

For each $k \in \N$, we let $\qlegrowth^{\delta_k}$ be the growth process associated with $\nu_{\delta_k}$.  We take $\qlegrowth^{\delta_k}$ to be parameterized by capacity as seen from $\infty$.  By the construction of the $\delta_k$-approximation to $\QLE(8/3,0)$, we know the law of the regions swallowed by $\qlegrowth^{\delta_k}$ ordered by the time at which they are swallowed and viewed as quantum surfaces is the same as for a whole-plane $\SLE_6$ on $(\CS,x,y)$ connecting $x$ to~$y$.  Our aim is to show that this result holds in the limit as $k \to \infty$.

For each $k$, we let $\CU^{\delta_k}$ be the collection of surfaces swallowed by $\qlegrowth^{\delta_k}$ ordered by the time at which they are swallowed.  We note that for each $\epsilon > 0$, the collection $\CU_\epsilon^{\delta_k}$ which consists of those elements of $\CU^{\delta_k}$ with quantum boundary length at least $\epsilon$ is finite a.s.  Indeed, this follows because for $\lexcursion$ a.e.\ $e$ we have that the number of jumps made by $e$ of size at least $\epsilon$ is finite.  Since the law of each of the $\CU^{\delta_k}$ is the same for each $k$, by possibly passing to a further subsequence (and recoupling the laws using the Skorokhod representation theorem so that we have a.s.\ convergence) we have that each $|\CU_\epsilon^{\delta_k}|$ converges a.s.\  to a finite limit as $k \to \infty$.

For each $j$ we let $U_{j,\epsilon}^{\delta_k}$ be the $j$th element of $\CU_\epsilon^{\delta_k}$, $z_{j,\epsilon}^{\delta_k}$ be a point chosen uniformly at random from the quantum measure restricted to $U_{j,\epsilon}^{\delta_k}$, and let $\varphi_{j,\epsilon}^{\delta_k} \colon \D \to U_{j,\epsilon}^{\delta_k}$ be the unique conformal map with $\varphi_{j,\epsilon}^{\delta_k}(0) = z_{j,\epsilon}^{\delta_k}$ and $(\varphi_{j,\epsilon}^{\delta_k})'(0) > 0$.  (If $j > |\CU_\epsilon^{\delta_k}|$ then we take $\varphi_{j,\epsilon}^{\delta_k} \equiv 0$.)  Fix an orthonormal basis $(\phi_n)$ of $H(\D)$ consisting of $C_0^\infty(\D)$ functions.  Each element of $\CU_\epsilon^{\delta_k}$ is a quantum disk hence can be described by a distribution on $\D$ given by $h_{\delta_k} \circ \varphi_{j,\epsilon}^{\delta_k} + Q \log |(\varphi_{j,\epsilon}^{\delta_k})'|$.  We can express this distribution in terms of coordinates with respect to the orthonormal basis $(\phi_n)$.  Therefore by possibly passing to a further, diagonal subsequence (and recoupling the laws using the Skorokhod representation theorem so that we have a.s.\ convergence) we have for each fixed $j$ that the $j$th element of $\CU_\epsilon^{\delta_k}$ a.s.\  converges to a limiting distribution on $\D$ in the sense that each of its coordinates with respect to $(\phi_n)$ converge a.s.\  (if $j$ is larger than the number of elements of $\CU_\epsilon^{\delta_k}$ then we take the $j$th element to be the zero distribution on $\D$).  Combining, we have that $\CU_\epsilon^{\delta_k}$ converges a.s.\  to a limit $\CU_\epsilon$ which has the same law as each of the $\CU_\epsilon^{\delta_k}$ in the sense described just above (number of elements converges and each element converges weakly a.s.\ when parameterized by $\D$).

We note that the laws of the $(\varphi_{j,\epsilon}^{\delta_k})'(0)$ are tight as $k \to \infty$.  Indeed, if there was a uniformly positive chance that one of the $(\varphi_{j,\epsilon}^{\delta_k})'(0)$ is arbitrarily large for large $k$, then we would have that there is a uniformly positive chance that $(\CS^k,x^k,y^k) = (\C,h_{\delta_k},0,\infty)$ assigns a uniformly positive amount of area to $\C \setminus B(0,R)$ for each $R > 0$ and $k$ large enough.  This, in turn, would lead to the contradiction that the limiting surface $(\CS,x,y)$ would have an atom at $y$ with positive probability.  It therefore follows that by passing to a further subsequence if necessary (and recoupling the laws using the Skorokhod representation theorem so that we have a.s.\ convergence), we can arrange so that each of the conformal maps $\varphi_{j,\epsilon}^{\delta_k}$ converge locally uniformly to a limiting conformal map $\varphi_{j,\epsilon}$.  Since $(\CS,x,y)$ a.s.\  does not have atoms, it follows that each of the limiting $\varphi_{j,\epsilon}$ with $j \leq |\CU_\epsilon|$ satisfy $\varphi_{j,\epsilon}'(0) \neq 0$.

For each $j,\epsilon$, we let $U_{j,\epsilon} = \varphi_{j,\epsilon}(\D) \in \CU_\epsilon$.  Then we note that each of the $U_{j,\epsilon}$ is swallowed by $\Gamma$.  Indeed, let $\tau_{j,\epsilon}^{\delta_k}$ denote the time at which $U_{j,\epsilon}^{\delta_k}$ is swallowed by $\Gamma^{\delta_k}$.  Then by passing to a further subsequence if necessary (and recoupling the laws using the Skorokhod representation theorem so that we have a.s.\ convergence), we may assume that $\tau_{j,\epsilon}^{\delta_k}$ a.s.\ converges to a limit $\tau_{j,\epsilon}$ for all $j,\epsilon$.  It is not difficult to see that $\Gamma_{\tau_{j,\epsilon}}$ contains $U_{j,\epsilon}$ for each $j,\epsilon$ as $\Gamma_{\tau_{j,\epsilon}^{\delta_k}}^{\delta_k}$ contains $U_{j,\epsilon}^{\delta_k}$ and we have the local uniform convergence of $\varphi_{j,\epsilon}^{\delta_k}$ to $\varphi_{j,\epsilon}$.  Let $\sigma_{j,\epsilon} = \inf\{t \geq 0 : U_{j,\epsilon} \subseteq \Gamma_t\} \leq \tau_{j,\epsilon}$.  Note that if $\zeta > 0$, then $\sigma_{j,\epsilon} - \zeta$ is strictly smaller than $\tau_{j,\epsilon}^{\delta_k}-\zeta/2$ for all $k$ large enough.  Thus as $\Gamma_{\tau_{j,\epsilon}^{\delta_k}-\zeta/2}^{\delta_k}$ is disjoint from $U_{j,\epsilon}^{\delta_k}$ for all $k$ and $\Gamma_{\sigma_{j,\epsilon}-\zeta}$ does not contain $U_{j,\epsilon}$, it must be that $\Gamma_{\sigma_{j,\epsilon}-\zeta}$ is disjoint from $U_{j,\epsilon}$.  Since $\zeta > 0$ was arbitrary, it therefore follows that $U_{j,\epsilon}$ is contained in the interior of $\Gamma_{\sigma_{j,\epsilon}} \setminus \Gamma_{\sigma_{j,\epsilon}^-}$.

Since the $U_{j,\epsilon}$ as $j,\epsilon$ vary cover all of the quantum area, it follows that the regions swallowed by $\Gamma$ at the times $\sigma_{j,\epsilon}$ are made up entirely of the interior of the closure of the union of families of such sets (if there was an open ball which was contained in a region swallowed by $\Gamma$ which was disjoint from all of the $U_{j,\epsilon}$ we would have a contradiction since it would have to contain a positive amount of mass).  However, one could worry that there exists such a set $U_{j',\epsilon'}$ distinct from $U_{j,\epsilon}$ which is contained in the interior of $\Gamma_{\sigma_{j,\epsilon}} \setminus \Gamma_{\sigma_{j,\epsilon}^-}$.  Suppose that there is such a set; we will derive a contradiction which will imply that $U_{j,\epsilon}$ is precisely the set swallowed by $\Gamma$ at time $\sigma_{j,\epsilon}$.  Fix $\zeta > 0$.  Define times $u_i^{\delta_k}$ so that $u_0^{\delta_k} = 0$ and the quantum natural time elapsed by $\Gamma^{\delta_k}$ in $[u_i^{\delta_k},u_{i+1}^{\delta_k}]$ is equal to $\zeta$.  Let $i_0 \in \N_0$ be the largest such that $u_{i_0}^{\delta_k} \leq \sigma_{j,\epsilon}$ and let $u^{\delta_k} = u_{i_0-1}^{\delta_k}$, $v^{\delta_k} = u_{i_0+2}^{\delta_k}$.  On the event that $U_{j,\epsilon}$, $U_{j',\epsilon'}$ are both swallowed by $\Gamma$ at the time $\sigma_{j,\epsilon}$, we claim that the probability that $U_{j,\epsilon}^{\delta_k}$, $U_{j',\epsilon'}^{\delta_k}$ are not both swallowed in the interval $[u^{\delta_k},v^{\delta_k}]$ tends to $0$ as $k \to \infty$.  To see this, recall that the quantum surface parameterized by $\C \setminus \Gamma_{u_i^{\delta_k}}^{\delta_k}$ is a quantum disk weighted by its quantum area.  Therefore the probability that $h_{\delta_k} \circ (g_{u_i^{\delta_k}}^{\delta_k})^{-1} + Q\log| (g_{u_i^{\delta_k}}^{\delta_k})^{-1})'|$ assigns quantum area at least $a > 0$ to $B(0,1+b) \setminus \ol{\D}$ tends to $0$ as $b \to 0$ with $a > 0$ fixed.  By taking a union bound over $i \in \N_0$ (note that the number of $i \in \N_0$ such that $\C \setminus \Gamma_{u_i^{\delta_k}}^{\delta_k} \neq \emptyset$ is tight) implies that (uniformly in $k$) the probability that there exists an $i \in \N_0$ so that $h_{\delta_k} \circ (g_{u_i^{\delta_k}}^{\delta_k})^{-1} + Q\log| (g_{u_i^{\delta_k}}^{\delta_k})^{-1})'|$ assigns mass at least $a > 0$ to $B(0,1+b) \setminus \ol{\D}$ tends to $0$ as $b \to 0$ with $a > 0$ fixed.  This implies that the probability that $U_{j,\epsilon}^{\delta_k}$, $U_{j',\epsilon'}^{\delta_k}$ are not both swallowed in $[u^{\delta_k},v^{\delta_k}]$ tends to $0$ as $k \to \infty$ for otherwise it would be a positive probability event that for every $b > 0$ and $k \in \N$ large enough there exists $i \in \N_0$ and $a > 0$ so that $h_{\delta_k} \circ (g_{u_i^{\delta_k}}^{\delta_k})^{-1} + Q\log| (g_{u_i^{\delta_k}}^{\delta_k})^{-1})'|$ assigns mass at least $a > 0$ to $B(0,1+b) \setminus \ol{\D}$.  This, in turn, leads to the desired contradiction because as $\zeta \to 0$, the probability that the boundary length process for $\Gamma^{\delta_k}$ makes two macroscopic downward jumps in a (quantum natural time) interval of length $\zeta > 0$ tends to $0$.

Combining everything implies the result, as the law of~$\CU_\epsilon^{\delta_k}$ is the same as the law of each of the~$\CU_\epsilon$.
\end{proof}

\begin{remark}
\label{rem::qle_quantum_natural_defined}
Suppose that $Y_t$ is a $3/2$-stable L\'evy process with only positive jumps and that $\Lambda$ is a Poisson point process on $[0,t] \times \R_+$ with intensity measure $c ds \otimes u^{-5/2} du$ where $ds, du$ both denote Lebesgue measure on $\R_+$ and $c > 0$ is a constant.  Then there exists a value of $c > 0$ such that $\Lambda$ is equal in distribution to the set which consists of the pairs $(t,u)$ where $t$ is the time at which $Y$ makes a jump and $u$ is the size of the jump.  This implies that if we observe only the jumps made by $Y$ up to a random time $t$, then we can determine $t$ by counting the number of jumps that $Y$ has made with size between $e^{-j-1}$ and $e^{-j}$, dividing by the factor $c_0 e^{3 j/2}$ where $c_0 = \tfrac{2 c}{3}(e^{3/2}-1)$, and then sending $j \to \infty$.  Using the same principle, we can a.s.\  determine the length of time that the time-reversal of a $3/2$-stable L\'evy excursion $e \colon [0,T] \to \R_+$ has been run if we only observe its jumps.  Combining this with Proposition~\ref{prop::qle_bubbles}, this allows us to make sense of the quantum natural time parameterization of $\qlegrowth$ where the jumps are provided by the quantum boundary lengths of the bubbles cut off by $\qlegrowth$.
\end{remark}

Building on Remark~\ref{rem::qle_quantum_natural_defined}, we have that the ordered collection of components cut off by a $\QLE(8/3,0)$ from its target point a.s.\  determines a $3/2$-stable L\'evy excursion.  We will now show that the time-reversal of this excursion is a.s.\  equal to the process which gives the boundary length of $\qlegrowth$ when parameterized using the quantum natural time parameterization (in the same way that the ordered sequence of bubbles cut off by an $\SLE_6$ a.s.\  determines the evolution of the quantum length of its outer boundary) and that the conditional law of the region (viewed as a quantum surface) given its quantum boundary length which contains the target point is the same as for an $\SLE_6$ --- a quantum disk with the given boundary length weighted by its quantum area.

\begin{proposition}
\label{prop::qle_boundary_length}
Suppose that $(\CS,x,y)$ is a doubly-marked quantum sphere with distribution $\Mstwo$ and that $\qlegrowth$ is a $\QLE(8/3,0)$ on $\CS$ from $x$ to $y$.  We take $\qlegrowth$ to be parameterized by quantum natural time as described in Remark~\ref{rem::qle_quantum_natural_defined}.  For each fixed $t$, the conditional law of the (unique) complementary component of $\qlegrowth_t$ viewed as a quantum surface given its boundary length is equal to the conditional law of the complementary component containing $y$ of a whole-plane $\SLE_6$ on $\CS$ from $x$ to $y$ viewed as a quantum surface given its boundary length.  That is, it is given by a quantum disk with the given boundary length weighted by its area.  As $t$ varies, the quantum boundary length of the complementary component of~$\qlegrowth_t$ evolves in the same manner as for a whole-plane $\SLE_6$ (i.e., the time-reversal of a $3/2$-stable L\'evy excursion with only upward jumps) and is equal to the time-reversal of the L\'evy excursion whose ordered sequence of jumps is given by the boundary lengths of the components swallowed by $\qlegrowth$.
\end{proposition}
\begin{proof}
We shall assume that $(\CS,x,y) = (\C,h,0,\infty)$ and that $\qlegrowth$ is parameterized by capacity as in the construction of the subsequential limits in the beginning of this subsection.  Let $(\delta_k)$ be a sequence as in the beginning of this section and, for each $k$, let $\qlegrowth^{\delta_k}$ be the $\delta_k$-approximation to $\QLE(8/3,0)$.  We assume that each of the $\qlegrowth^{\delta_k}$ are parameterized by capacity.  For each $\epsilon > 0$, we let $\tau_{j,\epsilon}^{\delta_k}$ and $\tau_{j,\epsilon}$ be as in the proof of Proposition~\ref{prop::qle_bubbles} (i.e., passing to a further subsequence and recoupling the laws of the $\delta_k$-approximations using the Skorokhod representation theorem).  Then we know that the law of the quantum surface parameterized by $\C \setminus \qlegrowth_{\tau_{j,\epsilon}^{\delta_k}}^{\delta_k}$ is equal in distribution to the corresponding surface in the case of an $\SLE_6$ exploration of a doubly-marked quantum sphere sampled from $\Mstwo$.  Repeating the argument of Proposition~\ref{prop::qle_bubbles} (i.e., passing to a further subsequence and recoupling the laws of the $\delta_k$-approximations using the Skorokhod representation theorem), we can construct a coupling with an $\SLE_6$ exploration of a doubly-marked quantum sphere so that the unexplored region for the $\QLE(8/3,0)$ at the time $\tau_{j,\epsilon}$ is equal to the unexplored region at the corresponding time for an $\SLE_6$.  Moreover, we can arrange so that the ordered sequence of bubbles cut off by $\qlegrowth$ by the time $\tau_{j,\epsilon}$ is equal to the ordered sequence of bubbles cut off by the $\SLE_6$ by the corresponding time.  Remark~\ref{rem::qle_quantum_natural_defined} implies that the boundary length of the unexplored region for the $\SLE_6$ at the time corresponding to $\tau_{j,\epsilon}$ is determined by the ordered sequence of bubbles cut off by this time.  By our choice of coupling, the same is also true for the $\QLE(8/3,0)$.  This implies the result because the times of the form $\tau_{j,\epsilon}$ are dense in the quantum natural time parameterization.
\end{proof}

\begin{proposition}
\label{prop::qle_outer_boundary}
Suppose that we have the same setup as in Proposition~\ref{prop::qle_boundary_length} where we have taken $(\CS,x,y) = (\C,h,0,\infty)$ with $\qlegrowth$ parameterized by quantum natural time.  Then for each time $t$, we have that $\C \setminus \qlegrowth_t$ is a H\"older domain for each $t > 0$ fixed.  In particular, $\partial \qlegrowth_t$ is a.s.\  conformally removable.
\end{proposition}
\begin{proof}
This follows because the law of the surface parameterized by $\C \setminus \qlegrowth_t$ has the same law as the corresponding surface for an $\SLE_6$, hence we can apply the argument of Proposition~\ref{prop::figure_8_quantum_disk_and_removable}.
\end{proof}

\subsection{Time parameterizations}
\label{subsec::qle_tim}

In this article, we consider three different time parameterizations for $\QLE(8/3,0)$:
\begin{enumerate}
\item Capacity time,
\item Quantum natural time, and
\item Quantum distance.
\end{enumerate}
We have already introduced the first two parameterizations.  We will now give the definition of the third.  Suppose that $\qlegrowth$ is a $\QLE(8/3,0)$ parameterized by quantum natural time and let $X_t$ be the quantum boundary length of the complementary component of $\qlegrowth_t$.  We then set
\[ s(t) = \int_0^t \frac{1}{X_u} du.\]
We refer to the time-parameterization given by $s(t)$ as the {\bf quantum distance} parameterization for $\QLE(8/3,0)$.  The reason for this terminology is that, as we shall see in the next section, those points which are swallowed by $\qlegrowth$ at quantum distance time $d$ have distance $d$ from $x$ in the metric space that we construct.  We also note that this time parameterization is particularly natural from the perspective that $\QLE(8/3,0)$ represents a continuum form of the Eden model.  Recall that quantum natural time for $\SLE_6$ is the continuum analog of the time parameterization for percolation in which each edge is traversed in each unit of time.  The quantum distance time therefore corresponds to the time parameterization in which the number of edges traversed in a unit of time is proportional to the boundary length of the cluster.  This is the usual time parameterization for the Eden model.  Let $\zeta = \inf\{t > 0: X_t = 0\}$ and let
\begin{equation}
\label{eqn::dl_def}
\qdist(x,y) = \int_0^\zeta \frac{1}{X_u} du = s(\zeta).
\end{equation}
Then $\qdist(x,y)$ measures the quantum distance ``from the left'' of $x$ and $y$ (in the sense that $x$ is the ``left'' argument of $\qdist(x,y)$ and $y$ is the ``right'' argument).  One of the main inputs into the proof of Theorem~\ref{thm::qle_metric} given in Section~\ref{subsec::coupling_explorations} is that $\qdist(x,y)$ is symmetric in $x$ and $y$ and in fact a.s.\  determined by the underlying surface $\CS$.

\begin{lemma}
\label{lem::capacity_increases_qd}
Suppose that~$\qlegrowth$ is a $\QLE(8/3,0)$ with either the quantum distance or quantum natural time parameterization.  We a.s.\ have for all $0 \leq s < t$ that~$\qlegrowth_s$ is a strict subset of~$\qlegrowth_t$.
\end{lemma}
\begin{proof}
This follows because it is a.s.\ the case that in each open interval of time in the quantum natural time parameterization $\qlegrowth$ swallows a bubble and the same is also true for the quantum distance parameterization.
\end{proof}

\begin{remark}
\label{rem::capacity_vs_quantum_distance}
Lemma~\ref{lem::capacity_increases_qd} implies that the function which converts from the quantum natural time (resp.\ distance) parameterization to the capacity time parameterization is a.s.\  strictly increasing.  It does not, however, rule out the possibility that this time change has jumps.
\end{remark}

\section{Symmetry}
\label{sec::qle_symmetry}

The purpose of this section is to establish an analog of Theorem~\ref{thm::typical_cut_point} for $\QLE(8/3,0)$, stated as Theorem~\ref{thm::qle_symmetry} below.  This will be a critical ingredient for our proof of Theorem~\ref{thm::qle_metric} given in Section~\ref{sec::metric}.  We begin by introducing a $\QLE(8/3,0)$ analog of the measure $\MstwoD$ from Section~\ref{subsec::weighted_measures}.  To define this measure, we let $(\delta_k)$ and $(\ol{\delta}_k)$ be two sequences as in Section~\ref{subsec::qle_subsequential_limits} along which the $\delta$-approximations to $\QLE(8/3,0)$ converge.  We do not assume that $(\delta_k)$ and $(\ol{\delta}_k)$ are the same.  We then let $\lawonall$ be the measure on quadruples consisting of a doubly-marked quantum sphere $(\CS,x,y)$, $\QLE(8/3,0)$ growth processes $\qlegrowthleft$ and $\qlegrowthright$ from $x$ to $y$ (resp.\ $y$ to $x$), and $U \in [0,1]$.  We assume that $\qlegrowthleft$ (resp.\ $\qlegrowthright$) is produced from the limiting law associated with $(\delta_k)$ (resp.\ $(\ol{\delta}_k)$).  A sample from $\lawonall$ is produced by:
\begin{enumerate}
\item Picking $(\CS,x,y)$ from $\Mstwo$,
\item Picking $\qlegrowthleft$ and $\qlegrowthright$ conditionally independently given $(\CS,x,y)$, and
\item Taking $U$ to be uniform in $[0,1]$ independently of everything else.
\end{enumerate}
We assume that both $\qlegrowthleft$ and $\qlegrowthright$ have the quantum distance parameterization.  Let $\qdist(x,y)$ be as in~\eqref{eqn::dl_def} for $\qlegrowthleft$ and $\oqdist(y,x)$ be as in~\eqref{eqn::dl_def} for $\qlegrowthright$.  We note that it is not \emph{a priori} clear that the quantities $\qdist(x,y)$ and $\oqdist(y,x)$ should be related because the former is a measurement of distance between $x$ and $y$ from the ``left'' while the latter measures distance from the ``right.''  Moreover, $\qdist(x,y)$ and $\oqdist(y,x)$ are defined with what \emph{a priori} might be different subsequential limits of $\QLE(8/3,0)$.

We also let $\stoppingleft = U \qdist(x,y)$, $\stoppingright = U \oqdist(y,x)$,
\[ \righthit = \inf\{t \geq 0 : \qlegrowthright_t \cap \qlegrowthleft_\stoppingleft \neq \emptyset\}, \quad\text{and}\quad \lefthit = \inf\{t \geq 0 : \qlegrowthleft_t \cap \qlegrowthright_\stoppingright \neq \emptyset\}.\]
Note that $\stoppingleft$ (resp.\ $\stoppingright$) is uniform in $[0,\qdist(x,y)]$ (resp.\ $[0,\oqdist(y,x)]$).  We define measures $\lawonall^{x\to y}$ and $\lawonall^{y \to x}$ by setting
\begin{equation}
\label{eqn::dm_d_q_def}
\frac{ d \lawonall^{x \to y}}{d \lawonall} = \qdist(x,y) \quad\text{and}\quad \frac{ d \lawonall^{y \to x}}{d \lawonall} = \oqdist(y,x).
\end{equation}

\begin{figure}[ht!]
\begin{center}
\includegraphics[scale=0.85, page=3]{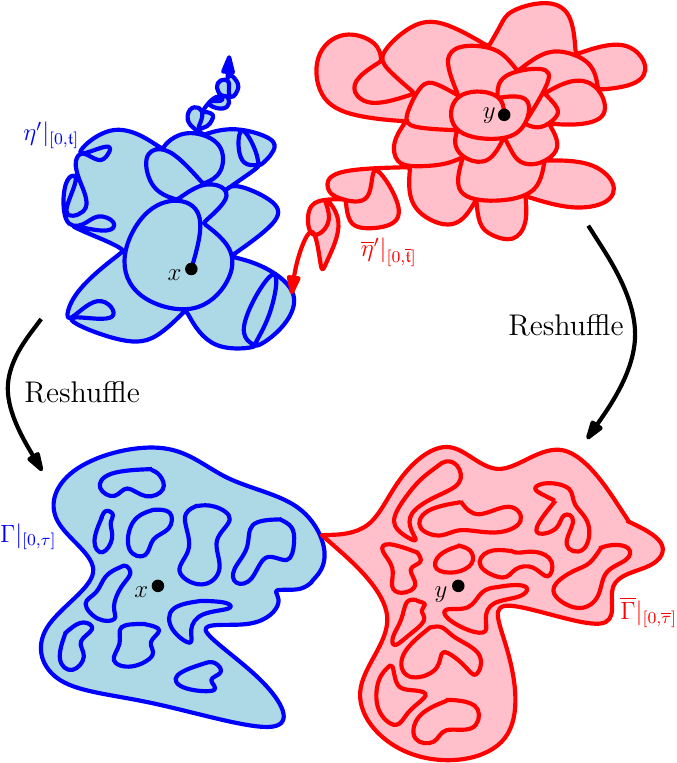}
\end{center}
\caption{\label{fig::qle_symmetry} Illustration of the argument used to prove Theorem~\ref{thm::qle_symmetry}.  The top shows the configuration of paths as in Figure~\ref{fig::sle_6_first_hit}.  By Theorem~\ref{thm::typical_cut_point}, we know that the distribution of the path-decorated quantum surfaces parameterized by the light blue and red regions on the top and marked by the tip of the corresponding $\SLE_6$ segments is symmetric under $\MstwoD$.  Reshuffling the two $\SLE_6$ processes on the top yields the pair of $\QLE(8/3,0)$ growths on the bottom.  We will deduce our symmetry result for the surfaces on the bottom from the symmetry of the surfaces on the top.  Part of this involves arguing that the asymmetry in the top which arises from knowing that $\ol{\eta}'(\ol{\ttime}) \in \eta'([0,\ttime])$ disappears in the limiting procedure used to construct $\QLE(8/3,0)$.}
\end{figure}

\begin{theorem}
\label{thm::qle_symmetry}
Suppose that we have the setup described just above.  Then the $\lawonall^{x \to y}$ distribution of $\qlegrowthleft|_{[0,\stoppingleft]}$, $\qlegrowthright|_{[0,\righthit]}$,$(\CS,x,y)$ is equal to the $\lawonall^{y \to x}$ distribution of $\qlegrowthleft|_{[0,\lefthit]}$, $\qlegrowthright|_{[0,\stoppingright]}$, $(\CS,x,y)$.
\end{theorem}

\subsection{Conditional independence in the limit}

We will need to have a version of the following statement in the proof of Theorem~\ref{thm::qle_symmetry}.  Suppose that we start with a doubly-marked quantum sphere $(\CS,x,y)$ together with a $\delta_k$-approximation of $\QLE(8/3,0)$ from $x$ to $y$ and a $\ol{\delta}_k$-approximation of $\QLE(8/3,0)$ from $y$ to $x$, taken to be conditionally independent given~$(\CS,x,y)$.  Then in the weak limit along sequences $(\delta_k)$ and $(\ol{\delta}_k)$ as in Section~\ref{subsec::qle_subsequential_limits} we have that the resulting $\QLE(8/3,0)$'s are conditionally independent given the underlying quantum surface, at least when the processes are grown up to a time at or before they first touch.

\begin{proposition}
\label{prop::qle_conditional_independence_in_limit_stopping_times}
Suppose that $(\delta_k)$, $(\ol{\delta}_k)$ are two sequences as in Section~\ref{subsec::qle_subsequential_limits}.  Suppose further that, for each $k$, we have a triple consisting of a doubly-marked quantum sphere $(\CS^k,x^k,y^k) = (\C,h^k,0,\infty)$, a $\delta_k$-approximation $\qlegrowthleft^k$ of $\QLE(8/3,0)$ from $0$ to $\infty$, and a $\ol{\delta}_k$-approximation $\qlegrowthright^k$ of $\QLE(8/3,0)$ from $\infty$ to $0$.  We assume further that $\qlegrowthleft^k$ and $\qlegrowthright^k$ are conditionally independent given $h^k$ and that both are parameterized by capacity as seen from their target point.  Let $(\CS,x,y) = (\C,h,0,\infty)$, $\qlegrowthleft$, $\qlegrowthright$ be distributed as the weak limit of the aforementioned triple (so both of the limiting $\QLE(8/3,0)$'s are parameterized by capacity as seen from their target point and the type of limit is as in Section~\ref{subsec::qle_subsequential_limits}).  Let $\stoppingleft$ be any stopping time for the filtration generated by $\qlegrowthleft$ and the restriction of $h$ to the interior of $\qlegrowthleft_t$ so that $\qlegrowth_{\stoppingleft}$ is a.s.\  bounded and let $\righthit$ be the first time that $\qlegrowthright$ hits $\qlegrowthleft_{\stoppingleft}$.  Then the joint distribution of $(\C,h,0,\infty)$, $\qlegrowthleft|_{[0,\stoppingleft]}$, and $\qlegrowthright|_{[0,\righthit]}$ is equal to the corresponding distribution if we had taken $\qlegrowthleft$ and $\qlegrowthright$ to be $\QLE(8/3,0)$ processes generated respectively from the limiting law associated with $(\delta_k)$ and $(\ol{\delta}_k)$ sampled conditionally independently given $h$.
\end{proposition}

In order to establish Proposition~\ref{prop::qle_conditional_independence_in_limit_stopping_times}, we will need to control the Radon-Nikodym derivative of the conditional law of the approximation of one of the $\QLE(8/3,0)$'s given the other.  It is not hard to find a counterexample which shows that if a sequence $(X_k,Y_k,Z_k)$ converges weakly to $(X,Y,Z)$ such that $X_k$, $Y_k$ are conditionally independent given $Z_k$ for each $k$ then it is not necessarily true that $X$, $Y$ are conditionally independent given $Z$.  We begin with Lemma~\ref{lem::cond_ind_limit} which gives a condition under which the conditional independence of $X_k,Y_k$ given $Z_k$ implies the conditional independence of $X,Y$ given $Z$.  We will then use this condition in Lemma~\ref{lem::qle_conditional_independence_in_limit_neighborhoods} to get the conditional independence of the limiting $\QLE(8/3,0)$'s stopped upon exiting a pair of disjoint open sets from which Proposition~\ref{prop::qle_conditional_independence_in_limit_stopping_times} will follow.

\begin{lemma}
\label{lem::cond_ind_limit}
Suppose that $(X_k,Y_k,Z_k)$ is a sequence of random variables such that $X_k$ and $Y_k$ are conditionally independent given $Z_k$ for each $k$ and $(X_k,Y_k,Z_k) \to (X,Y,Z)$ weakly as $k \to \infty$.  For each $k$, let $f_k$ (resp.\ $g_k$) be the Radon-Nikodym derivative of the conditional law of $X_k$ (resp.\ $Y_k$) given $Z_k$ with respect to its unconditioned law.  For each $k$, we let $m_k^X, m_k^Y,m_k^Z$ respectively denote the laws of $X_k,Y_k,Z_k$.  Assume that $f_k = f$ and $g_k = g$ do not depend on $k$, that $f,g$ are continuous, and that $f(X_k,Z_k) g(Y_k,Z_k)$ is uniformly integrable for $ dm_k^X d m_k^Y d m_k^Z$.  Then $X,Y$ are conditionally independent given~$Z$.
\end{lemma}
\begin{proof}
For simplicity we will prove the result in the case that $f,g$ are bounded.  The result in the general case follows from a simple truncation argument.  Let $m^X,m^Y,m^Z$ be the respective laws of $X,Y,Z$.  Then the joint law of $(X_k,Y_k,Z_k)$ is given by
\[ f(X_k,Z_k) g(Y_k,Z_k) dm_k^X dm_k^Y dm_k^Z.\]
Suppose that $F$ is a bounded, continuous function of $(X_k,Y_k,Z_k)$.  Then the weak convergence of $(X_k,Y_k,Z_k)$ to $(X,Y,Z)$ as $k \to \infty$ implies that
\begin{align*}
 &\int F(X_k,Y_k,Z_k) f(X_k,Z_k) g(Y_k,Z_k) dm_k^X dm_k^Y dm_k^Z\\
 \to& \int F(X,Y,Z) f(X,Z) g(Y,Z) dm^X dm^Y dm^Z \quad\text{as}\quad k \to \infty.
 \end{align*}
Therefore $f(X,Z) g(Y,Z) dm^X dm^Y dm^Z$ gives the joint law of $(X,Y,Z)$.  From the form of the joint law, it is clear that $X,Y$ are conditionally independent given $Z$, as desired.
\end{proof}

In order to be able to apply Lemma~\ref{lem::cond_ind_limit}, we need to give the Radon-Nikodym derivative of the conditional law of a quantum sphere in a region conditional on its realization in a disjoint region.  We will state this result in the setting of the infinite measure $\bsphere$ on doubly-marked quantum spheres described in~\cite{dms2014mating, quantum_spheres} constructed using the excursion measure associated with a Bessel process because this will allow us to give the simplest formulation of the result.  We recall \cite[Theorem~1.4]{quantum_spheres} which states that there exists a constant $c_{\rm LB} > 0$ such that $\Mstwo = c_{\rm LB} \bsphere$.  It will be convenient in what follows to produce samples of quantum spheres using $\bsphere$ rather than $\Mstwo$ because the Bessel description is more amenable to using the Markov property for the GFF.  We will briefly recall this construction; see the introductions of \cite{quantum_spheres,dms2014mating} as well as \cite[Section~4.5, Appendix~A]{dms2014mating} for additional detail.  We let $\cyl = \R \times [0,2\pi]$ with the top and bottom identified be the infinite cylinder.  We then let $\CH(\cyl)$ be the Hilbert space closure of $C_0^\infty(\cyl)$ with respect to the Dirichlet inner product
\begin{equation}
\label{eqn::dirichlet_inner_product}
(f,g)_\nabla = \frac{1}{2\pi} \int \nabla f(x) \cdot \nabla g(x) dx
\end{equation}
and we let $\CH_1(\cyl)$ (resp.\ $\CH_2(\cyl)$) be the subspace of $\CH(\cyl)$ consisting of those functions which are constant (resp.\ have mean zero) on vertical lines.  Then $\CH_1(\cyl) \oplus \CH_2(\cyl)$ gives an orthogonal decomposition of $\CH(\cyl)$ \cite[Lemma~4.2]{dms2014mating}.  A sample can be produced from $\bsphere$ as follows.
\begin{itemize}
\item Take the projection of $h$ onto $\CH_1(\cyl)$ to be given by $\tfrac{2}{\gamma} \log Z$ reparameterized to have quadratic variation $du$ where $Z$ is picked from the It\^o excursion measure $\nu_\delta^\bes$ of a Bessel process of dimension $\delta = 4-\tfrac{8}{\gamma^2}$.  This defines $Z$ modulo horizontal translation.
\item Sample the projection onto $\CH_2(\cyl)$ independently from the law of the corresponding projection of a whole-plane GFF on $\cyl$.
\end{itemize}
(See \cite[Section~2.1.1]{quantum_spheres} for a reminder of the construction of $\nu_\delta^\bes$.)  Throughout, we let $\cyl_\pm \subseteq \cyl$ be the half-infinite cylinders given by $[0,2\pi] \times \R_\pm$ with the top and bottom identified.

Fix $r \in \R$ and suppose that $(\cyl,h,-\infty,+\infty)$ is sampled from $\bsphere$ conditioned on the supremum of the projection of $h$ onto $\CH_1(\cyl)$ exceeding $r$ (this defines a probability measure).  We take the horizontal translation so that the projection of~$h$ onto $\CH_1(\cyl)$ first hits $r$ at $u=0$.  We note that the projection of $h$ onto $\CH_1(\cyl)$ takes a simple form for $u \geq 0$: it is given by the function whose common value on $u + [0,2\pi i]$ is equal to $B_u + (\gamma-Q) u$ where $B$ is a standard Brownian motion with $B_0 = r$.

Suppose that $\wh{h}$ is a whole-plane GFF on $\cyl$ plus the function $(\gamma-Q) \re(z)$.  We assume that the additive constant for $\wh{h}$ is fixed so that its average on $[0, 2\pi i]$ is equal to $r$.  Then the restriction to $\cyl_+$ of the projection of $\wh{h}$ onto $\CH_1(\cyl)$ has the same law as for $h$ above.  This implies that the restrictions of $h$ and $\wh{h}$ to $\cyl_+$ have the same distribution (as the projections of $h,\wh{h}$ onto $\CH_2(\cyl)$ also have the same distribution).  In particular, by the Markov property of the GFF, the conditional law of $h$ given its values in any neighborhood $V$ of $\cyl_-$ is that of a GFF in $\cyl \setminus V$ with zero boundary values plus the function which is harmonic in $\cyl \setminus V$ which has the same boundary values on $\partial V$ and at $+\infty$ as $h$.

\begin{lemma}
\label{lem::quantum_sphere_rn}
Fix $r \in \R$.  Suppose that $(\cyl,h,-\infty,+\infty)$ is sampled from $\bsphere$ conditioned on the supremum of the projection of $h$ onto $\CH_1(\cyl)$ exceeding $r$.  We take the horizontal translation so that the projection of~$h$ onto $\CH_1(\cyl)$ first hits $r$ at $u=0$.  Let $U \subseteq \cyl$ be a neighborhood of $+\infty$ which is contained in $\cyl_+$ and let $V$ be a neighborhood of $\cyl_-$ such that $\dist(U,V) > 0$.  Let $g$ be equal to the harmonic extension of $f-h|_V$ from $V$ to $\cyl \setminus V$ and let $\wt{g} = g \phi$ where $\phi \in C^\infty(\CQ)$ is such that $\phi \equiv 0$ in a neighborhood of $V$ and $\phi|_U \equiv 1$.  Let $\CZ_{U,V}$ be defined by
\begin{equation}
\label{eqn::quantum_sphere_rn}
\CZ_{U,V} =  \E\!\left[ \exp\left((h , \wt{g})_\nabla - \| \wt{g} \|_\nabla^2/2 \right) \giv h|_U  \right].
\end{equation}
Then the law of a sample produced from the law of the restriction $h|_U$ of $h$ to $U$ weighted by $\CZ_{U,V}$ is equal to the law of $h|_U$ conditional on the restriction $h|_V$ of $h$ to $V$ being equal to $f$.
\end{lemma}
\begin{proof}
We first recall that if $h$ is a GFF on a domain $D \subseteq \C$ and $f \in H(D)$ then the Radon-Nikodym derivative of the law of $h+f$ with respect to the law of $h$ is given by $\exp((h,f)_\nabla -\| f \|_\nabla^2/2)$.  (This is proved by using that the Radon-Nikodym derivative of the law of a $N(\mu,1)$ random variable with respect to the law of a $N(0,1)$ random variable is given by $e^{x \mu - \mu^2/2}$.)

In order to make use of the aforementioned fact in the setting of the lemma, we first recall from above that from the construction of $\bsphere$ it follows that the conditional law of $h|_U$ given $h|_V$ is equal to that of a GFF on $\cyl \setminus V$ restricted to $U$ with Dirichlet boundary conditions on $\partial V$ and at $+\infty$ given by those of $h$.  Consequently, by the Markov property for the GFF, we can sample from the law of $h|_U$ conditioned on $h|_V = f$ by:
\begin{enumerate}
\item Sampling $h$ according to its unconditioned law and then
\item Adding to $h|_U$ the harmonic extension $g$ of $f-h|_V$ from $V$ to $\cyl \setminus V$.
\end{enumerate}
We can extract from this the result as follows.  We have that $(h+\wt{g})|_U$ has the law of $h|_U$ given $h|_V \equiv f$.  Moreover, we have that the Radon-Nikodym derivative of the conditional law of $h+\wt{g}$ given $h|_V$ with respect to the conditional law of $h$ given $h|_V$ is equal to
\begin{equation}
\label{eqn::rn_two_parts}
\exp( (h,\wt{g})_\nabla - \| \wt{g} \|_\nabla^2/2).
\end{equation}
This implies the result.
\end{proof}

\begin{lemma}
\label{lem::qle_conditional_independence_in_limit_neighborhoods}
Suppose that $(\CS,x,y) = (\C,h,0,\infty)$ is as in Lemma~\ref{lem::quantum_sphere_rn} where the embedding is given by applying the change of coordinates $\cyl \to \C$ given by $z \mapsto e^z$.  Let $V \subseteq \C$ be a bounded open neighborhood of $\D$ and $U \subseteq \C$ be an open neighborhood of $\infty$ such that $\dist(U,V) > 0$.  Let $(\delta_k)$, $(\ol{\delta}_k)$ be sequences as in Section~\ref{subsec::qle_subsequential_limits}.  Suppose that $\qlegrowthleft^k$ is a $\delta_k$-approximation to $\QLE(8/3,0)$ from $0$ to $\infty$ and $\qlegrowthright^k$ is a $\ol{\delta}_k$-approximation to $\QLE(8/3,0)$ from $\infty$ to $0$.  Assume that $\qlegrowthleft^k$ and $\qlegrowthright^k$ are conditionally independent given $h$ and that both have the capacity time parameterization as seen from their target point.  Let $\tau_V$ (resp.\ $\tau_U$) be the infimum of times $t$ that $\qlegrowthleft^k$ (resp.\ $\qlegrowthright^k$) is not in $V$ (resp.\ $U$).  Then the joint law of $(\C,h,0,\infty)$, $\qlegrowthleft^k|_{[0,\tau_V]}$, $\qlegrowthright^k|_{[0,\tau_U]}$ converges weakly as $k \to \infty$ to a triple which consists of a doubly-marked quantum sphere $(\C,h,0,\infty)$ conditioned as in Lemma~\ref{lem::quantum_sphere_rn}, a $\QLE(8/3,0)$ process $\qlegrowthleft$ from $0$ to $\infty$ stopped upon exiting $V$, and a $\QLE(8/3,0)$ process $\qlegrowthright$ from $\infty$ to $0$ stopped upon exiting $U$ with both $\qlegrowthleft$ and $\qlegrowthright$ parameterized by capacity as seen from their target point.  Moreover, $\qlegrowthleft$ and $\qlegrowthright$ are conditionally independent given $h$.
\end{lemma}
\begin{proof}
We know that the joint law of the restriction $h|_V$ of $h$ to $V$ along with $\qlegrowthleft^k$ converges weakly as $k \to \infty$.  We likewise know that the same is true for $h|_U$ along with~$\qlegrowthright^k$.  As mentioned earlier, by \cite[Theorem~1.4]{quantum_spheres}, the Radon-Nikodym derivative of $\Mstwo$ with respect to $\bsphere$ is given by a constant.  It therefore follows that we also have weak convergence as $k \to \infty$ when we produce our sample of $(\CS,x,y)$ using $\bsphere$ in place of $\Mstwo$ and it suffices to show that the weak limit which comes from $\bsphere$ has the desired conditional independence property.  So, in what follows, we assume that $(\CS,x,y)$ is sampled from $\bsphere$.

By the locality property for $\SLE_6$, observe that the Radon-Nikodym derivative $\CZ_{U,V}$ of the conditional law of $h|_U$ and $\qlegrowthright^k|_{[0,\tau_U]}$ given $h|_V$ and $\qlegrowthleft^k|_{[0,\tau_V]}$ with respect to the unconditioned law is equal to the Radon-Nikodym derivative of the conditional law of $h|_U$ given $h|_V$ with respect to the unconditioned law of $h|_U$.   By the explicit form of $\CZ_{U,V}$ given in~\eqref{eqn::quantum_sphere_rn} of Lemma~\ref{lem::quantum_sphere_rn}, we see that $\CZ_{U,V}$ is a continuous function of the harmonic extension of $h|_V$ to $\C \setminus V$ restricted to $U$.  By passing to subsequences of $(\delta_k)$ and $(\ol{\delta}_k)$ if necessary, we may assume that this harmonic extension converges weakly with along the other previously mentioned variables $k \to \infty$.  Combining everything with Lemma~\ref{lem::cond_ind_limit} implies the result.
\end{proof}

\begin{proof}[Proof of Proposition~\ref{prop::qle_conditional_independence_in_limit_stopping_times}]
Lemma~\ref{lem::qle_conditional_independence_in_limit_neighborhoods} implies that the assertion of the proposition holds if we replace $\tau$ with the first time $t$ that $\qlegrowthleft$ exits a bounded open neighborhood~$V$ of~$0$ and we replace $\righthit$ with the first time $t$ that $\qlegrowthright$ exits a neighborhood $U$ of $\infty$ with $\dist(U,V) > 0$.  By taking a limit as $U$ increases to the interior of the complement of $V$, it follows that the assertion of the lemma holds if we take $\tau$ to be the first time $t$ that $\qlegrowthleft$ exits $V$ and $\righthit$ to be the first time $t$ that $\qlegrowthright$ exits the interior of $\C \setminus V$.  The result follows because this holds for all $V$ which are bounded and open.
\end{proof}

\subsection{Reshuffling a pair of $\SLE_6$'s}
\label{subsec::sle_6_reshuffle}

\subsubsection{Setup}

We begin by describing the setup and notation that we will use throughout this subsection.  We first suppose that $\CX$, $\ol{\CX}$, $\CB$, $\eta'$, $\ol{\eta}'$,~$\ttime$, and~$\ol{\ttime}$ are as in Theorem~\ref{thm::typical_cut_point}.  (In the setting of the statement of Theorem~\ref{thm::qle_symmetry},~$\ttime$ plays the role of $\stoppingleft$ and $\ol{\ttime}$ plays the role of $\righthit$.)  Then we know that the $\MstwoD$ distribution of $(\CX,\ol{\CX},\CB)$ is invariant under the operation of swapping $\CX$ and $\ol{\CX}$.  Let $(\delta_k)$, $(\ol{\delta}_k)$ be sequences of positive numbers decreasing to $0$ as in Section~\ref{subsec::qle_subsequential_limits}.  Let $\qlegrowthleft^k|_{[0,\ttime]}$ be the $\delta_k$-approximation to $\QLE(8/3,0)$ which we take to be coupled with $\eta'|_{[0,\ttime]}$ so that the bubbles that it separates from its target are the same as quantum surfaces as the bubbles separated by $\eta'|_{[0,\ttime]}$ from its target point.  That is, in the construction of $\qlegrowthleft^k|_{[0,\ttime]}$ we take the time-reversal of the L\'evy excursion $e$ to be the same as the time-reversal of the L\'evy excursion for $\eta'$ up to time~$\ttime$ and we take the bubbles that $\qlegrowthleft^k|_{[0,\ttime]}$ separates from its target point up to time $\ttime$ to be the same as the bubbles that $\eta'$ separates from its target point up to time~$\ttime$.  We sample the rest of the process and quantum sphere conditionally independently given its realization of up to time~$\ttime$.

We likewise let $\qlegrowthright^k|_{[0,\ol{\ttime}]}$ be the $\ol{\delta}_k$-approximation to $\QLE(8/3,0)$ which we take to be coupled with $\ol{\eta}'|_{[0,\ol{\ttime}]}$ so that the bubbles it separates from its target are the same as quantum surfaces as the bubbles separated by $\ol{\eta}'|_{[0,\ol{\ttime}]}$ from its target point.  We sample the rest of the process and the quantum sphere on which it is growing conditionally independently given its realization up to time $\ol{\ttime}$.

Consider the doubly-marked quantum sphere $(\CS^k,x^k,y^k)$ which arises by starting with the doubly-marked sphere $(\CS,x,y)$ as above, cutting out the quantum surfaces separated by $\eta'|_{[0,\ttime]}$ and $\ol{\eta}'|_{[0,\ol{\ttime}]}$ from their respective target points, and then gluing in according to quantum boundary length the corresponding quantum surfaces for $\qlegrowthleft^k|_{[0,\ttime]}$ and $\qlegrowthright^k|_{[0,\ol{\ttime}]}$.  (See Figure~\ref{fig::qle_symmetry} for an illustration.)  We identify the tip of the final $\SLE_6$ segment used to build $\qlegrowthleft^k|_{[0,\ttime]}$ with $\eta'(\ttime)$ and likewise for $\qlegrowthright^k|_{[0,\ol{\ttime}]}$ and $\ol{\eta}'(\ol{\ttime})$.  Proposition~\ref{prop::figure_8_quantum_disk_and_removable} implies that the resulting quantum surface is uniquely determined by this gluing operation.  Moreover, the distribution of $(\CS^k,x^k,y^k)$ is equal to that of $(\CS,x,y)$.  We abuse notation and let $\qlegrowthleft^k$ and $\qlegrowthright^k$ be the resulting $\delta_k$ and $\ol{\delta}_k$-approximations to $\QLE(8/3,0)$ and let $\CB^k$ be the region in $\CS^k$ which is not in either $\qlegrowthleft_\ttime^k$ or $\qlegrowthright_{\ol{\ttime}}^k$.

Throughout, we will write $\stoppingleft$ (resp.\ $\righthit$) for $\ttime$ (resp.\ $\ol{\ttime}$) when referring to either $\qlegrowthleft^k$ or $\qlegrowthleft$ (resp.\ $\qlegrowthright^k$ or $\qlegrowthright$) and use $\ttime$ (resp.\ $\ol{\ttime}$) when referring to $\eta'$ (resp.\ $\ol{\eta}'$).

We let $\lawonall_k$ denote the joint distribution of $(\CS^k,x^k,y^k)$, $\qlegrowthleft^k$,  $\qlegrowthright^k$, $\stoppingleft$ and $\righthit$ and let $\lawonall_k^{x \to y}$ be given by $\lawonall_k$ weighted by the amount of time it takes $\qlegrowthleft^k$ to reach its target point.  We define $\lawonall_k^{y \to x}$ similarly except we weight by the amount of time it takes $\qlegrowthright^k$ to reach its target point and we write $\stoppingright$ for a uniform variable between $0$ and this time and let $\lefthit$ be the first time that $\qlegrowthleft_t^k$ hits $\qlegrowthright_\stoppingright^k$.  We write $\ol{\CB}^k$ for the surface which is the complement of $\qlegrowthleft_\lefthit^k$ and $\qlegrowthright_\stoppingright^k$.  Note that $\lawonall_k^{x \to y}$ and $\lawonall_k^{y \to x}$ are the analogs of $\lawonall^{x \to y}$ and $\lawonall^{y \to x}$ defined at the beginning of this section except we have used the approximations to $\QLE(8/3,0)$ instead of the limits.

Let $(\qlegrowthleft|_{[0,\stoppingleft]},\qlegrowthright|_{[0,\righthit]},\CB)$ be the configuration on the doubly-marked quantum sphere $(\CS,x,y)$ which arises by taking the limit of $(\qlegrowthleft^k|_{[0,\stoppingleft]},\qlegrowthright^k|_{[0,\righthit]},\CB^k)$ under $\lawonall_k^{x \to y}$ as $k \to \infty$.  The sense in which this limit is taken is as in Proposition~\ref{prop::qle_conditional_independence_in_limit_stopping_times}.  Proposition~\ref{prop::qle_conditional_independence_in_limit_stopping_times} implies its existence and that the distribution of $(\CS,x,y)$, $\qlegrowthleft|_{[0,\stoppingleft]}$, and $\qlegrowthright|_{[0,\righthit]}$ as constructed is equal to its distribution under $\lawonall^{x \to y}$.  We likewise let $(\qlegrowthleft|_{[0,\lefthit]},\qlegrowthright|_{[0,\stoppingright]},\ol{\CB})$ be the configuration on the doubly-marked quantum sphere $(\CS,x,y)$ which arises by taking the limit of $(\qlegrowthleft^k|_{[0,\lefthit]},\qlegrowthright^k|_{[0,\stoppingright]},\ol{\CB}^k)$ under $\lawonall_k^{y \to x}$ as $k \to \infty$.  Analogously, we can construct the law of $(\CS,x,y)$, $\qlegrowthleft|_{[0,\lefthit]}$, $\qlegrowthright|_{[0,\stoppingright]}$ under $\lawonall^{y \to x}$ as the same type of limit.

\subsubsection{Structure of the meeting $\QLE$s}

In this section we will prove two results regarding the structure of the triple $\qlegrowthleft|_{[0,\stoppingleft]}$, $\qlegrowthright|_{[0,\righthit]}$, and $\CB$ under $\lawonall^{x \to y}$ (as well as analogous results in the setting of $\lawonall^{y \to x}$).

\begin{figure}[ht!!]
\begin{center}
\includegraphics[scale=0.85]{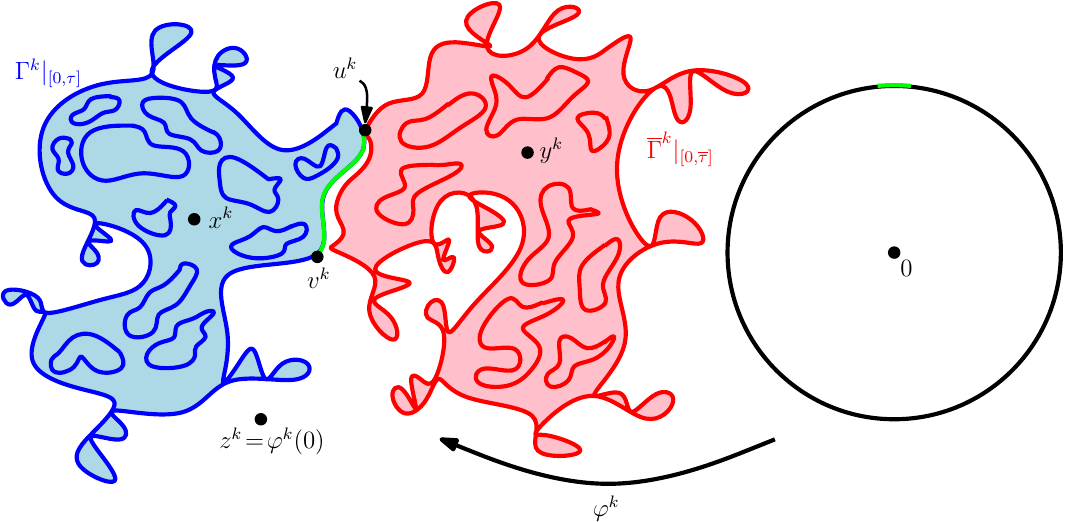}
\end{center}
\caption{\label{fig::hitting_set} Illustration of the setup and proof of Lemma~\ref{lem::hitting_set}.  Shown on the left are two conditionally independent $\delta_k$-approximations $\qlegrowthleft^k|_{[0,\stoppingleft]},\qlegrowthright^k|_{[0,\righthit]}$ of $\QLE(8/3,0)$ on a doubly-marked quantum sphere $(\CS^k,x^k,y^k)$ where $\qlegrowthleft^k|_{[0,\stoppingleft]}$ is grown up to a typical quantum distance time and $\qlegrowthright^k|_{[0,\righthit]}$ is grown up until the first time that it hits $\qlegrowthleft^k|_{[0,\stoppingleft]}$.  It is a.s.\  the case that these two processes intersect at exactly one point because radial $\SLE_6$ a.s.\  hits the boundary of any given simply connected domain for the first time at a single point; call the complementary region $\CB^k$ and let $z^k$ be a uniformly chosen point from the quantum measure on $\CB^k$.  By Proposition~\ref{prop::figure_8_quantum_disk_and_removable}, $\CB^k$ is a quantum disk conditionally on its boundary length and the distribution of $\CB^k$ does not depend on $k$.  If, in the limit as $k \to \infty$, there was a positive chance that the processes intersect at more than one point then there would be a positive chance that the $\delta_k$-approximations would form a narrow fjord the part of whose boundary which is in $\partial \qlegrowthleft_\stoppingleft^k$ has uniformly positive quantum length.  This leads to a contradiction because harmonic measure considerations imply that the inverse image of the green part of the fjord under the unique conformal map $\varphi^k \colon \D \to \CB^k$ with $\varphi^k(0) = z^k$ and $(\varphi^k)'(0) > 0$ would be a segment with a uniformly positive amount of quantum length (indicated in green on the right) but with arc length tending to $0$.  In particular, it would converge in the $k \to \infty$ limit to an atom of the boundary measure.}
\end{figure}

\begin{lemma}
\label{lem::hitting_set}
Under $\lawonall^{x \to y}$, we have that~$\qlegrowthleft|_{[0,\stoppingleft]}$ and~$\qlegrowthright|_{[0,\righthit]}$ a.s.\  intersect at a single point.  In particular, the region~$\CB$ of $\CS$ outside of~$\qlegrowthleft|_{[0,\stoppingleft]}$ and~$\qlegrowthright|_{[0,\righthit]}$ is a.s.\  connected.  Moreover, the conditional law of $\CB$ given its boundary length is that of a quantum disk and $\partial \CB$ is a.s.\  conformally removable in any embedding of $\CS$ into $\C$.  The same holds in the setting of $\lawonall^{y \to x}$ in place of $\lawonall^{x \to y}$.
\end{lemma}

See Figure~\ref{fig::hitting_set} for an illustration of the setup and proof of Lemma~\ref{lem::hitting_set}.

\begin{proof}[Proof of Lemma~\ref{lem::hitting_set}]
We know that the statement of the lemma holds in the setting of each of the $\delta_k$-approximations; see also the proof of Lemma~\ref{lem::approximations_symmetric} just below.  Indeed, this follows because a radial $\SLE_6$ segment a.s.\  hits the boundary of any fixed domain for the first time at a single point.  For each $k \in \N$, we pick $z^k$ in $\CB^k$ according to the quantum area measure and then parameterize $\CB^k$ by $(\D,h^k)$ by using the change of coordinates given by the unique conformal map $\varphi^k \colon \D \to \CB^k$ with $\varphi^k(0) = z^k$ and $(\varphi^k)'(0) > 0$.  We note that the distribution of $(\D,h^k)$ does not depend on $k$ and that, by Proposition~\ref{prop::figure_8_quantum_disk_and_removable} its law conditionally on its boundary length is that of a quantum disk.

Fix $\ell > 0$ and assume that we are working on the event that the quantum boundary length of $\partial \qlegrowthleft_\stoppingleft^k$ is in $[\ell,\ell+1]$.  Since this event has positive and finite $\lawonall_k^{x \to y}$ mass, we can condition $\lawonall_k^{x \to y}$ on it and we get a probability measure.  Suppose for contradiction that there is a chance $p > 0$ that $\partial \qlegrowthleft_\stoppingleft \cap \partial \qlegrowthright_{\righthit}$ consists of more than a single point.  This would imply that the following is true.  Let $u^k$ be the unique element of $\partial \qlegrowthleft_\stoppingleft^k \cap \partial \qlegrowthright_{\righthit}^{k}$.  Then there exists $r > 0$ such that with probability at least $p/2$ there is a point $v^k$ on $\partial \qlegrowthleft_\stoppingleft^k$ such that the quantum length of either the clockwise or counterclockwise segment of $\partial \qlegrowthleft_\stoppingleft^k$ from $u^k$ to $v^k$ is at least $r$ and its harmonic measure as seen from $z^k$ is $o(1)$ as $k \to \infty$.  In particular, the arc length of the preimage of this interval under $\varphi^k$ is $o(1)$ as $k \to \infty$.  This implies that $(\D, h^k)$ converges as $k \to \infty$ to a field whose associated quantum boundary measure has a positive chance of having an atom on $\partial \D$.  This is a contradiction because the law of $(\D, h^k)$ does not depend on $k$ and its associated quantum boundary measure a.s.\  does not have an atom on $\partial \D$ (Proposition~\ref{prop::figure_8_quantum_disk_and_removable} and \cite{DS08}).  This proves the first assertion of the lemma.

The second assertion follows because the distribution of $\CB^k$ does not depend on $k$ and the final assertion follows by using the argument of Proposition~\ref{prop::figure_8_quantum_disk_and_removable} (i.e., any embedding of a quantum disk into a quantum sphere a.s.\  has a conformally removable boundary). 
\end{proof}

\begin{lemma}
\label{lem::pinch_point_glued}
Under $\lawonall^{x \to y}$, the point on $\partial \qlegrowthleft_{\stoppingleft}$ (resp.\ $\partial \qlegrowthright_{\righthit}$) which is glued to the a.s.\  unique pinch point of $\partial \CB$ is uniformly distributed from the quantum boundary measure on $\partial \qlegrowthleft_{\stoppingleft}$ (resp.\ $\partial \qlegrowthright_{\righthit}$).  The same holds in the setting of $\lawonall^{y \to x}$ in place of $\lawonall^{x \to y}$.
\end{lemma}
\begin{proof}
It is easy to see that this is the case for $\qlegrowthleft|_{[0,\stoppingleft]}$ based on the construction.  Let $[\xi,\righthit]$ be the interval of time in which $\qlegrowthright^k|_{[0,\righthit]}$ is drawing its final segment of $\SLE_6$.  The claim in the case of $\qlegrowthright|_{[0,\righthit]}$ follows because for each of the approximations $\qlegrowthright^k|_{[0,\righthit]}$ we know that the starting point of the final $\SLE_6$ segment is uniformly distributed on $\partial \qlegrowthright_\xi^k$ and this final $\SLE_6$ segment collapses to a point as $k \to \infty$.
\end{proof}

\subsubsection{Preliminary symmetry statements}

In order to work towards the proof of Theorem~\ref{thm::qle_symmetry}, we will now:
\begin{itemize}
\item Show that the $\lawonall_k^{x \to y}$ distribution of $(\qlegrowthleft^k|_{[0,\stoppingleft]},\qlegrowthright^k|_{[0,\righthit]},\CB^k)$ is equal to the $\lawonall_k^{y \to x}$ distribution of $(\qlegrowthleft^k|_{[0,\lefthit]},\qlegrowthright^k|_{[0,\stoppingright]},\ol{\CB}^k)$ (Lemma~\ref{lem::approximations_symmetric})
\item Deduce from this that the $\lawonall^{x \to y}$ distribution of $(\qlegrowthleft|_{[0,\stoppingleft]},\qlegrowthright|_{[0,\righthit]},\CB)$ is equal to the $\lawonall^{y \to x}$ distribution of $(\qlegrowthleft|_{[0,\lefthit]},\qlegrowthright|_{[0,\stoppingright]},\ol{\CB})$ (Lemma~\ref{lem::continuum_equivalence})
\end{itemize}
The proof of Theorem~\ref{thm::qle_symmetry} will not require much additional work once we have established these lemmas.

\begin{lemma}
\label{lem::approximations_symmetric}
The $\lawonall_k^{x \to y}$ distribution of $(\qlegrowthleft^k|_{[0,\stoppingleft]},\qlegrowthright^k|_{[0,\righthit]},\CB^k)$ is the same as the $\lawonall_k^{y \to x}$ distribution of $(\qlegrowthleft^k|_{[0,\lefthit]},\qlegrowthright^k|_{[0,\stoppingright]},\ol{\CB}^k)$.
\end{lemma}

In the statement of Lemma~\ref{lem::approximations_symmetric}, we view $\qlegrowthleft^k|_{[0,\stoppingleft]}$, $\qlegrowthright^k|_{[0,\righthit]}$, $\qlegrowthleft^k|_{[0,\lefthit]}$, $\qlegrowthright^k|_{[0,\stoppingright]}$ as random variables taking values in the space of beaded quantum surfaces with a single marked point which corresponds to the tip of the final $\SLE_6$ segment in the growth process and we view $\CB^k, \ol{\CB}^k$ as random variables taking values in the space of quantum surfaces.

\begin{remark}
\label{rem::qle_not_symmetric}
As in Remark~\ref{rem::sle_whole_picture_not_symmetric}, we emphasize that Lemma~\ref{lem::approximations_symmetric} does not imply that the $\lawonall_k^{x \to y}$ distribution of $(\qlegrowthleft^k|_{[0,\stoppingleft]},\qlegrowthright^k|_{[0,\righthit]},\CS^k)$ is equal to the $\lawonall_k^{y \to x}$ distribution of $(\qlegrowthleft^k|_{[0,\lefthit]},\qlegrowthright^k|_{[0,\stoppingright]},\CS^k)$ because of the asymmetry which arises since the tip of $\qlegrowthright^k|_{[0,\righthit]}$ is contained in $\partial \qlegrowthleft_\stoppingleft^k$ while the tip of $\qlegrowthleft^k|_{[0,\stoppingleft]}$ is uniformly distributed according to the quantum boundary measure on $\partial \qlegrowthleft_\stoppingleft^k$.
\end{remark}

\begin{proof}[Proof of Lemma~\ref{lem::approximations_symmetric}]
The result is a consequence of the following three observations.

First, the ordered sequence of bubbles separated by $\qlegrowthleft^k|_{[0,\stoppingleft]}$ from its target point is equal to the ordered sequence of bubbles separated by $\eta'|_{[0,\ttime]}$ from its target point and the same is true for $\qlegrowthright^k|_{[0,\righthit]}$ and $\ol{\eta}'|_{[0,\ol{\ttime}]}$.  Theorem~\ref{thm::typical_cut_point} thus implies that the collection of $\delta_k$-quantum natural time length $\SLE_6$ segments (viewed as beaded quantum surfaces) which make up $\qlegrowthleft^k|_{[0,\stoppingleft]}$ under $\lawonall_k^{x \to y}$ have the same joint law as the corresponding collection for $\qlegrowthleft^k|_{[0,\lefthit]}$ under $\lawonall_k^{y \to x}$.  It also implies that the quantum lengths of the outer boundary of $\qlegrowthleft^k|_{[0,\stoppingleft]}$ at the reshuffling times $j \delta_k$ for $1 \leq j \leq \lfloor \stoppingleft/\delta_k \rfloor$ have the same joint law under $\lawonall_k^{x \to y}$ as the quantum lengths of the outer boundary of $\qlegrowthleft^k|_{[0,\lefthit]}$ at the reshuffling times $j \delta_k$ for $1 \leq j \leq \lfloor \lefthit/\delta_k \rfloor$ under $\lawonall_k^{y \to x}$.  Indeed, this follows since these quantum boundary lengths are determined by the bubbles that the processes separate from their target points and the bubbles determine the $\SLE_6$ segments.

Second, the starting locations of each of the $\SLE_6$ segments which make up $\qlegrowthleft^k|_{[0,\stoppingleft]}$ have the same conditional law under $\lawonall_k^{x \to y}$ given the bubbles as in the case of $\qlegrowthleft^k|_{[0,\lefthit]}$ under $\lawonall_k^{y \to x}$.  Indeed, in both cases, these starting locations are distributed uniformly from the quantum boundary measure.

Both of the first two observations also apply for $\qlegrowthright^k|_{[0,\righthit]}$ under $\lawonall_k^{x \to y}$ and $\qlegrowthright^k|_{[0,\stoppingright]}$ under~$\lawonall_k^{y \to x}$.

Third, the previous two observations imply that the induced joint distribution of $\qlegrowthleft^k|_{[0,\stoppingleft]}$ and $\qlegrowthright^k|_{[0,\righthit]}$ under $\lawonall_k^{x \to y}$ is the same as that of $\qlegrowthleft^k|_{[0,\lefthit]}$ and $\qlegrowthright^k|_{[0,\stoppingright]}$ under $\lawonall_k^{y \to x}$.  We also note that $\CB^k$ is conditionally independent given its boundary length of $\qlegrowthleft^k|_{[0,\stoppingleft]}$ and $\qlegrowthright^k|_{[0,\righthit]}$ under $\lawonall_k^{x \to y}$.  The same is also true under $\lawonall^{y \to x}$, so combining implies the desired result.
\end{proof}

The next step is to show that the symmetry established in Lemma~\ref{lem::approximations_symmetric} also holds for $(\qlegrowthleft|_{[0,\stoppingleft]},\qlegrowthright|_{[0,\righthit]},\CB)$ under $\lawonall^{x \to y}$ and $(\qlegrowthleft|_{[0,\lefthit]},\qlegrowthright|_{[0,\stoppingright]},\ol{\CB})$ under $\lawonall^{y \to x}$.

Before we state this result, we need to introduce a certain notion of equivalence for $\QLE(8/3,0)$ processes.  We recall from Proposition~\ref{prop::qle_boundary_length} that the complementary region of a $\QLE(8/3,0)$ run up to a given quantum distance time has the same law as the corresponding region for an $\SLE_6$ and from Proposition~\ref{prop::qle_outer_boundary} that the boundary of a $\QLE(8/3,0)$ can be parameterized by a continuous curve.  This allows us to induce a quantum boundary length measure on the prime ends of the boundary of a $\QLE(8/3,0)$ even though it is not {\it a priori} clear at this point whether this boundary measure can be defined in a way which is intrinsic to the $\QLE(8/3,0)$.  (At this point, we do know that the total boundary length is intrinsic to the $\QLE(8/3,0)$ because it can be determined by boundary lengths of the bubbles it has cut off from its target point.  This follows because the value of a stable L\'evy process at a given time can be recovered from the jumps that it has made up to this time.  See \cite[Chapter~I, Theorem~1]{bertoin96levy}.)  

Suppose that $(\CS,x,y)$, $\qlegrowthleft$ and $(\wt{\CS},\wt{x},\wt{y})$, $\wt{\qlegrowthleft}$ are two $\QLE(8/3,0)$-decorated quantum surfaces which are defined on a common probability space.  Suppose further that $\tau$ (resp.\ $\wt{\tau}$) is a stopping time for $\qlegrowthleft$ (resp.\ $\wt{\qlegrowthleft}$).  Let $h$ (resp.\ $\wt{h}$) be the field which describes $\CS$ (resp.\ $\wt{\CS}$).  We say that $\qlegrowthleft|_{[0,\tau]}$ is equivalent to $\wt{\qlegrowthleft}|_{[0,\wt{\tau}]}$ if there exists a homeomorphism $\varphi$ taking $\qlegrowthleft_\tau$ to $\wt{\qlegrowthleft}_{\wt{\tau}}$ with $\varphi(\qlegrowthleft_t) = \wt{\qlegrowthleft}_t$ for all $t \leq \tau$, $\varphi(\qlegrowthleft_\tau) = \wt{\qlegrowthleft}_{\wt{\tau}}$, and which
\begin{itemize}
\item Takes the quantum boundary measure defined on the prime ends of the outer boundary of $\qlegrowthleft_\tau$ to the quantum boundary measure on the prime ends of the outer boundary of $\wt{\qlegrowthleft}_{\wt{\tau}}$, and is
\item Conformal on the interior of $\qlegrowthleft_\tau$ with $\wt{h} \circ \varphi + Q\log|\varphi'| = h$ on the interior of $\qlegrowthleft_\tau$.
\end{itemize}
We emphasize that, at this point, we have not proved that the boundary of a $\QLE(8/3,0)$ is equal to the boundary of its interior or that the boundary of a $\QLE(8/3,0)$ does not contain cut points or spikes.  This is the reason that we consider prime ends in this notion of equivalence.  (It will be a consequence of the subsequent work \cite{qle_continuity} that this type of behavior does not occur.)  As we will shall see later, this notion of equivalence is useful because it encodes the information necessary for the operation of welding according to quantum boundary length to be well-defined and uniquely determined (i.e., we will be in a setting in which we can apply the conformal removability of the boundary; recall Proposition~\ref{prop::removability_usage}).

\begin{lemma}
\label{lem::continuum_equivalence}
The $\lawonall^{x \to y}$ distribution of $(\qlegrowthleft|_{[0,\stoppingleft]},\qlegrowthright|_{[0,\righthit]},\CB)$ is equal to the $\lawonall^{y \to x}$ distribution of $(\qlegrowthleft|_{[0,\lefthit]},\qlegrowthright|_{[0,\stoppingright]},\ol{\CB})$ where we use the notion of equivalence for $\QLE(8/3,0)$ processes as described just above.
\end{lemma}

See Figure~\ref{fig::last_step_bl_change} and Figure~\ref{fig::continuum_equivalence} for an illustration of the proof.  Lemma~\ref{lem::approximations_symmetric} implies that one can construct a coupling of an instance $(\qlegrowthleft^{1,k}|_{[0,\stoppingleft^1]}, \qlegrowthright^{1,k}|_{[0,\righthit^1]},\CB^{1,k})$ from $\lawonall_k^{x \to y}$ and an instance $(\qlegrowthleft^{2,k}|_{[0,\lefthit^2]}, \qlegrowthright^{2,k}|_{[0,\stoppingright^2]},\ol{\CB}^{2,k})$ from $\lawonall_k^{y \to x}$ so that each of the components are equivalent as (marked and beaded) quantum surfaces.  In particular, there is a homeomorphism of $\qlegrowthleft_{\stoppingleft^1}^{1,k}$ to $\qlegrowthleft_{\lefthit^2}^{2,k}$ which is conformal on the interior of $\qlegrowthleft_{\stoppingleft^1}^{1,k}$.  Since we need to establish the equivalence of the limits as $k \to \infty$ in the sense described above, we need to know that the limit of this map is a homeomorphism and preserves the boundary length measure.  In order to show that this is the case, we will want construct the coupling so that we can conformally map an entire neighborhood of $\qlegrowthleft_{\stoppingleft^1}^{1,k}$ to a neighborhood of $\qlegrowthleft_{\lefthit^2}^{2,k}$.  The challenge is that this is not exactly possible due to the asymmetry in the locations of the tips of the two $\QLE$ approximations.

To explain this point in further detail, the tip of $\qlegrowthleft^{2,k}|_{[0,\lefthit^2]}$ is contained in the outer boundary of $\qlegrowthright^{2,k}|_{[0,\stoppingright^2]}$ while the the tip of $\qlegrowthleft^{1,k}|_{[0,\stoppingleft^1]}$ is uniformly random on its boundary (see Figure~\ref{fig::last_step_bl_change}).  The difficulty that this causes is that if we couple an instance $(\qlegrowthleft^{1,k}|_{[0,\stoppingleft^1]}, \qlegrowthright^{1,k}|_{[0,\righthit^1]},\CB^{1,k})$ from $\lawonall_k^{x \to y}$ and an instance $(\qlegrowthleft^{2,k}|_{[0,\lefthit^2]}, \qlegrowthright^{2,k}|_{[0,\stoppingright^2]},\ol{\CB}^{2,k})$ from $\lawonall_k^{y \to x}$ so that each of the components are equivalent as (marked and beaded) quantum surfaces, then it is not possible to arrange so that the map which takes $\qlegrowthleft_{\stoppingleft^1}^{1,k}$ to $\qlegrowthleft_{\lefthit^2}^{2,k}$ agrees with the map which takes $\CB^{1,k}$ to $\ol{\CB}^{2,k}$ on the boundary to define a welding hence be conformal in a neighborhood of $\qlegrowthleft_{\stoppingleft^1}^{1,k}$.  To handle the issue with the asymmetry, we will couple so that $(\qlegrowthleft^{1,k}|_{[0,\stoppingleft^1]}, \qlegrowthright^{1,k}|_{[0,\righthit^1]})$ and $(\qlegrowthleft^{2,k}|_{[0,\lefthit^2]}, \qlegrowthright^{2,k}|_{[0,\stoppingright^2]})$ are equivalent as (marked and beaded) quantum surfaces and only consider the map of the part of $\qlegrowthleft^{1,k}$ up to the time $\xi^1 = \delta_k \lfloor \stoppingleft^1/\delta_k \rfloor$ to the part of $\qlegrowthleft^{2,k}$ up to the time $\xi^2 = \delta_k \lfloor \lefthit^2/\delta_k \rfloor$.  We would like to then glue the points on $\partial \qlegrowthleft_{\stoppingleft^1}^{1,k} \setminus \partial \qlegrowthleft_{\xi^1}^{1,k}$ to the points on $\partial \qlegrowthleft_{\xi^1}^{1,k} \setminus \partial \qlegrowthleft_{\stoppingleft^1}^{1,k}$ according to boundary length (boundaries of the yellow region which are part of the outer boundary of $\qlegrowthleft^{1,k}$ at the times $\xi^1$, $\stoppingleft^1$).  If the boundary lengths of $\partial \qlegrowthleft_{\stoppingleft^1}^{1,k} \setminus \partial \qlegrowthleft_{\xi^1}^{1,k}$ and $\partial \qlegrowthleft_{\xi^1}^{1,k} \setminus \partial \qlegrowthleft_{\stoppingleft^1}^{1,k}$ are not the same (which is a.s.\ the case), then it is not possible to glue according to boundary length.  To circumvent this problem, we grow the segment of $\SLE_6$ being drawn by $\qlegrowthleft^{1,k}$ at time $\stoppingleft^1$ further until the first time $\zeta^1$ that the quantum length of the outer boundary of the surface formed is the same as the quantum length of the outer boundary of $\qlegrowthleft^{1,k}$ at time $\xi^1$.  We let $\CR^1$ be the surface that this $\SLE_6$ cuts off from $\infty$ in the interval $[\xi^1,\zeta^1]$ (union of yellow and green regions in Figure~\ref{fig::last_step_bl_change}) and $\wh{\CB}^{1,k} = \CB^{1,k} \setminus \CR^1$.  Note that the surface $\wh{\CB}^{1,k}$ a.s.\ does not have the same boundary length as $\ol{\CB}^{2,k}$, so they cannot be coupled to be the same.  We correct this by growing the segment $\ol{\eta}'$ of $\SLE_6$ being drawn by $\qlegrowthright^{2,k}$ at time $\stoppingright^2$ until the first time $\ol{\zeta}^2$ that the boundary length of the region $\wh{\CB}^{2,k}$ outside of $\qlegrowthleft_{\lefthit^2}^{2,k} \cup \qlegrowthright_{\stoppingright^2}^{2,k} \cup \ol{\eta}'([\stoppingright^2,\ol{\zeta}^2])$ is the same as that of $\wh{\CB}^{1,k}$.  If we couple so that $\wh{\CB}^{1,k} = \wh{\CB}^{2,k}$, then we can define a map $\qlegrowthleft_{\xi^1}^{1,k} \cup \wh{\CB}^{1,k} \to \qlegrowthleft_{\xi^2}^{2,k} \cup \wh{\CB}^{2,k}$ by gluing the part of $\qlegrowthleft_{\xi^1}^{1,k} \cap \partial \CR^1$ to $\partial \CR^1 \setminus \qlegrowthleft_{\xi^1}^{1,k}$ according to boundary length.  This yields a welding as the two boundary lengths agree, so this map will thus be conformal in a neighborhood of $\qlegrowthleft_{\xi^1}^{1,k}$ except near $\CR^1$ and where the two $\QLE(8/3,0)$'s meet.  The former does not cause a problem as $\diam(\CR^1) \to 0$ in probability as $k \to \infty$ and the latter does not cause a problem since the two $\QLE(8/3,0)$'s a.s.\ meet at a single point.

\begin{figure}[ht!]
\begin{center}
\includegraphics[scale=0.85, page=1, trim = {0 1.5cm 7.5cm 0},clip]{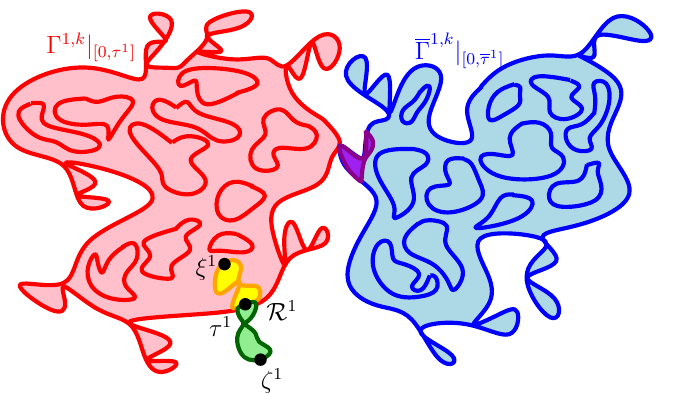}
\end{center}
\vspace{-0.03\textheight}
\caption{\label{fig::last_step_bl_change} Illustration of the setup for Lemma~\ref{lem::continuum_equivalence}.  In the figure and in the proof, we use a superscript $1$ (resp.\ $2$) to indicate a sample from the measure which is weighted by quantum distance time from $x$ to $y$ (resp.\ $y$ to $x$).  Here, $[\xi^1,\stoppingleft^1]$ is the interval of time in which $\qlegrowthleft^{1,k}|_{[0,\stoppingleft^{1}]}$ is drawing its final segment of $\SLE_6$ (shown in yellow).  The green region shows the surface which is cut off by running this $\SLE_6$ further until the first time that the boundary length of the surface which corresponds to $\qlegrowthleft_{\stoppingleft^{1}}^{1,k}$ and the additional $\SLE_6$ segment is equal to the boundary length of $\partial \qlegrowthleft_{\xi^{1}}^{1,k}$.  It is a high probability event that this time occurs very quickly after time $\stoppingleft^{1}$ (when $\delta_k > 0$ is small).  The purple region shows the final segment of $\SLE_6$ for $\qlegrowthright^{1,k}|_{[0,\righthit^{1}]}$.  The region $\CR^{1}$ is the union of the yellow and green regions.}
\end{figure}

\begin{figure}[ht!]
\begin{center}
\includegraphics[scale=0.8, page=2]{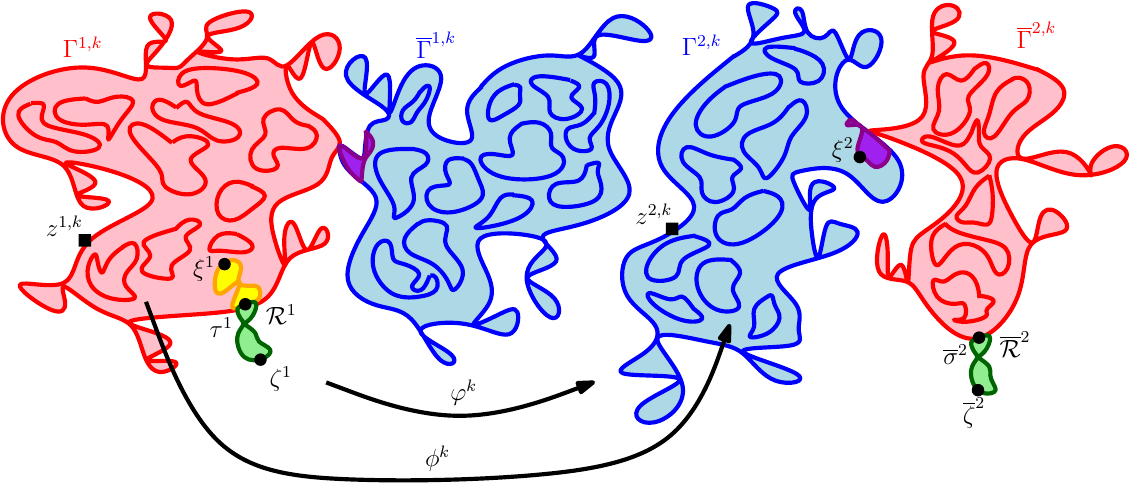}
\end{center}
\vspace{-0.025\textheight}
\caption{\label{fig::continuum_equivalence} (Continuation of Figure~\ref{fig::last_step_bl_change}.)  We construct a coupling under which $\qlegrowthleft^{1,k}|_{[0,\stoppingleft^1]}$ (resp.\ $\qlegrowthright^{1,k}|_{[0,\righthit^1]}$) is equal to $\qlegrowthleft^{2,k}|_{[0,\lefthit^2]}$ (resp.\ $\qlegrowthright^{2,k}|_{[0,\stoppingright^2]}$) as marked (by their tips) and beaded quantum surfaces.  The region $\CR^1$ is as in Figure~\ref{fig::last_step_bl_change}.  The quantum surface $\wh{\CB}^{1,k}$ parameterized by the complement of $\qlegrowthleft_{\stoppingleft^1}^{1,k} \cup \qlegrowthright_{\righthit^1}^{1,k} \cup \CR^1$ is a quantum disk given its quantum boundary length.  The region $\ol{\CR}^2$ is the continuation of the $\SLE_6$ being drawn by $\qlegrowthright^{2,k}$ at time $\stoppingright^2$ up until the first time that the quantum length of the outside component is equal to that of $\partial \wh{\CB}^{1,k}$; call this outside component $\wh{\CB}^{2,k}$.  Then $\wh{\CB}^{1,k}$, $\wh{\CB}^{2,k}$ are both quantum disks given their quantum boundary lengths.  Since these quantum boundary lengths are the same, we can therefore couple so that the two surfaces are the same.  We therefore have maps $\phi^k$ which take $\qlegrowthleft_{\xi^1}^{1,k}$ to $\qlegrowthleft_{\xi^2}^{2,k}$ and $\varphi^k$ which take $\wh{\CB}^{1,k}$ to $\wh{\CB}^{2,k}$ which are conformal on their interior and preserve quantum boundary length.  By removability, these maps together define a conformal transformation $\psi^k$ in the complement of $\CR^1 \cup \qlegrowthright_{\righthit^1}^{1,k}$.  By passing to a further subsequence (and recoupling) if necessary, $\psi^k$ converges in the limit as $k \to \infty$ to a map which is conformal in the complement of $\qlegrowthright^1|_{[0,\righthit^1]}$, preserves quantum area and boundary length on $\qlegrowthleft_{\stoppingleft^1}^1$, and takes $\qlegrowthleft^1|_{[0,\stoppingleft^1]}$ to $\qlegrowthleft^2|_{[0,\lefthit^2]}$.}
\end{figure}

\begin{proof}[Proof of Lemma~\ref{lem::continuum_equivalence}]\noindent{\it Step 1.  Coupling of the $\QLE$ approximations.}  We begin by starting off with samples $(\qlegrowthleft^{1,k}|_{[0,\stoppingleft^1]},\qlegrowthright^{1,k}|_{[0,\righthit^1]},\CB^{1,k})$ and $(\qlegrowthleft^{2,k}|_{[0,\lefthit^2]},\qlegrowthright^{2,k}|_{[0,\stoppingright^2]},\ol{\CB}^{2,k})$ respectively from $\lawonall_k^{x \to y}$ and $\lawonall_k^{y \to x}$.  We take these to be coupled together so that (as marked (by their tips) beaded quantum surfaces) we have $\qlegrowthleft^{1,k}|_{[0,\stoppingleft^1]} = \qlegrowthleft^{2,k}|_{[0,\lefthit^2]}$ and $\qlegrowthright^{1,k}|_{[0,\righthit^1]} = \qlegrowthright^{2,k}|_{[0,\stoppingright^2]}$.  We note that both $\CB^{1,k}$ and $\ol{\CB}^{2,k}$ are quantum disks conditional on their quantum boundary length.  Under this coupling, we have that $\stoppingleft^1 = \lefthit^2$ and $\righthit^1 = \stoppingleft^2$.

Let $\xi^1 = \delta_k \lfloor \tau^1/\delta_k \rfloor$ so that $[\xi^1,\tau^1]$ is the interval of time in which $\qlegrowthleft^{1,k}|_{[0,\stoppingleft^1]}$ is drawing its final segment of $\SLE_6$.  Let $\eta'$ be this $\SLE_6$ process (normalized so that it starts at time $\xi^1$) and let $\zeta^1$ be the first time after $\stoppingleft^1$ that the outer boundary of the cluster $\partial \qlegrowthleft_{\xi^1}^{1,k} \cup \eta'([\xi^1,\zeta^1])$ is equal to the boundary length of $\partial \qlegrowthleft_{\xi^1}^{1,k}$.  Since the boundary length evolves as the time-reversal of a $3/2$-stable L\'evy excursion (when parameterized by quantum natural time), we note that $\zeta^1$ converges to $\tau^1$ in probability as $k \to \infty$.  (In particular, the probability that $\zeta^1$ is finite tends to $1$ as $k \to \infty$.)  We let $\CR^1$ be the quantum surface which is cut out by $\eta'([\xi^1,\zeta^1])$.  On the event that $\qlegrowthleft_{\stoppingleft^1}^{1,k} \cup \CR^1$ is disjoint from $\qlegrowthright_{\righthit^1}^{1,k}$, the conditional law of the quantum surface $\wh{\CB}^{1,k}$ parameterized by the region outside of $\qlegrowthleft_{\stoppingleft^1}^{1,k} \cup \CR^1 \cup \qlegrowthright_{\righthit^1}^{1,k}$ is that of a quantum disk given its quantum boundary length.  We also let $\ol{\eta}'$ be the $\SLE_6$ process which is drawing the final segment of $\qlegrowthright^{2,k}|_{[0,\stoppingright^{2}]}$ (normalized so that it starts at time $\stoppingright^2$) and let $\ol{\zeta}^2$ be the first time after $\stoppingright^2$ that the quantum length of the surface outside of $\qlegrowthleft_{\lefthit^2}^{2,k} \cup \qlegrowthright_{\stoppingright^2}^{2,k} \cup \ol{\eta}'([\stoppingright^2,\ol{\zeta}^2])$ is equal to that of $\partial \wh{\CB}^{1,k}$ and let $\ol{\CR}^2$ be the surface which is cut out by $\ol{\eta}'([\stoppingright^2,\ol{\zeta}^2])$.  As before, we have that $\ol{\zeta}^2 \to \stoppingright^2$ in probability as $k \to \infty$ and the probability that $\ol{\CR}^2$ is disjoint from $\qlegrowthleft_{\lefthit^2}^{2,k}$ tends to $1$ as $k \to \infty$.  Let $\wh{\CB}^{2,k}$ be the quantum surface parameterized by the region outside of $\qlegrowthleft_{\lefthit^2}^{2,k} \cup \ol{\CR}^2 \cup \qlegrowthright_{\stoppingright^2}^{2,k}$.  On the event that $\ol{\CR}^2 \cap \qlegrowthleft_{\lefthit^2}^{2,k} = \emptyset$, we have that $\wh{\CB}^{2,k}$ given its boundary length is a quantum disk.  In particular, on the events $\CR^1 \cap \qlegrowthright_{\righthit^1}^{1,k} = \emptyset$ and $\ol{\CR}^2 \cap \qlegrowthleft_{\lefthit^2}^{2,k} = \emptyset$, we have that $\wh{\CB}^{1,k}$ and $\wh{\CB}^{2,k}$ have the same boundary length so we can recouple so that $\wh{\CB}^{1,k} = \wh{\CB}^{2,k}$.

Let $\xi^2 = \delta_k \lfloor \lefthit^2/\delta_k \rfloor$ so that $[\xi^2,\lefthit^2]$ is the interval of time in which $\qlegrowthleft^{2,k}|_{[0,\lefthit^2]}$ is drawing its final segment of $\SLE_6$.  We note that $\xi^1 = \xi^2$ as $\stoppingleft^1 = \lefthit^2$, but we will write $\xi^1$ when referring to $\qlegrowthleft^{1,k}$ and $\xi^2$ when referring to $\qlegrowthleft^{2,k}$.

Let $\phi^k$ be the equivalence map (as marked (by their tips) beaded quantum surfaces) which takes $\qlegrowthleft_{\xi^1}^{1,k}$ to $\qlegrowthleft_{\xi^2}^{2,k}$.  Then $\phi^k$ in particular takes the tip $z^{1,k}$ of $\qlegrowthleft^{1,k}$ at time $\xi^1$ to the tip $z^{2,k}$ of $\qlegrowthleft^{2,k}$ at time $\xi^2$.  Let $\varphi^k$ be the equivalence map which takes $\wh{\CB}^{1,k}$ to $\wh{\CB}^{2,k}$.  As $z^{1,k}$ is uniformly distributed according to the quantum measure on $\partial \wh{\CB}^{1,k} \cap \partial \qlegrowthleft_{\xi^1}^{1,k}$ and $z^{2,k}$ is uniformly distributed according to the quantum measure on $\partial \wh{\CB}^{2,k} \cap \qlegrowthleft_{\xi^2}^{2,k}$, it follows that we can take the coupling so that $\varphi^k(z^{1,k}) = z^{2,k}$ with probability tending to $1$ as $k \to \infty$.

\noindent{\it Step 2.  Welding inside and outside of the approximations.} Since $\varphi^k$ and $\phi^k$ both preserve the quantum boundary measure and both take $z^{1,k}$ to $z^{2,k}$, it therefore follows that $\varphi^k$ and $\phi^k$ agree on the component of $\partial \qlegrowthleft_{\stoppingleft^1}^{1,k} \setminus (\CR^1 \cup \partial \qlegrowthright_{\righthit^1}^{1,k})$ which contains $z^{1,k}$.  Note that $\partial \qlegrowthleft_{\stoppingleft^1}^{1,k} \setminus (\CR^1 \cup \partial \qlegrowthright_{\righthit^1}^{1,k})$ consists of two components and we want to argue that $\varphi^k$ and $\phi^k$ also agree on the other connected component.  This, however, follows since the boundary length of $\partial \qlegrowthleft_{\xi^1}^{1,k} \cap \partial \CR^1$ is equal to the boundary length of $\partial \CR^1 \setminus \partial \qlegrowthleft_{\xi^1}^{1,k}$.  That is, the map~$\psi^k$ which takes $\qlegrowthleft_{\xi^1}^{1,k} \cup \wh{\CB}^{1,k}$ to $\qlegrowthleft_{\xi^2}^{2,k} \cup \wh{\CB}^{2,k}$ given by~$\varphi^k$ on $\wh{\CB}^{1,k}$ and~$\phi^k$ on $\qlegrowthleft_{\xi^1}^{1,k}$ is conformal on the interiors of $\wh{\CB}^{1,k}$ and~$\qlegrowthleft_{\xi^1}^{1,k}$ and a homeomorphism on its entire domain.  Therefore by the conformal removability of~$\partial \qlegrowthleft_{\xi^1}^{1,k} \cup \partial \CR^1$, we have that~$\psi^k$ is conformal on the interior of its domain.

\noindent{\it Step 3. Limit of welding yields equivalence of $\QLE$'s.}  Since $\diam(\CR^1) \to 0$, $\diam(\ol{\CR}^2) \to 0$ in probability as $k \to \infty$, it therefore follows that, by passing to a further subsequence if necessary (and recoupling the laws using the Skorokhod representation theorem), $\psi^k$ converges a.s.\  as $k \to \infty$ with respect to the local uniform topology to a map $\psi \colon \qlegrowthleft_{\stoppingleft^1}^1 \cup \CB^1 \to \qlegrowthleft_{\lefthit^2}^2 \cup \ol{\CB}^2$ which is a homeomorphism and conformal in its interior.  By Lemma~\ref{lem::hitting_set}, the interior of the domain of $\psi$ includes all of $\qlegrowthleft_{\stoppingleft^1}^1$ except possibly a single point --- the unique point $w_0$ where it hits $\qlegrowthright|_{[0,\righthit^1]}^1$.  However, it is not difficult to see that $\psi$ is continuous at $w_0$.  Indeed, suppose that $(w_k)$ is a sequence in $\qlegrowthleft_{\stoppingleft^1}^1$ which converges to $w_0$.  Let $(w_{j_k})$ be a subsequence of $(w_k)$ and let $(w_{j_{k'}})$ be a further subsequence so that $(\psi(w_{j_k'}))$ converges (note that $\qlegrowthleft_{\lefthit^2}^2$ is compact).  If $\lim_{k \to \infty} \psi(w_{j_k'}) \neq \psi(w_0)$, then it follows since $\psi^{-1}$ is a homeomorphism away from $\psi(w_0)$ that $w_{j_k'} = \psi^{-1}(\psi(w_{j_k'}))$ converges to a point distinct from $w_0$.  This is a contradiction, and therefore $\psi$ is continuous at $w_0$.  It therefore follows that the distribution of $\qlegrowthleft|_{[0,\stoppingleft]}$ under $\lawonall^{x \to y}$ is the same as that of $\qlegrowthleft|_{[0,\lefthit]}$ under $\lawonall^{y \to x}$ (in the sense described just before the statement of the lemma).  The same argument also implies that the distribution of $\qlegrowthright|_{[0,\righthit]}$ under $\lawonall^{x \to y}$ is the same as the distribution of $\qlegrowthright|_{[0,\stoppingright]}$ under $\lawonall^{y \to x}$.  The result follows because $\qlegrowthleft|_{[0,\stoppingleft]},\qlegrowthright|_{[0,\righthit]},\CB$ are conditionally independent under $\lawonall^{x \to y}$ given their boundary lengths, the same holds under $\lawonall^{y \to x}$, and the joint law of the boundary lengths is the same as that of $\qlegrowthleft|_{[0,\lefthit]}$, $\qlegrowthright|_{[0,\stoppingright]}$, $\ol{\CB}$ under $\lawonall^{y \to x}$.
\end{proof}

\subsubsection{Proof of Theorem~\ref{thm::qle_symmetry}}

Lemma~\ref{lem::continuum_equivalence} implies that the $\lawonall^{x \to y}$ distribution of $(\qlegrowthleft|_{[0,\stoppingleft]},\qlegrowthright|_{[0,\righthit]},\CB)$ is equal to the $\lawonall^{y \to x}$ distribution of $(\qlegrowthleft|_{[0,\lefthit]},\qlegrowthright|_{[0,\stoppingright]},\ol{\CB})$ (where the sense of equivalence for the first two components is as described just before the statement of Lemma~\ref{lem::continuum_equivalence} and the third component is in the usual sense of a quantum surface).  This implies that we can construct a coupling of copies $(\qlegrowthleft^1|_{[0,\stoppingleft^1]},\qlegrowthright^1|_{[0,\righthit^1]},\CB^1)$ and $(\qlegrowthleft^2|_{[0,\lefthit^2]},\qlegrowthright^2|_{[0,\stoppingright^2]},\ol{\CB}^2)$ so that we have $\qlegrowthleft_{\stoppingleft^1}^1 = \qlegrowthleft_{\lefthit^2}^2$, $\qlegrowthright_{\righthit^1}^1 = \qlegrowthright_{\stoppingright^2}^2$, and $\CB^1 = \ol{\CB}^2$.  This implies the existence of homeomorphisms $\varphi \colon \qlegrowthleft_{\stoppingleft^1}^1 \to \qlegrowthleft_{\lefthit^2}^2$, $\ol{\varphi} \colon \qlegrowthright_{\righthit^1}^1 \to \qlegrowthright_{\stoppingright^2}^2$, and $\psi \colon \CB^1 \to \ol{\CB}^2$ which are each conformal in the interior of their domain and preserve the quantum boundary length.  By Lemma~\ref{lem::pinch_point_glued}, we know that the point on $\partial \qlegrowthleft_{\stoppingleft^1}^1$ which is glued to the a.s.\  unique pinch point on $\partial \CB^1$ is uniformly random from the quantum boundary measure on $\partial \qlegrowthleft_{\stoppingleft^1}^1$ and the same is likewise true for $\qlegrowthleft_{\lefthit^2}^2$.  Thus we can couple so that $\varphi$ (resp.\ $\ol{\varphi}$) sends this point to the corresponding point on $\partial \qlegrowthleft_{\lefthit^2}^2$ (resp.\ $\partial \qlegrowthright_{\stoppingright^2}^2$).  Let $\CS^j$ for $j=1,2$ be the doubly-marked quantum sphere associated with $(\qlegrowthleft_{\stoppingleft^1}^1,\qlegrowthright_{\righthit^1}^1,\CB^1)$ and $(\qlegrowthleft_{\lefthit^2}^2,\qlegrowthright_{\stoppingright^2}^2,\ol{\CB}^2)$, respectively.  Letting $\phi$ be the map $\CS^1 \to \CS^2$ which on $\qlegrowthleft_{\stoppingleft^1}^1$ (resp.\ $\qlegrowthright_{\righthit^1}^1$) is given by $\varphi$ (resp.\ $\ol{\varphi}$) and on $\CB^1$ is given by $\psi$, we thus see that $\phi$ is a homeomorphism which is conformal in the complement of $\partial \CB^1 = \partial \qlegrowthleft_{\stoppingleft^1}^1 \cup \partial \qlegrowthright_{\righthit^1}^1$.  By the conformal removability of $\partial \CB^1$ (Lemma~\ref{lem::hitting_set}), we thus have that $\phi$ is conformal everywhere hence must be a M\"obius transformation, which completes the proof. \qed

\subsubsection{Conditional law of the unexplored surface}

We finish this section by stating the analog of the final part of Proposition~\ref{prop::measure_properties} for $\QLE(8/3,0)$.

\begin{proposition}
\label{prop::m_d_l_conditional_law}
The $\lawonall^{x \to y}$ conditional distribution given $(\CS,x,y)$ and $\qlegrowth|_{[0,\stoppingleft]}$ is equal to the $\lawonall$ conditional distribution given $(\CS,x,y)$ and $\qlegrowth|_{[0,\stoppingleft]}$.
\end{proposition}
\begin{proof}
The last part of Proposition~\ref{prop::measure_properties} implies that the analogous statement holds for each of the $\delta$-approximations to $\QLE(8/3,0)$.  Therefore it follows that this property holds in the limit.
\end{proof}

\section{Metric construction}
\label{sec::metric}

The purpose of this section is to show that $\QLE(8/3,0)$ equipped with the quantum distance parameterization defines a metric on a countable, dense subset consisting of a sequence of i.i.d.\ points chosen from the quantum measure of a $\sqrt{8/3}$-LQG sphere. The approach we take, as explained and outlined in Section~\ref{subsec::overview}, is based on ideas from a joint work by the second co-author, Sam Watson, and Hao Wu.

We begin by considering a pair of $\QLE(8/3,0)$ explorations $\qlegrowthleft$ and $\qlegrowthright$ on a doubly marked quantum sphere $(\CS,x,y)$ where the first (resp.\ second) process starts from $x$ (resp.\ $y$) and is targeted at~$y$ (resp.\ $x$).  We then use several results accumulated earlier in the paper to prove that the ``distances'' computed by these two explorations are a.s.\  the same.

The remainder of this section is structured as follows.  In Section~\ref{subsec::metric_setup}, we describe the setup and notation that we will use throughout the rest of the section.  We will complete the proof of Theorem~\ref{thm::qle_metric} in Section~\ref{subsec::coupling_explorations}.

\subsection{Setup and notation}
\label{subsec::metric_setup}

We will use the same setup and notation described in the beginning of Section~\ref{sec::qle_symmetry}.  Since $\lawonall^{x \to y}$ and $\lawonall^{y \to x}$ are infinite measures, we cannot normalize them to be probability measures.  However, these measures conditioned on certain quantities do yield probability measures.  In these cases, we will let $\p_{\lawonall^{x \to y}}[ \cdot \giv \cdot]$ (resp.\ $\p_{\lawonall^{y \to x}}[ \cdot \giv \cdot]$) and $\E_{\lawonall^{x \to y}}[ \cdot \giv \cdot]$ (resp.\ $\E_{\lawonall^{y \to x}}[ \cdot \giv \cdot]$) denote the corresponding probability and expectation.  We will similarly write $\p_\lawonall[ \cdot \giv \cdot]$ and $\E_\lawonall[ \cdot \giv \cdot]$ for the probability and expectation which correspond to $\lawonall$ conditioned on a quantity which leads to a probability measure.  For each $t \geq 0$, we also let $\CF_t = \sigma( \CS,x,y, \qlegrowthleft_s : s \leq t)$ and $\ol{\CF}_t = \sigma(\CS,x,y,\qlegrowthright_s : s \leq t)$.  That is, $\CF_t$ (resp.\ $\ol{\CF}_t$) is the filtration generated by $(\CS,x,y)$ and $\qlegrowthleft$ (resp.\ $\qlegrowthright$) with the quantum distance parameterization.

\subsection{Coupling explorations and the metric property}
\label{subsec::coupling_explorations}

We will complete the proof of Theorem~\ref{thm::qle_metric} in this section.  The first step is to prove the following proposition, which is in some sense the heart of the matter.

\begin{proposition}
\label{prop::symmetry}
The $\lawonall^{x \to y}$ distribution of $(\CS,x,y,\stoppingleft,\qlegrowthleft,\righthit,\qlegrowthright)$ is the same as the $\lawonall^{y \to x}$ distribution of $(\CS,x,y,\lefthit,\qlegrowthleft,\stoppingright,\qlegrowthright)$. 
\end{proposition}

Theorem~\ref{thm::qle_symmetry} implies that the $\lawonall^{x \to y}$ distribution of $(\CS,x,y, \qlegrowthleft|_{[0,\stoppingleft]}, \qlegrowthright|_{[0,\righthit]})$ is equal to the $\lawonall^{y \to x}$ distribution of $(\CS,x,y, \qlegrowthleft|_{[0,\lefthit]}, \qlegrowthright|_{[0,\stoppingright]})$.  To prove Proposition~\ref{prop::symmetry}, we will try to understand the conditional law of $\qlegrowthleft$ and $\qlegrowthright$ given this information.

We provide here a sketch of the argument before we proceed to the details.  Because of the way that $\lawonall^{x \to y}$ was constructed, one can show quite easily that the $\lawonall^{x \to y}$ conditional law of $\qlegrowthleft$ {\em given} the five-tuple $(\CS,x,y,\qlegrowthleft|_{[0,\stoppingleft]}, \qlegrowthright |_{[0,\righthit]})$ is a function of the four-tuple $(\CS,x,y,\qlegrowthleft|_{[0,\stoppingleft]})$. (The existence of regular conditional probabilities follows from the fact that the GFF and associated growth processes can be defined as random variables in a standard Borel space; see further discussion for example in \cite{SchrammShe10, ms2013qle}.)
The proof of Proposition~\ref{prop::symmetry} will be essentially done once we establish the following claim: the $\lawonall^{y \to x}$ conditional law of $\qlegrowthleft$ {\em given} the five-tuple $(\CS,x,y,\qlegrowthleft|_{[0,\lefthit]},\qlegrowthright|_{[0,\stoppingright]})$ is described by the same function applied to the four-tuple $(\CS,x,y,\qlegrowthleft|_{[0, \lefthit]})$.  Indeed, a symmetric argument implies an analogous statement about the conditional law of $\qlegrowthright$ under $\lawonall^{x \to y}$ and $\lawonall^{y \to x}$ given the corresponding five-tuple. Proposition~\ref{prop::symmetry} will follow readily from this symmetric pair of statements and the {\it a priori} conditional independence of $\qlegrowthleft$ and $\qlegrowthright$ given $(\CS,x,y)$.

The claim stated just above may appear to be obvious, but there is still some subtlety arising from the fact that $\stoppingleft$ and $\lefthit$ are not defined in symmetric ways {\em a priori}, and in fact $\lefthit$ is a complicated stopping time for $\qlegrowthleft$ (which depends both on $(\CS,x,y)$ and on the additional randomness encoded in $\qlegrowthright$ and $\stoppingright$), and we have not proved anything like a strong Markov property for the $\QLE(8/3,0)$ growth that would hold for arbitrary stopping times. We begin by fixing a bounded function $F$ and let
\begin{equation}
\label{eqn::m_def}
M_t = \E_{\lawonall}[ F(\qlegrowthleft) \giv \CF_t] \quad\text{and}\quad \ol{M}_t = \E_{\lawonall}[ F(\qlegrowthright) \giv \ol{\CF}_t]  \quad\text{for each}\quad t \geq 0.
\end{equation}

We note that we can write $M_t = A(\CS,x,y,\qlegrowthleft|_{[0,t]})$ for some measurable function $A$.  Moreover, by varying $F$, this family of measurable functions determines the conditional law of $\qlegrowthleft$ given $(\CS,x,y,\qlegrowthleft|_{[0,t]})$.  We will check that $M$ is a.s.\ continuous (as a function of~$t$) at the time $\stoppingleft$. Since $M$ is a continuous-time martingale, it a.s.\ has only countably many discontinuities.  That is, the limits $\lim_{s \uparrow t} M_s$ and $\lim_{s \downarrow t} M_s$ exist for a.e.\ $t$ and necessarily coincide except at possibly a countable set of times.  It will thus suffice to prove that when $(\CS,x,y,\qlegrowthleft,\qlegrowthright)$ is given and $\stoppingleft$ is chosen uniformly from $[0,\qdist(x,y)]$ the probability that either $\stoppingleft$ assumes any {\em fixed} value is zero.  This, however, is obviously true from the definition of $\stoppingleft$.

 We likewise have that $\ol{M}_t = \ol{A}(\CS,x,y,\qlegrowthright|_{[0,t]})$ for some measurable function $\ol{A}$ and this family determines the conditional law of $\qlegrowthright$ given $(\CS,x,y,\qlegrowthright|_{[0,t]})$.  Moreover, $\ol{M}$ has the same continuity properties as $M$ since it is also a continuous-time martingale.  We will aim to show that $\ol{M}$ is a.s.\ continuous at $\lefthit$.  As in the case of $M$ and $\stoppingleft$, in the case of $\ol{M}$ and $\lefthit$ it suffices to show that the probability that $\lefthit$ assumes any {\em fixed} value is zero when $(\CS,x,y,\qlegrowthleft,\qlegrowthright)$ is given and $\stoppingright$ is chosen uniformly in $[0,\oqdist(y,x)]$.  This follows because the map $\stoppingright \to \lefthit$  (a random function that depends on $(\CS,x,y,\qlegrowth,\qlegrowthright)$) is non-increasing by definition, and the symmetry property implies that~$\qlegrowth$ a.s.\ first hits $\qlegrowthright_\stoppingright$ at exactly time~$\lefthit$, which implies that there a.s.\ cannot be a positive interval of~$\stoppingright$ values on which~$\lefthit$ is constant.

The following lemma is the main input into the proof of Proposition~\ref{prop::symmetry}.

\begin{lemma}
\label{lem::first_equivalence}
We have on a set of full $\lawonall$ measure that
\begin{equation}
\label{eqn::r_tau_r_equivalence}
\begin{split}
A(\CS,x,y,\qlegrowthleft|_{[0,\lefthit]}) &= \E_{\lawonall^{y \to x}}[ F(\qlegrowthleft) \giv \CF_\lefthit, \ol{\CF}_\stoppingright] \quad\text{and}\\
 \ol{A}(\CS,x,y,\qlegrowthright|_{[0,\righthit]}) &= \E_{\lawonall^{x \to y}}[ F(\qlegrowthright) \giv \CF_{\stoppingleft}, \ol{\CF}_{\righthit} ].
\end{split}
\end{equation}
\end{lemma}
To further clarify the statement of Lemma~\ref{lem::first_equivalence}, we note that $\sigma$ is a stopping time for the filtration generated by $\CF_t, \ol{\CF}_{\ol{\sigma}}$.  Thus the conditional expectation in the first equality of~\eqref{eqn::r_tau_r_equivalence} is with respect to the stopped $\sigma$-algebra for this filtration at the time $\sigma$.  Similarly, $\ol{\tau}$ is a stopping time for the filtration generated by $\CF_\tau, \ol{\CF}_t$ and the conditional expectation in the second equality of~\eqref{eqn::r_tau_r_equivalence} is with respect to the stopped $\sigma$-algebra for this filtration at the time $\ol{\tau}$.

Lemma~\ref{lem::first_equivalence} implies that $\E_{\lawonall^{y \to x}}[ F(\qlegrowthleft) \giv \CF_\lefthit, \ol{\CF}_\stoppingright] = A(\CS,x,y,\qlegrowthleft|_{[0,\sigma]})$.  Since $\E_{\lawonall^{x \to y}}[F(\qlegrowthleft) \giv \CF_\tau] = \E_\lawonall[F(\qlegrowthleft) \giv \CF_\tau] = A(\CS,x,y,\qlegrowthleft|_{[0,\tau]})$, we thus see that Lemma~\ref{lem::first_equivalence} implies that the $\lawonall^{y \to x}$ conditional law of $\qlegrowthleft$ given $(\CS,x,y,\qlegrowthleft|_{[0,\lefthit]},\qlegrowthright|_{[0,\stoppingright]})$ is the same as the $\lawonall^{x \to y}$ conditional law of $\qlegrowthleft$ given $(\CS,x,y,\qlegrowthleft|_{[0,\stoppingleft]})$.

\begin{proof}[Proof of Lemma~\ref{lem::first_equivalence}]
We will only establish the second equality in~\eqref{eqn::r_tau_r_equivalence}; the proof of the first equality is analogous.  We first recall that on a set of full $\lawonall$ measure we have $\E_{\lawonall^{x \to y}}[ \cdot \giv \CF_{\stoppingleft}] =\E_\lawonall[ \cdot \giv \CF_{\stoppingleft}]$; see Proposition~\ref{prop::m_d_l_conditional_law}.  Throughout, we will write $dt$ for Lebesgue measure on~$\R_+$.  We now observe that $d\lawonall \otimes dt$ a.e.\ we have that
\begin{align}
\label{eqn::rt1}
\E_\lawonall[ F(\qlegrowthright) \giv \CF_{\stoppingleft}, \ol{\CF}_t] = \E_{\lawonall^{x \to y}}[ F(\qlegrowthright) \giv \CF_{\stoppingleft}, \ol{\CF}_t].
\end{align}
Since $\CF_{\stoppingleft}$ and $\ol{\CF}_t$ are conditionally independent given $(\CS,x,y)$ under $\lawonall$, we have $d\lawonall \otimes dt$ a.e.\ that
\begin{align}
\label{eqn::rt2}
\ol{M}_t = \E_\lawonall[ F(\qlegrowthright) \giv \CF_{\stoppingleft}, \ol{\CF}_t].
\end{align}
Combining~\eqref{eqn::rt1} and~\eqref{eqn::rt2}, we have $d \lawonall \otimes dt$ a.e.\ that
\begin{equation}
\label{eqn::rt3}
\ol{M}_t = \E_{\lawonall^{x \to y}}[ F(\qlegrowthright) \giv \CF_{\stoppingleft}, \ol{\CF}_t].
\end{equation}
In particular, the event $E$ that~\eqref{eqn::rt3} holds for all rational times simultaneously has full $\lawonall$ measure.

We will deduce~\eqref{eqn::r_tau_r_equivalence} from~\eqref{eqn::rt3} by showing that on a set of full $\lawonall$ measure we have both
\begin{align}
   \ol{M}_t &\to \ol{A}(\CS,x,y,\qlegrowthright|_{[0,\righthit]}) \quad\text{as}\quad t \uparrow \righthit,\ \ t \in \Q_+ \quad\text{and} \label{eqn::r_t_limit}\\
\E_{\lawonall^{x \to y}}[ F(\qlegrowthright) \giv \CF_{\stoppingleft}, \ol{\CF}_t] &\to \E_{\lawonall^{x \to y}}[ F(\qlegrowthright) \giv \CF_{\stoppingleft}, \ol{\CF}_{\righthit}] \quad\text{as}\quad t \uparrow \righthit,\ t \in \Q_+. \label{eqn::e_l_d_r_limit}
\end{align}
We emphasize that in~\eqref{eqn::r_t_limit} and~\eqref{eqn::e_l_d_r_limit} we take the limit along positive rational $t$.

To prove~\eqref{eqn::e_l_d_r_limit} it suffices to show that the $\sigma$-algebra generated by $\CF_\tau, \ol{\CF}_t$ for $t < \righthit$ is equal to the $\sigma$-algebra generated by $\CF_\tau, \ol{\CF}_{\righthit}$.  It in turn suffices to show that the closure of $\cup_{t < \righthit} \qlegrowthright_t$ is a.s.\  equal to~$\qlegrowthright_{\righthit}$.  This will follow by showing that $\righthit$ does not correspond to a jump in the time change from capacity to quantum distance time.  Since there can only be a countable number of such jump times and $\stoppingright$ is uniform in $[0,\oqdist(y,x)]$, it follows that $\stoppingright$ a.s.\  does not correspond to such a jump time for $\qlegrowthright$.  By Theorem~\ref{thm::qle_symmetry}, the $\lawonall^{x \to y}$ distribution of $\qlegrowthright|_{[0,\righthit]}$ is the same as the $\lawonall^{y \to x}$ distribution of $\qlegrowthright|_{[0,\stoppingright]}$.  We therefore conclude that $\righthit$ similarly a.s.\ does not correspond to a time at which the capacity of $\qlegrowthright$ jumps.

The existence of the limit in~\eqref{eqn::e_l_d_r_limit} combined with~\eqref{eqn::rt3} implies that $\ol{M}_t$ a.s.\  has a limit as $t \uparrow \righthit$ along rationals.  We need to show that this limit is a.s.\  equal to~$\ol{A}(\CS,x,y,\qlegrowthright|_{[0,\righthit]})$.  As we remarked above, we know that $\ol{M}_t$ can jump at most countably many times as it is a continuous-time martingale.  Since $\p_\lawonall[ \stoppingright = t \giv \ol{\CF}_{\oqdist(y,x)}] = 0$ a.s.\ for any fixed $t$, it follows that $\stoppingright$ is a.s.\ not a jump time for $\ol{M}_t = \ol{A}(\CS,x,y,\qlegrowthright|_{[0,t]})$.  Consequently, we have that
\begin{equation}
\label{eqn::l_t_to_tau_l}
\lim_{t \to \stoppingright^-} \ol{A}(\CS,x,y,\qlegrowthright|_{[0,t]}) = \lim_{t \to \stoppingright^-} \ol{M}_t = \ol{M}_{\stoppingright} = \ol{A}(\CS,x,y,\qlegrowthright|_{[0,\stoppingright]})
\end{equation}
on a set of full $\lawonall$ measure.  By Theorem~\ref{thm::qle_symmetry}, we know that the $\lawonall^{y \to x}$ distribution of $\ol{A}(\CS,x,y,\qlegrowthright|_{[0,t]})$ for $t \in [0,\stoppingright]$ is equal to the $\lawonall^{x \to y}$ distribution of $\ol{A}(\CS,x,y,\qlegrowthright|_{[0,t]})$ for $t \in [0,\righthit]$.  Therefore~\eqref{eqn::l_t_to_tau_l} implies that $\lim_{t \to \righthit^-} \ol{A}(\CS,x,y,\qlegrowthright|_{[0,t]}) = \ol{A}(\CS,x,y,\qlegrowthright|_{[0,\righthit]})$ on a set of full $\lawonall^{x \to y}$ measure.  Since $\lawonall$ is absolutely continuous with respect to $\lawonall^{x \to y}$, it therefore follows that $\lim_{t \to \righthit^-} \ol{A}(\CS,x,y,\qlegrowthright|_{[0,t]}) = \ol{A}(\CS,x,y,\qlegrowthright|_{[0,\righthit]})$ on a set of full $\lawonall$ measure.  Combining everything completes the proof.
\end{proof}

\begin{proof}[Proof of Proposition~\ref{prop::symmetry}]
Theorem~\ref{thm::qle_symmetry} implies that the $\lawonall^{x \to y}$ distribution of $(\CS,x,y)$, $\qlegrowthleft|_{[0,\stoppingleft]}$, and $\qlegrowthright|_{[0,\righthit]}$ is equal to the $\lawonall^{y \to x}$ distribution of $(\CS,x,y)$, $\qlegrowthleft|_{[0,\lefthit]}$, and $\qlegrowthright|_{[0,\stoppingright]}$.

We claim that the $\lawonall^{x \to y}$ conditional law of $\qlegrowthleft$ given $\CF_\stoppingleft,\ol{\CF}_{\righthit}$ is the same as the $\lawonall^{y \to x}$ conditional law of $\qlegrowthleft$ given $\CF_\lefthit, \ol{\CF}_{\stoppingright}$.  The same argument will give that the $\lawonall^{y \to x}$ conditional law of $\qlegrowthright$ given $\CF_\lefthit, \ol{\CF}_{\stoppingright}$ is the same as the $\lawonall^{x \to y}$ conditional law of $\qlegrowthright$ given $\CF_\stoppingleft,\ol{\CF}_{\righthit}$.  Upon showing this, the proof will be complete.  To see the claim, we first note that Proposition~\ref{prop::m_d_l_conditional_law} implies that the $\lawonall^{x \to y}$ conditional law of $\qlegrowthleft$ given $\CF_{\stoppingleft},\ol{\CF}_{\righthit}$ is the same as its $\lawonall$ conditional law given $\CF_{\stoppingleft},\ol{\CF}_{\righthit}$.  By the $\lawonall$ conditional independence of $\qlegrowthleft$ and $\qlegrowthright$ given $(\CS,x,y)$, this conditional law is in turn equal to the $\lawonall$ conditional law of $\qlegrowthleft$ given $\CF_\stoppingleft$.  Lemma~\ref{lem::first_equivalence} then implies that the $\lawonall$ conditional law of $\qlegrowthleft$ given $\CF_\stoppingleft$ is equal to the $\lawonall^{y \to x}$ conditional law of $\qlegrowthleft$ given $\CF_\lefthit, \ol{\CF}_\stoppingright$ (this follows from the first equation in~\eqref{eqn::r_tau_r_equivalence}).  This proves the claim, hence the proposition.
\end{proof}

\begin{proposition}
\label{prop::distance_symmetric_and_determined}
We have on a set of full $\lawonall$ measure that $\qdist(x,y) = \oqdist(y,x)$ and that the common value of $\qdist(x,y)$ and $\oqdist(y,x)$ is a.s.\  determined by $(\CS,x,y)$.
\end{proposition}
\begin{proof}
Proposition~\ref{prop::symmetry} implies that
\[ \qdist(x,y)  = \frac{d\lawonall^{x \to y}}{d \lawonall} = \frac{d\lawonall^{y \to x}}{d \lawonall} = \oqdist(y,x) .\]
This implies that $\qdist(x,y) = \oqdist(y,x)$ on a set of full $\lawonall$ measure.  Moreover, the common value of $\qdist(x,y)$ and $\oqdist(y,x)$ is a.s.\  determined by $(\CS,x,y)$ because we took $\qlegrowthleft$ and $\qlegrowthright$ to be conditionally independent given $(\CS,x,y)$ under $\lawonall$.  In particular, this implies that $\qdist(x,y)$ and $\oqdist(y,x)$ are conditionally independent given $(\CS,x,y)$ and the only way that two conditionally independent random variables given $(\CS,x,y)$ can be equal on a set of full measure is if they are a.s.\  determined by $(\CS,x,y)$.
\end{proof}

Proposition~\ref{prop::symmetry} implies that if $\CS$ has the law of a unit area quantum sphere and $x,y \in \CS$ are chosen independently from the quantum measure on $\CS$, then the amount of time that it takes a $\QLE(8/3,0)$ with the quantum distance parameterization starting from~$x$ to reach~$y$ is a.s.\  determined by $\CS$ and is equal to the amount of time that it takes for a $\QLE(8/3,0)$ with the quantum distance parameterization starting from~$y$ to reach~$x$.  Moreover, this quantity does not depend on the choice of sequence that we chose in the construction of the $\QLE(8/3,0)$.  Therefore from now on we will only write $\qdist$ (and not $\oqdist$).  Conditionally on $\CS$, we let $(x_n)$ be a sequence of i.i.d.\ points picked from the quantum area measure.  For each $i,j$, we let $\qdist(x_i,x_j)$ be the amount of quantum distance time that it takes for a $\QLE(8/3,0)$ starting from~$x_i$ to reach~$x_j$.  Then it follows that $\qdist(x_i,x_j) = \qdist(x_j,x_i)$ for all $i,j$ and $\qdist$ is a.s.\  determined by $\CS$.  Moreover, it is immediate from the construction that $\qdist(x_i,x_j) > 0$ a.s.\  for any $i \neq j$.

Our next goal is to establish Proposition~\ref{prop::meeting_distance_sums_up}, which will be used in the proof of Theorem~\ref{thm::qle_metric} to show that $\qdist$ satisfies the triangle inequality hence is a distance function on $(x_n)$.  Before we state and prove this result, we need to collect the following elementary lemma.

\begin{lemma}
\label{lem::uniform_distribution_determines_f}
Fix $D > 0$.  Suppose that $F \colon [0,D] \to [0,D]$ is a non-increasing function such that if $U$ is uniform in $[0,D]$ then $F(U)$ is uniform in $[0,D]$.  Then $F(d) = D-d$ for all $d \in [0,D]$.
\end{lemma}
\begin{proof}
This is essentially obvious, but let us explain the point to be clear.  Since $F(U)$ is uniform in $[0,D]$, we have $\p[F(U) \geq d] = 1-d/D$ for all $d \in [0,D]$.  Since $F$ is non-increasing and $F(U)$ is uniform it follows that there cannot be a non-empty open interval in $[0,D]$ on which $F$ is constant.  Consequently, $\p[F(U) \leq F(d)] = \p[U \geq d] = 1-d / D$. Since $V=F(U)$ is uniform and we have shown that $\p[V \leq F(d)] = 1 - d/D$ for all $d \in [0,D]$, we conclude that $F(d) = D - d$ for all $d \in [0,D]$.
\end{proof}

\begin{proposition}
\label{prop::meeting_distance_sums_up}
On a set of full $\lawonall$ measure, we have that $\stoppingleft + \righthit = D$ where $D$ is the common value of $\qdist(x,y)$ and $\qdist(y,x)$.
\end{proposition}
\begin{proof}
It suffices to show that $\stoppingleft + \righthit = D$ under $\lawonall^{x \to y}$ since this measure is mutually absolutely continuous with respect to $\lawonall$.

For each $d \in [0,D]$, let $F(d)$ be the first time that $\qlegrowthright$ hits $\qlegrowth_d$.  We emphasize that $F$ is determined by $\qlegrowthleft$, $\qlegrowthright$, and $\CS$.  We already know that $\qdist(x,y) = \qdist(y,x)$ a.s.\ and that this quantity is a.s.\ determined by $(\CS,x,y)$. We know that $U = \stoppingleft/\qdist(x,y)$ is uniform in $[0,1]$ conditionally on $(\CS,x,y)$ and $\qlegrowth$. The symmetry we have established in Proposition~\ref{prop::symmetry} then implies that $\ol{U} = \righthit / \qdist(y,x)$ is uniform in $[0,1]$ conditionally on $(\CS,x,y)$ and $\qlegrowthright$. 
In the case of $\stoppingleft$, it is clearly also uniform in $[0,D]$ once we condition on $(\CS,x,y,\qlegrowth, \qlegrowthright)$, which determines $F$.  That this is true for $\righthit$ as well follows from Proposition~\ref{prop::symmetry}, so that both~$\stoppingleft$ and~$\righthit$ are uniform in $[0,D]$ conditionally on~$F$.  To finish the proof of the proposition, it suffices to show that $F(d) = D-d$.  We know that $F \colon [0,D] \to [0,D]$ is non-decreasing and that $F(\stoppingleft) = \righthit$.  Therefore the result follows from Lemma~\ref{lem::uniform_distribution_determines_f}.
\end{proof}

\begin{proof}[Proof of Theorem~\ref{thm::qle_metric}]
Let $\qdist$ and $(x_n)$ be as defined just before the statement of Lemma~\ref{lem::uniform_distribution_determines_f}.  To finish the proof, it suffices to show that $\qdist$ a.s.\ satisfies the strict triangle inequality.  Suppose that $x,y,z \in (x_n)$ are distinct.  We will argue that a.s.\ we have
\begin{equation}
\label{eqn::triangle_inequality}
\qdist(x,z) < \qdist(x,y) + \qdist(y,z).
\end{equation}
We shall assume that we are working on the event that $\qdist(x,y) < \qdist(x,z)$, for otherwise~\eqref{eqn::triangle_inequality} is trivial.  Consider the $\QLE(8/3,0)$ growth $\Upsilon$ starting from $x$ and stopped upon hitting $y$.  Given this, we then consider the $\QLE(8/3,0)$ growth $\ol{\Upsilon}$ starting from $z$ and stopped upon hitting $\Upsilon$.  Then we must have that the radius of this growth is at most $\qdist(y,z)$ because $y \in \Upsilon$.  As it is easy to see that the a.s.\ unique point in $\ol{\Upsilon} \cap \Upsilon$ is uniformly distributed in $\partial \Upsilon$ according to the quantum measure, we a.s.\ have that $y \notin \ol{\Upsilon}$.  That is, the radius of $\ol{\Upsilon}$ a.s.\ is strictly less than $\qdist(y,z)$.  On the other hand, Proposition~\ref{prop::meeting_distance_sums_up} implies that the sum of $\qdist(x,y)$ and the radius of this growth is a.s.\ equal to $\qdist(x,z)$.  Combining proves~\eqref{eqn::triangle_inequality}, which completes the proof.
\end{proof}

\bibliographystyle{hmralphaabbrv}
\addcontentsline{toc}{section}{References}
\bibliography{sle_kappa_rho}

\bigskip

\filbreak
\begingroup
\small
\parindent=0pt

\bigskip
\vtop{
\hsize=5.3in
Department of Mathematics\\
Massachusetts Institute of Technology\\
Cambridge, MA, USA } \endgroup \filbreak

\end{document}